\documentclass[10pt]{amsart}
\usepackage{amscd, amssymb}
\usepackage{pb-diagram}
\usepackage[all]{xypic}
\theoremstyle{plain}
\newtheorem{Thm}[equation]{Theorem}
\newtheorem{Cor}[equation]{Corollary}
\newtheorem{Prop}[equation]{Proposition}
\newtheorem{Lem}[equation]{Lemma}
\newtheorem{Rmk}[equation]{Remark}
\numberwithin{equation}{section}

\newcommand{\Sym}{\operatorname{Sym}}
\newcommand{\St}{\operatorname{St}}
\newcommand{\st}{\operatorname{st}}
\newcommand{\Isom}{\operatorname{Isom}}
\newcommand{\OO}{\operatorname{O}}
\newcommand{\GO}{\operatorname{GO}}
\newcommand{\GSO}{\operatorname{GSO}}
\newcommand{\SO}{\operatorname{SO}}
\newcommand{\PGSp}{\operatorname{PGSp}}
\newcommand{\GSp}{\operatorname{GSp}}

\newcommand{\Sp}{\operatorname{Sp}}
\newcommand{\Ind}{\operatorname{Ind}}
\newcommand{\ind}{\operatorname{ind}}
\newcommand{\Hom}{\operatorname{Hom}}

\newcommand{\GL}{\operatorname{GL}}

\newcommand{\PGL}{\operatorname{PGL}}
\newcommand{\SL}{\operatorname{SL}}

\newcommand{\Span}{\operatorname{span}}
\newcommand{\C}{\mathbb{C}}

\newcommand{\N}{\mathbb{N}}
\newcommand{\W}{\mathbb{W}}

\newcommand{\Y}{\mathbb{Y}}
\renewcommand{\H}{\mathbb{H}}

\newcommand{\bm}{\begin{multline*}}
\newcommand{\tu}{\end  {multline*}}

\newcommand{\simi}{\operatorname{sim}}

\newcommand{\tT}{\tilde{T}}

\title{Theta Correspondences for $\GSp(4)$}
\author{Wee Teck Gan and Shuichiro Takeda}
\address{Mathematics Department, University of California, San Diego, 9500 Gilman Drive, La Jolla,
92093}
\address{Department of Mathematics, Purdue University, 150 N. University Street, West Lafayette, IN 47907-2067}
\email{wgan@math.ucsd.edu}
\email{stakeda@math.purdue.edu}

\begin{document}
\maketitle

\begin{abstract}
We explicitly determine the theta correspondences for $\GSp_4$ and orthogonal similitude groups associated to various quadratic spaces of rank $4$ and $6$. The results are needed in our proof of the local Langlands correspondence for $\GSp_4$ (\cite{GT1}).
  \end{abstract}

\section{\bf Introduction}  \label{S:intro}

This is a companion paper to \cite{GT1}, in which we proved the local Langlands conjecture for $\GSp_4$ over a non-archimedean local  field $F$ of characteristic zero and residue characteristic $p$.  The results of \cite{GT1} depended on a study of the theta correspondences between $\GSp_4$ and the three orthogonal similitude groups in the following diagram:
\vskip 5pt
\[
\xymatrix@R=2pt{
&&\GSO_{3,3}\\
&\GSp_4 \ar@{-}[ru]\ar@{-}[ld]\ar@{-}[rd]&\\
\GSO_{4,0}&&\GSO_{2,2}.}
\]
\vskip 5pt
For example, one of the results needed in \cite{GT1} is that every irreducible representation of $\GSp_4(F)$ participates in theta correspondence with exactly one of $\GSO_{4,0}$ or $\GSO_{3,3}$. This dichotomy follows from a fundamental result of Kudla-Rallis \cite{KR} (which we recall below) for a large class of representations of $\GSp_4(F)$, but to take care of the remaining representations, the results of [KR] needs to be supplemented by the results of this paper.
Moreover, the complete determination of the above theta correspondences serves to make the Langlands parametrization for non-supercuspidal representations
completely explicit and transparent.
\vskip 10pt

To state a sample result in the introduction, let us note that the orthogonal similitude groups mentioned above are closely related to the groups $\GL_2$ and $\GL_4$. Indeed, one has:
\[  \begin{cases}
\GSO_{2,2}  \cong (\GL_2 \times \GL_2)/\{ (z,z^{-1}): z \in \GL_1 \} \\
\GSO_{4,0} \cong (D^{\times} \times D^{\times})/\{ (z,z^{-1}): z \in \GL_1 \} \\
\GSO_{3,3} \cong (\GL_4 \times \GL_1)/ \{ (z,z^{-2}): z \in \GL_1 \} \end{cases} \]
where $D$ is the quaternion division algebra over $F$.
Via these isomorphisms, an irreducible representation of $\GSO_{2,2}(F)$ (resp. $\GSO_{4,0}(F)$) has the form $\tau_1 \boxtimes \tau_2$ where $\tau_1$ and $\tau_2$  are representations of $\GL_2(F)$ (resp. $D^{\times}$) with the same central characters. Similarly, an irreducible representation of $\GSO_{3,3}(F)$ has the form $\Pi \boxtimes \mu$ with $\Pi$ is a representation of $\GL_4(F)$ with central character $\omega_{\Pi} = \mu^2$.
\vskip 10pt

For a reductive group $G$ over $F$, we write $\Pi(G)$ for the set of isomorphism classes of irreducible smooth representations of $G(F)$. Now one of the results we show is the following theorem, which was stated in \cite[Thm. 5.6]{GT1}.
\vskip 5pt

\begin{Thm}  \label{T:summary}
(i) The theta correspondence for $\GSO(D) \times \GSp_4$
 defines a bijection
 \[  \Pi(\GSO_{4,0})/\sim\; \longleftrightarrow \Pi(\GSp_4)_{ng}^{temp}, \]
 where $\sim$ is the equivalence relation defined by the action of $\GO_{4,0}(F)$ on $\Pi(\GSO_{4,0})$ and
 \[ \Pi(\GSp_4)_{ng}^{temp}: =
 \{ \text{non-generic essentially tempered representations of $\GSp_4(F)$} \}.\]
 Moreover, the image of the subset of $\tau^D_1 \boxtimes \tau^D_2$'s,  with $\tau^D_1 \ne \tau^D_2$, is precisely the subset of non-generic supercuspidal representations of $\GSp_4(F)$. The other representations in the image are the non-discrete series representations in Table 1, NDS(c).
\vskip 5pt

\noindent (ii) The theta correspondence for $\GSO_{2,2} \times \GSp_4$ defines an injection
\[  \Pi(\GSO_{2,2})/\sim\; \longrightarrow \Pi(\GSp_4). \]
The image is disjoint from $\Pi(\GSp_4)_{ng}^{temp}$ and consists of:
\begin{itemize}
 \item[(a)]  the generic discrete series representations (including supercuspidal ones) whose standard $L$-factor $L(s, \pi, std)$ has a pole at $s = 0$.

\item[(b)] the non-discrete series representations in [Table 1, NDS(b, d,e)].
\end{itemize}
 Moreover, the images of the representations $\tau_1 \boxtimes \tau_2$'s, with $\tau_1 \ne \tau_2$ discrete series representations of $\GL_2(F)$, are precisely the representations in (a).
\vskip 5pt

\noindent (iii) The theta correspondence for $\GSp_4 \times \GSO_{3,3}$ defines an injection
\[  \Pi(\GSp_4) \smallsetminus \Pi(\GSp_4)_{ng}^{temp} \longrightarrow \Pi(\GSO_{3,3}) \subset \Pi(\GL_4) \times \Pi(\GL_1). \]
Moreover, the representations  of $\GSp_4(F)$ not accounted for by (i) and (ii) above are
\begin{itemize}
\item[(a)]  the generic discrete series representations $\pi$ whose standard factor $L(s,\pi, std)$ is holomorphic at $s=0$.   The images of these representations  under the above map are precisely the discrete series representations $\Pi \boxtimes \mu$ of
$\GL_4(F) \times  \GL_1(F)$ such that $L(s, {\bigwedge}^2 \phi_{\Pi} \otimes \mu^{-1})$ has a pole at $s = 0$. Here, $\phi_{\Pi}$ is the $L$-parameter of $\Pi$.

\vskip 5pt
\item[(b)] the non-discrete series representations in [Table 1, NDS(a)]. The images of these under the above map consists of non-discrete series representations  $\Pi \boxtimes \mu$ such that
\[ \phi_{\Pi} = \rho \oplus \rho \cdot \chi \quad \text{and} \quad \mu = \det \rho \cdot \chi,  \]
for  an irreducible two dimensional $\rho$ and a character $\chi \ne 1$.
\end{itemize}
\vskip 5pt

\noindent (iv) If a representation $\pi$ of $\GSp_4(F)$ participates in theta correspondence with  $\GSO(V_2)  $, so that
\[  \pi = \theta(\tau_1 \boxtimes \tau_2) = \theta(\tau_2 \boxtimes \tau_1),\]
then $\pi$ has a nonzero theta lift to $\GSO_{3,3}$.
If $\Pi \boxtimes \mu$ is the small theta lift of $\pi$ to $\GSO_{3,3}(F)$, with $\Pi$ a representation of  $\GL_4(F)$, then
\[  \phi_{\Pi} = \phi_{\tau_1} \oplus \phi_{\tau_2} \quad \text{and} \quad  \mu = \omega_{\pi} = \det \phi_{\tau_1} = \det \phi_{\tau_2}. \]
  \end{Thm}

\vskip 10pt

Indeed, our main results give much more complete and explicit information than the above theorem.  In particular, we take note of Thms. \ref{T:local-theta-SO(D)},
\ref{T:local-theta-GSO(2,2)} and \ref{T:local-theta}, Cors. \ref{C:comp} and \ref{C:unram}, as well as Props. \ref{P:L-par} and \ref{P:generic}.
\vskip 10pt

\noindent{\bf Acknowledgments:} W. T. Gan is partially supported by NSF grant DMS-0801071.

\vskip 10pt

\section{\bf Similitude Theta Correspondences}  \label{S:theta}

In this section, we shall describe some basic properties of the theta correspondence for similitude groups. The definitive reference for this subject matter is the paper \cite{Ro1} of B. Roberts.
However, the results of \cite{Ro1} are not sufficient for our purposes and need to be somewhat extended.

\vskip 10pt
Consider the dual pair $\OO(V) \times \Sp(W)$; for simplicity, we assume that $\dim V$ is even.
For each non-trivial additive character $\psi$, let $\omega_{\psi}$ be the Weil representation for $\OO(V) \times \Sp(W)$, which can be described as follows. Fix a Witt decomposition
$W = X \oplus Y$ and let $P(Y) = \GL(Y) \cdot N(Y)$ be the parabolic subgroup stabilizing the maximal isotropic subspace $Y$. Then
\[  N(Y) = \{ b \in Hom(X,Y) : b^t  = b \}, \]
where $b^t \in Hom(Y^*, X^*) \cong Hom(X,Y)$.
The Weil representation $\omega_{\psi}$  can be realized on $S(X \otimes V)$ and the action of $P(Y) \times \OO(V)$ is
given by the usual formulas:
\vskip 5pt
\[  \begin{cases}
\omega_{\psi}(h)\phi(x) = \phi(h^{-1} x), \quad \text{for $h \in \OO(V)$;} \\
\omega_{\psi}(a)\phi(x) =  \chi_V(\det_Y(a)) \cdot |\det_Y(a)|^{\frac{1}{2} \dim V} \cdot \phi(a^{-1} \cdot x), \quad \text{for $a \in \GL(Y)$;} \\
\omega_{\psi}(b) \phi(x) = \psi( \langle bx, x \rangle) \cdot \phi(x), \quad \text{for $b \in N(Y)$,} \end{cases}
\]
\vskip 5pt
\noindent where  $\chi_V$ is the quadratic character associated to $\operatorname{disc} V \in F^{\times}/ F^{\times 2}$ and $\langle -, -\rangle$ is the natural symplectic form on $W \otimes V$.
To describe the full action of $\Sp(W)$, one needs to specify the action of a Weyl group element, which acts by a Fourier transform.

\vskip 10pt

If $\pi$ is an irreducible representation of $\OO(V)$ (resp. $\Sp(W)$), the maximal $\pi$-isotypic quotient has the form
\[  \pi \boxtimes  \Theta_{\psi}(\pi) \]
for some smooth representation of $\Sp(W)$ (resp. $\OO(V)$). We call $\Theta_{\psi}(\pi)$ the big theta lift of $\pi$. It is known that $\Theta_{\psi}(\pi)$ is of finite length and hence is admissible. Let
$\theta_{\psi}(\pi)$ be the maximal semisimple quotient of $\Theta_{\psi}(\pi)$; we call it the small theta lift of $\pi$.  Then it was a conjecture of Howe that
\vskip 5pt

\begin{itemize}
\item $\theta_{\psi}(\pi)$ is irreducible whenever $\Theta_{\psi}(\pi)$ is non-zero.
\vskip 5pt
\item the map $\pi \mapsto \theta_{\psi}(\pi)$ is injective on its domain.
\end{itemize}
\vskip 5pt

\noindent This has been proved by Waldspurger when the residual characteristic $p$ of $F$ is not $2$ and can be checked in many low-rank cases, regardless of the residual characteristic of $F$. If the Howe conjecture is true in general, our treatment for the rest of the paper can be somewhat simplified. However, because we would like to include the case $p =2$ in our discussion, we shall refrain from assuming Howe's conjecture in this paper.
\vskip 10pt

With this in mind, we take note of the following result which was shown by Kudla \cite{K} for any residual characteristic $p$:
\vskip 5pt

\begin{Prop} \label{P:Kudla}
\noindent (i) If $\pi$ is supercuspidal, $\Theta_{\psi}(\pi) = \theta_{\psi}(\pi)$ is irreducible or zero.
\vskip 5pt

\noindent (ii) If $\theta_{\psi}(\pi_1) = \theta_{\psi}(\pi_2) \ne 0$ for two supercuspidal representations
$\pi_1$ and $\pi_2$, then $\pi_1 = \pi_2$.
\end{Prop}
\vskip 5pt

\noindent One of the main purposes of this section is to extend this result of Kudla to the case of similitude groups.

\vskip 5pt

Let $\lambda_V$ and $\lambda_W$ be the similitude factors of $\GO(V)$ and $\GSp(W)$ respectively.
We shall consider the group
\[  R = \GO(V) \times \GSp(W)^+ \]
where $\GSp(W)^+$ is the subgroup of $\GSp(W)$ consisting of elements $g$ such that $\lambda_W(g)$ is in the image of $\lambda_V$. In fact, for the cases of interest in this paper (see the next section),  $\lambda_V$ is surjective, in which case $\GSp(W)^+ = \GSp(W)$.
\vskip 5pt

The group $R$ contains the subgroup
\[  R_0 = \{ (h,g) \in R: \lambda_V(h) \cdot \lambda_W(g) = 1\}. \]
 The Weil representation $\omega_{\psi}$ extends naturally to the group $R_0$ via
 \[  \omega_{\psi}(g,h)\phi = |\lambda_V(h)|^{- \frac{1}{8}\dim V \cdot \dim W} \omega(g_1, 1)(\phi \circ h^{-1}) \]
where
\[  g_1 = g \left(  \begin{array}{cc}
\lambda(g)^{-1} & 0 \\
0 & 1  \end{array} \right) \in \Sp(W). \]
Note that this differs from the normalization used in [Ro1].
Observe in particular that the central elements $(t,t^{-1}) \in R_0$ act by the quadratic character
$\chi_V(t)^{\frac{\dim W}{2}}$.
 \vskip 5pt

Now consider the (compactly) induced representation
\[  \Omega = ind_{R_0}^R \omega_{\psi}. \]
As a representation of $R$, $\Omega$ depends only on the orbit of $\psi$ under the evident action of  $Im \lambda_V \subset F^{\times}$. For example, if $\lambda_V$ is surjective, then $\Omega$ is independent of $\psi$. For any irreducible representation
$\pi$ of $\GO(V)$ (resp. $\GSp(W)^+$),   the maximal $\pi$-isotypic quotient of $\Omega$ has the form
\[    \pi \otimes  \Theta(\pi) \]
where $\Theta(\pi)$ is some smooth representation of $\GSp(W)^+$ (resp. $\GO(V)$). Further, we let
$\theta(\pi)$ be the maximal semisimple quotient of $\Theta(\pi)$. Note that though $\Theta(\pi)$
may be reducible, it has a central character $\omega_{\Theta(\pi)}$ given by
\[  \omega_{\Theta(\pi)} = \chi_V^{\frac{\dim W}{2}} \cdot \omega_{\pi}. \]
The extended Howe conjecture for similitudes says that $\theta(\pi)$ is irreducible whenever
$\Theta(\pi)$ is non-zero, and the map $\pi \mapsto \theta(\pi)$ is injective on its domain.
It was shown by Roberts [Ro1] that this follows from the Howe conjecture for isometry groups, and thus holds if $p \ne 2$.
\vskip 5pt

In any case, we have the following lemma which relates the theta correspondence for isometries and similitudes; the proof is given in [GT1, Lemma 2.2].
\vskip 5pt

\begin{Lem} \label{L:similitude1}

(i) Suppose that $\pi$ is an irreducible representation of a similitude group and $\tau$ is a constituent of the restriction of $\pi$ to the isometry group. Then $\theta_{\psi}(\tau) \ne 0$ implies that $\theta(\pi) \ne 0$.

\vskip 10pt

\noindent (ii) Suppose that
\[  \Hom_R(\Omega, \pi_1 \boxtimes \pi_2) \ne 0. \]
Suppose further that for each constituent $\tau_1$ in the restriction of $\pi_1$ to $\OO(V)$,
$\theta_{\psi}(\tau_1)$ is irreducible and
the map $\tau_1 \mapsto \theta_{\psi}(\tau_1)$ is injective on the set of irreducible constituents of
$\pi_1|_{\OO(V)}$. Then there is a uniquely determined bijection
\vskip 5pt
\[ f:  \{ \text{irreducible summands of $\pi_1|_{\OO(V)}$} \}  \longrightarrow
 \{ \text{irreducible summands of $\pi_2|_{\Sp(W)}$} \}\]
\vskip 10pt
\noindent such that for any irreducible summand $\tau_i$ in the restriction of $\pi_i$ to the relevant isometry group,
\vskip 3pt
\[  \tau _2 = f(\tau_1)  \Longleftrightarrow
\Hom_{\OO(V) \times \Sp(W)} (\omega_{\psi}, \tau_1  \boxtimes \tau_2) \ne 0 . \]
\vskip 5pt
\noindent One has the  analogous statement with the roles of $\OO(V)$ and $\Sp(W)$ exchanged.
 \vskip 10pt

\noindent (iii)  If $\pi$ is a representation of $\GO(V)$  (resp. $\GSp(W)^+$) and the restriction of $\pi$ to the relevant isometry group is $\oplus_i  \tau_i$, then as representations of $\Sp(W)$ (resp. $\OO(V)$),
\[  \Theta(\pi)  \cong  \bigoplus_i \Theta_{\psi}(\tau_i). \]
In particular, $\Theta(\pi)$ is admissible of finite length. Moreover,
if $\Theta_{\psi}(\tau_i) = \theta_{\psi}(\tau_i)$ for each  $i$, then
\[  \Theta(\pi) = \theta(\pi). \]
\end{Lem}

 \vskip 10pt

In addition, we have [GT1, Prop. 2.3]:
\vskip 5pt

\begin{Prop} \label{P:super}
Suppose that $\pi$ is a supercuspidal representation of $\GO(V)$ (resp. $\GSp(W)^+$). Then we have:
\vskip 5pt

\noindent (i) $\Theta(\pi)$ is either zero or is an irreducible representation of $\GSp(W)^+$ (resp. $\GO(V)$).
\vskip 5pt

\noindent (ii) If $\pi'$ is another supercuspidal representation such that $\Theta(\pi') = \Theta(\pi) \ne 0$, then $\pi' = \pi$.
\end{Prop}
\vskip 5pt

 We now specialize to the cases of interest in this paper.
 Henceforth, we shall only consider the case when $\dim W  =4$, so that
 \[  \GSp(W) \cong \GSp_4(F). \]
Moreover, we shall only consider quadratic spaces $V$ with $\dim V = 4$ or $6$. We describe these quadratic spaces in greater detail.
\vskip 5pt

Let $D$ be a (possibly split) quaternion algebra over $F$ and let $\mathbb{N}_D$ be its reduced norm.
Then $(D, \mathbb{N}_D)$ is a rank 4 quadratic space.
We have an isomorphism
\[  \GSO(D, \mathbb{N}_D) \cong (D^{\times} \times D^{\times})/ \{(z,z^{-1}): z \in \GL_1\} \]
via the action of the latter on $D$ given by
\[  (\alpha, \beta) \mapsto  \alpha x \overline{\beta}. \]
Moreover, an element of $\GO(D,\mathbb{N}_D)$ of determinant $-1$ is given by the conjugation action $c: x \mapsto \overline{x}$ on $D$.  We have:
\[  \GSO(D) \cong \begin{cases}
\GSO_{2,2}(F) \text{ if $D$ is split;} \\
\GSO_{4,0}(F) \text{  if $D$ is non-split.} \end{cases} \]
The similitude character of $\GSO(D)$ is given by
\[  \lambda_D: (\alpha, \beta) \mapsto N_D(\alpha \cdot \beta), \]
which is surjective onto $F^{\times}$.
Since $\lambda_D$ is surjective, we have $\GSp(W)^+ = \GSp_4(W)$ and the induced Weil representation is a representation of $\GO(D) \times \GSp(W)$. The investigation of the theta correspondence for these dual pairs has been initiated by B. Roberts [Ro2].
In  Thm. \ref{T:local-theta-GSO(2,2)} and Thm. \ref{T:local-theta-SO(D)} below, we shall  complete the study initiated in  [Ro2] by giving an explicit determination of the theta correspondence.

 \vskip 10pt

Now consider the rank 6 quadratic space:
\[  (V_D, q_D)  = (D , \mathbb{N}_D) \oplus \mathbb{H} \]
where $\mathbb{H}$ is the hyperbolic plane. Then one has an isomorphism
\[  \GSO(V_D) \cong (\GL_2(D)  \times \GL_1)/ \{ (z \cdot \operatorname{Id},\; z^{-2}): z \in \GL_1\}. \]
To see this, note that the quadratic space $V_D$ can also be described as the space of
$2 \times 2$-Hermitian matrices with entries in $D$, so that a typical element has the form
\[  (a,d; x) = \left( \begin{array}{cc}
a & x \\
\overline{x} & d \end{array} \right),  \qquad \text{$a, d \in F$ and $x \in D$},  \]
equipped with the quadratic form $- \det (a,d;x) =  -ad + \mathbb{N}_D(x)$.
The action of $\GL_2(D)  \times \GL_1$ on this space is given by
\[  (g,z)(X) = z \cdot g \cdot X \cdot \overline{g}^t. \]
\noindent  The similitude factor of $\GSO(V_D)$ is given by
\[  \lambda_{D}(g,z) = N(g) \cdot z^2, \]
where $N$ is the reduced norm on the central simple algebra $\text{M}_2(D)$. Thus,
 \[  \SO(V_D) = \{ (g,z) \in \GSO(V_D): N(g) \cdot z^2 = 1\}. \]
\vskip 5pt

In this paper,  we only need to consider $V_D$ when $D$ is split. Thus, we shall suppress $D$ from the notations,  so that from now on throughout this paper, $V$ denotes the 6 dimensional split quadratic space, i.e.
\[
    V=\H\oplus\H\oplus\H\quad\text{and}\quad\GSO(V)=\GSO_{3,3}.
\]
Moreover,  since $\lambda_V$ is surjective, we have $\GSp(W)^+ = \GSp(W)$, so that
 the induced Weil representation $\Omega$ is a representation of $R = \GSp(W) \times \GO(V)$.
We shall only consider the
theta correspondence for $\GSp(W) \times \GSO(V)$. There is no significant loss in restricting to
$\GSO(V)$ because of the following lemma:
\vskip 5pt

\begin{Lem}
Let $\pi$ (resp. $\tau$) be an irreducible representation of $\GSp(W)$ (resp. $\GO(V)$) and
suppose that
\[  \Hom_{\GSp(W) \times \GO(V)}(\Omega, \pi \otimes \tau) \ne 0. \]
Then $\tau$ is irreducible as a representation of $\GSO(V)$. If $\nu_0 = \lambda_V^{-3}  \cdot \det$ is the unique non-trivial quadratic character of $\GO(V)/\GSO(V)$, then $\tau \otimes \nu_0$ does not participate in the theta correspondence with $\GSp(W)$.
\end{Lem}
 \vskip 5pt

 \begin{proof}
  First note that $\tau$ is irreducible when restricted to $\GSO(V)$ if and only if
 $\tau \otimes \nu_0 \ne \tau$. By a well-known result of Rallis [R, Appendix] (see also [Pr1, \S 5, Pg. 282]),  the lemma holds in the setting of isometry groups.
 Suppose that $\tau|_{O(V)} = \oplus_i \tau_i$. Then this result of Rallis implies that $\tau_i$ is irreducible when restricted to $\SO(V)$, so that $\tau_i \otimes \nu_0 \ne \tau_i$, and
 $\tau_i \otimes \nu_0$ does not participate in the theta correspondence with $Sp(W)$.
 This implies that $\tau \otimes \nu_0 \ne \tau$ and $\tau \otimes \nu_0$ does not participate in the theta correspondence with $\GSp(W)$.
 \end{proof}
\vskip 10pt

 Proposition \ref{P:super} and the above lemma imply:
  \vskip 5pt

 \begin{Cor}
 If $\pi$ is a supercuspidal representation of $\GSp_4(F)$ and $\theta(\pi)$ is nonzero,
 then $\theta(\pi)$ is irreducible as a representation of $\GSO(V)$. Moreover, if $\theta(\pi) = \theta(\pi')$ as a representation of $\GSO(V)$ for some other supercuspidal $\pi'$, then $\pi = \pi'$.
 \end{Cor}
 \vskip 10pt

 If $p \ne 2$, then the above corollary would hold for any irreducible representation $\pi$ because one knows that Howe conjecture for isometry groups.
 \vskip 10pt

A study of  the local theta correspondence for $\GSp_4 \times \GSO_{3,3}$ was undertaken in [W] and [GT2].  In  Thm \ref{T:local-theta} below, we give a complete determination of this theta correspondence.
\vskip 15pt

 \section{\bf A Result of Kudla-Rallis}

 In this section, we recall a fundamental general result of Kudla-Rallis [KR] before specializing it to the cases of interest in this paper.

 \vskip 5pt

Let $W_n$ be the 2n-dimensional symplectic vector space with associated
symplectic group $\Sp(W_n)$ and consider the two towers of orthogonal groups attached to the quadratic spaces with trivial discriminant. More precisely, let
\[  V_m = \mathbb{H}^m \quad \text{and} \quad V^{\#}_m = D \oplus \mathbb{H}^{m-2} \]
and denote the orthogonal groups by $\OO(V_m)$ and $\OO(V^{\#}_{m})$ respectively.
For an irreducible representation $\pi$ of $\Sp(W_n)$, one may consider the theta lifts
$\theta_m(\pi)$ and $\theta^{\#}_m(\pi)$ to $\OO(V_m)$ and $\OO(V^{\#}_{m})$ respectively (with respect to
a fixed non-trivial additive character $\psi$). Set
\[  \begin{cases}
m(\pi) = \inf \{ m: \theta_m(\pi) \ne 0 \}; \\
m^{\#}(\pi) = \inf \{m: \theta^{\#}_m(\pi) \ne 0 \}.
\end{cases} \]
\vskip 5pt

\noindent Then Kudla and Rallis  \cite[Thms. 3.8 \& 3.9]{KR} showed:
\vskip 5pt

\begin{Thm}  \label{T:KR}
(i)  For any irreducible representation $\pi$ of $\Sp(W_n)$,
\[  m(\pi) + m^{\#}(\pi) \geq 2n  +2. \]
\vskip 5pt

\noindent (ii) If $\pi$ is a supercuspidal representation of $\Sp(W_n)$, then
\[  m(\pi) + m^{\#}(\pi) = 2n  +2. \]
\end{Thm}
\vskip 5pt

If we specialize this result to the case $\dim W_n = 4$ , we obtain:
 \vskip 5pt

\begin{Thm} \label{T:dichotomy}
(i) Let $\pi$ be an irreducible supercuspidal representation of $\GSp_4(F)$. Then one has the following two mutually exclusive possibilities:
\vskip 5pt

\noindent (A) $\pi$ participates in the theta correspondence with $\GSO(D) =\GSO_{4,0}(F)$, where $D$ is non-split;
\vskip 5pt

\noindent (B) $\pi$ participates in the theta correspondence with $\GSO(V) = \GSO_{3,3}(F)$.

 \end{Thm}

\vskip 5pt

One of the purposes of this paper is to extend the dichotomy of this theorem to all irreducible representations of $\GSp_4(F)$. Moreover, our main results make this dichotomy completely explicit. For example, we will see that the supercuspidal representations of Type (A) are precisely those which are non-generic.
 \vskip 10pt

 On the other hand, one may consider the mirror situation, where one fixes an irreducible representation of $\OO(V_m)$ or $\OO(V_m^{\#})$ and consider its theta lifts $\theta_n(\pi)$ to the tower of symplectic groups $\Sp(W_n)$. Then, with $n(\pi)$ defined in the analogous fashion, one expects that
 \[  n(\pi) + n(\pi \otimes \det) = 2m. \]
   For similitude groups, this implies that
   \[  n(\pi) + n(\pi \otimes \nu_0) = 2m, \]
   where $\nu_0$ is the non-trivial character of $\GO(V_m)/\GSO(V_m)$.
   When $m = 2$,  this expectation has been proved by B. Roberts [Ro2, Thm. 7.8 and Cor. 7.9] as follows:
 \vskip 5pt

 \begin{Thm} \label{T:roberts}
 Let $\pi$ be an irreducible representation of $\GO(D)$, where $D$ is possibly split.
 Then
 \[   n(\pi) + n(\pi \otimes \nu_0) = 4. \]
  In particular, if $\pi  = \pi \otimes \nu_0$, then  $n(\pi) = 2$.
 \end{Thm}

 \vskip 15pt

\section{\bf Whittaker Modules of Weil Representations} \label{S:Whit}

In this section, we describe the Whittaker modules of the Weil representations $\Omega_{W,D}$
and $\Omega_{W,V}$, with $\dim W = 4$. We omit the proofs since they are by-now-standard; see for example [GT2 Prop. 7.4] and [MS, Prop. 4.1].
\vskip 5pt

\begin{Prop} \label{P:Whit1}
Let $D$ be a (possibly split) quaternion algebra over $F$ and consider the Weil representation
$\Omega_{D, W}$ of $\GO(D) \times \GSp(W)$. Let $U$ be the unipotent radical of a Borel subgroup of $\GSp(W)$ and let $\psi$ be a generic character of $U(F)$. Similarly, let $U_0$ be the unipotent radical  of a Borel subgroup of $\GO(D)$ when $D$ is split, and let $\psi_0$ be a generic character of $U_0(F)$.
Then we have:
\[  (\Omega_{D,W})_{U, \psi}  = 0 \]
if $D$ is non-split, and
\[   (\Omega_{D,W})_{U, \psi} \cong ind_{U_0}^{\GSO(D)} \psi_0 \]
if $D$ is split.
\end{Prop}

\vskip 10pt

\begin{Cor} \label{C:Whit1}
(i) If $D$ is non-split, then no generic representation of $\GSp(W)$ participates in theta correspondence with $\GO(D)$.

\vskip 5pt
(ii) Let $\pi$ be an irreducible generic representation of $\GO(V_2) = \GSO_{2,2}(F)$. Then $\pi$ is generic if and only if $\Theta_{V_2,W}(\pi)$ contains a generic constituent, in which case this generic constituent is unique.
\end{Cor}

\vskip 10pt

\begin{Prop}  \label{P:Whit2}
Consider the Weil representation $\Omega_{W,V}$ of $\GSp(W) \times \GSO(V)$.
Let $U$ be the unipotent radical of a Borel subgroup of $\GSp(W)$ and let $\psi$ be a generic character of $U(F)$. Similarly, let $U_0$ be
the unipotent radical  of a Borel subgroup of $\GO(V)$, and let $\psi_0$ be a generic character of $U_0(F)$. Then we have
\[  (\Omega_{W,V})_{U_0, \psi_0}  \cong ind_U^{\GSp(W)} \psi. \]
\end{Prop}

\vskip 10pt

\begin{Cor} \label{C:Whit2}
Let $\pi$ be an irreducible representation of $\GSp(W)$. Then $\pi$ is generic if and only if
$\Theta_{W,V}(\pi)$ contains a generic constituent, in which case this generic constituent is unique.
\end{Cor}

\vskip 10pt
Corollaries \ref{C:Whit1}(i) and \ref{C:Whit2} imply:
\vskip 5pt

\begin{Cor}
The dichotomy result of Theorem \ref{T:dichotomy} holds for generic representations of $\GSp(W)$.
\end{Cor}

\vskip 15pt

\section{\bf Representations of $\GSp_4(F)$}

In order to state our main results, we need to introduce
some notations and recall some results about various
principal series representations of $\GSp_4(F)$.  In the following, by the term ``discrete series" or ``tempered" representations of $\GSp(W) = \GSp_4(F)$ or $\GSO(V)= \GSO_{3,3}(F)$,  we mean a representation which is
unitarizable discrete series or unitarizable tempered after twisting by a 1-dimensional character.
\vskip 10pt

\subsection{\bf Principal series representations.}
Recall from
Section \ref{S:theta} that we have a Witt decomposition
$W = Y^* \oplus Y$. Suppose that
\[  Y^* = F \cdot e_1 \oplus F \cdot e_2 \quad \text{and} \quad  Y =
F\cdot f_1 \oplus F \cdot f_2 \]
and consider the decomposition $W = Fe_1
\oplus W' \oplus Ff_1$, where $W' = \langle e_2, f_2 \rangle$. Let
$Q(Z) = L(Z) \cdot U(Z)$ be the maximal parabolic stabilizing the
line $Z = F \cdot f_1$, so that
\[  L(Z) = \GL(Z) \times \GSp(W') \]
and $U(Z)$ is a Heisenberg group:
\[  \begin{CD}
1 @>>> \operatorname{Sym}^2 Z @>>> U(Z) @>>> W' \otimes Z @>>> 1. \end{CD} \]
This is typically called the Klingen parabolic subgroup in the literature. A
representation of $L(Z)$ is thus of the form $\chi \boxtimes \tau$
where $\tau$ is a representation of $\GSp(W') \cong \GL_2(F)$. We let
$I_{Q(Z)}(\chi,\tau)$ be the corresponding parabolically induced
representation.  If $I_{Q(Z)}(\chi,\tau)$ is a standard module, then it has a unique irreducible quotient (the Langlands quotient), which we shall denote by $J_{Q(Z)}(\chi,\tau)$. The same notation applies to
other principal series representations to be introduced later.
\vskip 10pt

The module structure of $I_{Q(Z)}(\chi,\tau)$
is known by Sally-Tadic [ST] and a convenient reference is [RS].
In particular,  we note the following:

 \vskip 5pt

\begin{Lem}  \label{L:Q(Z)}
(a) Let $\tau$ be a supercuspidal representation of $\GL_2(F)$. The
induced representation $I_{Q(Z)}(\chi, \tau)$ is reducible iff one
of the following holds: \vskip 5pt

(i) $\chi = 1$; \vskip 5pt

(ii) $\chi = \chi_0 |-|^{\pm 1}$ and $\chi_0$ is a non-trivial
quadratic character  such that $\tau \otimes \chi_0 \cong \tau$.
\vskip 5pt

In case (i), the representation $I_{Q(Z)}(1, \tau)$ is the direct
sum of two irreducible tempered representations, exactly one of which is
generic.  In case (ii), assuming without loss of generality that
$\chi = \chi_0  \cdot |-|$, one has a (non-split) short exact
sequence:
\[  \begin{CD}
0 @>>>   St(\chi_0, \tau_0) @>>> I_{Q(Z)}(\chi_0|-|, \tau_0|-|^{-1/2}) @>>> Sp(\chi_0,\tau_0)
@>>> 0 \end{CD} \]
 where $St(\chi_0,\tau_0)$ is a generic discrete series
 representation and the Langlands quotient $Sp(\chi_0,\tau_0)$ is non-generic.
\vskip 10pt

(b) If $\tau$ is the twisted Steinberg representation  of
$\GL_2$, then $I_{Q(Z)}(\chi,\tau)$ is reducible iff one of the following holds:
\vskip 5pt

(i) $\chi = 1$; \vskip 5pt

(ii) $\chi = |-|^{\pm 2}$ \vskip 5pt

In case (i), $I_{Q(Z)}(1, st_{\chi})$ is the sum of two irreducible tempered representations,
exactly one of which is generic. In case (ii),   $I_{Q(Z)}(|-|^2, st_{\chi} \cdot |-|^{-1})$ has the
twisted Steinberg representation $St_{\PGSp_4} \otimes \chi$ as its
unique irreducible submodule. \vskip 10pt

(c) For general $\tau$, there is a standard intertwining operator
\[ I_{Q(Z)}(\chi^{-1},  \tau \otimes \chi) \longrightarrow I_{Q(Z)}(\chi, \tau), \]
which is an isomorphism if $I_{Q(Z)}(\chi,\tau)$ is irreducible.
If $ I_{Q(Z)}(\chi^{-1},  \tau \otimes \chi)$ is a standard module, then the image of this operator is the
unique irreducible submodule of $I_{Q(Z)}(\chi, \tau)$.
\end{Lem}

\vskip 10pt

Now let $P(Y) = M(Y) \cdot N(Y)$ be the Siegel parabolic subgroup stabilizing $Y$ so that
\[  M(Y) =  \GL(Y) \times \GL_1 \]
and $N(Y) \cong \Sym^2 Y$. A representation of $M(Y)$ is thus of the form $\tau \boxtimes \chi$ with $\tau$ a representation of $\GL(Y) \cong \GL_2(F)$ and $\chi$ a character of $\GL_1(F)$. We denote the normalized induced representation by $I_{P(Y)}(\tau, \chi)$. As  before, the module structure of
$I_{P(Y)}(\tau,\chi)$ is completely known [ST] and a convenient reference is [RS]. In particular, we note the following:
\vskip 10pt

\begin{Lem} \label{L:P(Y)}
(a) Suppose that $\tau$ is a supercuspidal representation of $\GL(Y) \cong \GL_2(F)$. Then
$I_{P(Y)}(\tau, \mu)$ is reducible iff $\tau = |-|^{\pm 1/2} \cdot \tau_0$ with $\tau_0$ having trivial central character.  In this case, one has a non-split short exact sequence:
\[  \begin{CD}
0 @>>> St(\tau_0, \mu_0) @>>> I_{P(Y)}(\tau_0|-|^{1/2},\mu_0 |-|^{-1/2}) @>>> Sp(\tau_0,\mu_0) @>>> 0  \end{CD} \]
where $St(\tau_0,\mu_0)$ is a generic discrete series representation and the Langlands quotient $Sp(\tau_0,\mu_0)$ is non-generic.
\vskip 10pt

(b) Suppose that $\tau$ is a twisted Steinberg representation of $\GL(Y)$. Then $I_{P(Y)}(\tau,\mu)$ is reducible iff one of the following holds:
\vskip 5pt

(i) $\tau = st \cdot |-|^{\pm 1/2}$;  in this case,
$I_{P(Y)}(st |-|^{1/2}, \mu)$ has a unique irreducible Langlands quotient and a unique irreducible tempered submodule, which is the generic summand of $I_{Q(Z)}(1, st_{\mu} \cdot |-|^{1/2})$.

\vskip 5pt

(ii) $\tau = st_{\chi} \cdot |-|^{\pm 1/2}$, with $\chi$ a non-trivial quadratic character. In this case,
the representation $I_{P(Y)}(st_{\chi} |-|^{1/2}, \mu_0|-|^{-1/2})$ has a unique irreducible Langlands quotient and a unique irreducible submodule which is a generic discrete series representation
$St( st_{\chi}, \mu_0)$. Moreover, $St(st_{\chi},\mu_0) \cong St(st_{\chi}, \chi \mu_0)$.
\vskip 5pt

(iii) $\tau = st \cdot |-|^{\pm 3/2}$; in this case, $I_{P(Y)}(st |-|^{3/2}, \mu|-|^{-3/2})$ has the twisted Steinberg representation $St_{\PGSp_4} \otimes \mu$ as its unique irreducible submodule.
\vskip 10pt

(c) There is a standard intertwining operator
\[  I_{P(Y)}(\tau, \mu) \longrightarrow I_{P(Y)}(\tau^{\vee}, \omega_{\tau} \mu), \]
which is an isomorphism if $I_{P(Y)}(\tau, \mu)$ is irreducible. If $I_{P(Y)}(\tau,\mu)$ is a standard module, then the image of this operator is the unique irreducible submodule of $I_{P(Y)}(\tau^{\vee},\omega_{\tau}\mu)$.
\end{Lem}

\vskip 10pt

Finally, let $B = P(Y) \cap Q(Z) = T \cdot U$ be a Borel subgroup of $\GSp(W)$, so that
\[  T \cong (\GL(F \cdot f_1) \times \GL(F \cdot f_2)) \times \GL_1. \]
 In particular, for characters $\chi_1$,
$\chi_2$ and $\chi$ of $\GL_1(F)$, we let $I_B(\chi_1, \chi_2; \chi)$ denote the normalized
induced representation. Again, we refer the reader to [RS] for the reducibility points and module structure of $I_B(\chi_1,\chi_2;\chi)$. We simply note here that if $\chi_1$ and $\chi_2$ are unitary,
then $I_B(\chi_1,\chi_2;\chi)$ is irreducible.

\vskip 10pt

\subsection{\bf Non-Supercuspidal Representations.}   \label{SS:non-super}
We can now give a concise enumeration of the non-supercuspidal representations of $\GSp_4(F)$.
\vskip 10pt

\subsubsection{\bf \underline{Discrete Series Representations}}  \label{SSS:DS}
\vskip 5pt

The non-supercuspidal discrete series representations of $\GSp_4(F)$ are precisely the following representations:
\vskip 5pt

\begin{itemize}
\item[(a)]  the generalized Steinberg representation $St(\chi,\tau)$ of
Lemma \ref{L:Q(Z)}(a)(ii), with $\tau$ supercuspidal and
$\chi$ a non-trivial quadratic character such that $\tau \otimes \chi = \tau$;
\vskip 5pt

\item[(b)]  the generalized Steinberg representation $St(\tau,\mu)$ of
Lemma \ref{L:P(Y)}(a) and (b)(ii), so that $\tau$ is a discrete series
representation of $\PGL_2$ and $\tau \ne st$.
\vskip 5pt

\item[(c)] the twisted Steinberg representation $St_{\PGSp_4} \otimes \chi$ of Lemma \ref{L:Q(Z)}(b)(ii) and Lemma \ref{L:P(Y)}(b)(iii).
\end{itemize}
\vskip 5pt

\noindent All these representations are generic.

\vskip 10pt
\subsubsection{\bf  \underline{Non-Discrete Series Representations}} \label{SSS:NDS}
\vskip 5pt

Now we consider non-discrete series representations.
Every tempered representation is a constituent of a twist of an induced representation
$I_P(\tau)$ with $\tau$ a {\em unitary} discrete series representation. Such an induced representation is irreducible (and thus generic) except in the setting of Lemma \ref{L:Q(Z)}(a)(i) and (b)(i). In this exceptional case,
for a discrete series representation $\tau$ of $\GSp(W') \cong \GL_2(F)$, we have:
\[  I_{Q(Z)}(1,\tau) = \pi_{gen}(\tau) \oplus \pi_{ng}(\tau) \]
where $\pi_{gen}(\tau)$ is generic and $\pi_{ng}(\tau)$ is non-generic.

\vskip 10pt

Suppose now that $\pi$ is an irreducible non-tempered representation of $\GSp_4(F)$.
By the Langlands classification, there is a unique standard representation $I_P(\sigma)$ (with $\sigma$ essentially tempered) which has $\pi$ as its unique irreducible quotient.  In fact, it will be more convenient for us to make use of the dual version of Langlands classification, so that $\pi$ is the unique submodule of $I_P(\sigma)$, for an essentially tempered representation $\sigma$
whose  central character is in the relevant negative Weyl chamber. Note that since the Levi subgroups of $\GSp_4$ are all of $\GL$-type,  any essentially tempered representation of a Levi subgroup of $\GSp_4(F)$ is fully induced from a twist of a discrete series representation.
\vskip 5pt

Summarizing the above discussions,  we have:
\vskip 5pt

\begin{Prop} \label{P:NDS}
The non-discrete series representations of $\GSp_4(F)$
fall into the following three disjoint families:
\vskip 5pt

\begin{enumerate}

\item[(a)]  $\pi \hookrightarrow I_{Q(Z)}(\chi |-|^{-s}, \tau)$ with $\chi$ a unitary character, $s \geq 0$
and $\tau$ a discrete series representation of $\GSp(W')$ up to twist. In fact, $\pi$ is the unique irreducible submodule, except in the exceptional tempered case mentioned above (with  $\chi$ trivial and $s = 0$).
\vskip 5pt

\item[(b)]  $\pi \hookrightarrow I_{P(Y)}(\tau |-|^{-s}, \chi)$
with $\chi$ an arbitrary character, $s \geq 0$
and $\tau$ a unitary discrete series representation of $\GL(Y)$. In this case, $\pi$ is the unique irreducible submodule.
\vskip 5pt

\item[(c)]  $\pi \hookrightarrow I_B(\chi_1 |-|^{-s_1}, \chi_2|-|^{-s_2}; \chi)$ where $\chi_1$ and $\chi_2$ are unitary and $s_1 \geq s_2 \geq  0$. By induction in stages, we see that
\[  I_B(\chi_1 |-|^{-s_1}, \chi_2|-|^{-s_2}; \chi) = I_{Q(Z)}(\chi_1|-|^{-s_1}, \pi(\chi \chi_2|-|^{-s_2}, \chi)). \]
We may now consider two subcases:
\vskip 5pt

\item[(i)] if $\chi_2 |-|^{-s_2} \ne |-|^{-1}$, then $\pi(\chi \chi_2|-|^{-s_2}, \chi)$  is irreducible and we have:
\[  \pi \hookrightarrow   I_{Q(Z)}(\chi_1|-|^{-s_1}, \pi(\chi \chi_2|-|^{-s_2}, \chi)) \]
as the unique irreducible submodule.
\vskip 5pt

\item[(ii)] if  $\chi_2 |-|^{-s_2} =  |-|^{-1}$, then $\pi(\chi \chi_2|-|^{-s_2}, \chi)$ contains $\chi|-|^{-1/2}$ as a unique irreducible submodule and so
\[   \pi \hookrightarrow   I_{Q(Z)}(\chi_1|-|^{-s_1}, (\chi \circ \det)\cdot  |\det|^{-1/2})  \]
as the unique irreducible submodule.
\end{enumerate}
\end{Prop}
\vskip 10pt

\section{\bf Representations of $\GSO(D)$}

In this section, we set up some notations for the irreducible representations of $\GSO(D)$.
In this case, we have
\[
 V = (D, -\N_D)
\]
where $D$ is a quaternion $F$-algebra (possibly split) with reduced norm $\N_D$.
We have the identification
\[
 \GSO(V) \cong (D^{\times} \times D^{\times}) /
 \{ (z, z^{-1}) \, | \, z \in F^{\times} \}
\]
via
\[
 (g_1, g_2) : x \longmapsto g_1 \cdot x \cdot \bar{g}_2.
\]
Moreover,
the main involution $x \mapsto \bar{x}$ on $D$ gives an order two element
$c$ of $\OO(V)$ with determinant $-1$,
so that $\GO(V) = \GSO(V) \rtimes \langle c \rangle$.
The conjugation of $c$ on $\GSO(V)$ is given by
\[
 (g_1, g_2) \longmapsto (g_2, g_1).
\]
\vskip 5pt

Thus, an irreducible representation of $\GSO(D)$ is of the form
$\tau_1 \boxtimes\tau_2$ for  irreducible representations $\tau_i$
of $D^{\times}$ with the same central character $\omega_{\tau_1} = \omega_{\tau_2}$.
Moreover, the action of $c$ sends $\tau_1 \boxtimes \tau_2$ to $\tau_2 \boxtimes \tau_1$.
If $\tau_1 \ =\tau_2$, then the representation $\tau_1 \boxtimes \tau_2$ extends to $\GO(D)$ in two different ways. If $\tau_1 \ne \tau_2$, then
\[  \ind_{\GSO(D)}^{\GO(D)} \tau_1 \boxtimes \tau_2 = \ind_{\GSO(D)}^{\GO(D)} \tau_2 \boxtimes \tau_1 \]
is an irreducible representation of $\GO(D)$. This describes all the irreducible representations of $\GO(D)$.

\vskip 10pt

When $D$ is split, the quadratic space $D$ is split
and we have a Witt decomposition
\[  D = X \oplus X^* \]
for a two dimensional isotropic space $X$.
Let $P(X) = M(X) \cdot N(X)$ be the parabolic subgroup stabilizing $X$,
so that
\[
 M(X) \cong \GL(X) \times \GL_1
 \quad \text{and} \quad
 N(X) \cong \wedge^2 X.
\]
For an irreducible representation
$\tau \boxtimes \chi$ of $\GL(X) \times \GL_1(F) \cong \GL_2(F) \times F^{\times}$,
we let $I_{P(X)}(\tau, \chi)$ denote the normalized induced representation.
The following lemma is easy to check.
\vskip 5pt

\begin{Lem}
Under the identification $\GSO(D) \cong (\GL_2 \times \GL_2) / F^{\times}$,
we have
\[
 \pi(\chi_1,\chi_2) \boxtimes \tau
 = I_{P(X)}(\tau^{\vee} \cdot \chi_1, \chi_2)
 = I_{P(X)}(\tau \cdot \chi_2^{-1}, \chi_2).
\]
\end{Lem}

\vskip 15pt

\section{\bf Representations of $\GSO(V)$.}
 Now we need to establish some notations for the representations of $\GSO(V) = \GSO_{3,3}(F)$.
Though we may identify $\GSO(V)$ as a quotient of $\GL_4(F) \times \GL_1(F)$ as in
Section \ref{S:theta}, it is in fact better not to do so
for the purpose of the computation of local theta lifting.
\vskip 5pt

Recall that we have a decomposition
\[ V = F \cdot (1,0) \oplus V_2 \oplus F \cdot (0,1) \]
where $V_2$ is the split rank 4 quadratic space.
Let $J = F \cdot (1,0)$ and let $P(J)= M(J) \cdot N(J)$
be the stabilizer of $J$, so that $M(J) = \GL(J) \times \GSO(V_2)$. We
represent an element of $M(J)$ by $(a, \alpha, \beta)$ with
\[ (\alpha,\beta) \in \GSO(V_2) \cong (\GL_2 \times \GL_2)/\{(z,z^{-1}): z \in F^{\times} \}. \]
For a character $\chi$ and a representation $\tau_1 \boxtimes
\tau_2$ of $\GSO(V_2)$, one may consider the normalized induced
representation $I_{P(J)}(\chi, \tau_1 \boxtimes \tau_2)$.
\vskip 5pt

Under the natural map $\GL_4 \times \GL_1 \longrightarrow \GSO(V)$, the inverse image of
$P(J)$ is the parabolic $P \times \GL_1$, where $P$ is the $(2,2)$-parabolic in $\GL_4$.
 Indeed,  under the natural map $P \times \GL_1 \longrightarrow P(J)$,
\[  \left( \left( \begin{array}{cc}
\alpha & \\
 & \beta  \end{array} \right), z \right)  \mapsto (a, \alpha' , \beta'), \]
 then we have:
 \[  \begin{cases}
 \alpha = a \cdot \overline{\alpha'}^{-1} \\
 \beta  = \beta' \\
 z = a^{-1}\cdot  \mathbb{N}(\alpha'). \end{cases} \]
 From this, one deduces that
 \[  I_P(\tau_1,\tau_2)\boxtimes \mu  \cong I_{P(J)}( \omega_1\mu^{-1} , \tau_1^{\vee} \mu \boxtimes \tau_2). \]
 \vskip 10pt
\noindent The following lemma is well-known:

\begin{Lem}
(i) Let $\tau$ be a supercuspidal representation of $\GL_2(F)$.
Then one has  a short exact sequence of representations of $\GL_4(F)$:
\[  \begin{CD}
0 @>>> St(\tau) @>>> I_P(\tau |-|^{1/2}, \tau |-|^{-1/2}) @>>>
Sp(\tau) @>>> 0
\end{CD} \]
where $St(\tau)$ is a discrete series representation (a generalized
Steinberg representation) and $Sp(\tau)$ is the unique Langlands
quotient (a generalized Speh representation). \vskip 10pt

(ii) If $\tau= st_{\chi}$, then the
principal series $I_P(\tau|-|, \tau|-|^{-1})$ has a unique
irreducible submodule which is the twisted Steinberg representation
$St_{\chi}:= St_{P\GL_4} \otimes \chi$. Moreover, $I_P(\tau|-|^{1/2}, \tau|-|^{-1/2})$ is also reducible, but its constituents are not discrete series representations.
\vskip 10pt

(iii) The situations described in (i) and (ii) are the only cases when $I_P(\tau_1,\tau_2)$ is reducible with $\tau_i$ discrete series representations.
\end{Lem}
\vskip 15pt

We shall need another parabolic of $\GSO(V)$.
Let $Q$ be  the (1,2,1)-parabolic of $\GL_4$ so that its Levi factor is
$\GL_1 \times \GL_2 \times \GL_1$. We let $I_Q(\chi_1, \tau,\chi_2)$ denote the normalized induced representation, with $\chi_i$ characters of $\GL_1(F)$ and $\tau$ a representation of $\GL_2(F)$.
Consider the image of $Q \times \GL_1$ under the natural map $\GL_4 \times \GL_1 \longrightarrow
\GSO(V)$. This image is the stabilizer $Q(E)$  of a 2-dimensional isotropic subspace $E$ containing
$J$. Writing
\[  V = E \oplus V_1 \oplus E^* \]
where $V_1$ is a split rank 2 quadratic space,  we see that the Levi factor of $Q(E)$ is
\[  L(E) = \GL(E) \times \GSO(V_1)\cong \GL_2 \times (\GL_1 \times \GL_1). \]
Thus, for a representation $\tau$ of $\GL(E)$ and a character $\chi$ of $F^{\times}$,
we may consider the normalized induced representation $I_{Q(E)}(\tau, \chi \circ \lambda_{V_1})$.
Now under the natural map $Q \times \GL_1 \longrightarrow Q(E)$, we have
\[ \left( \left( \begin{array}{ccc}
t_1 & & \\
& B &  \\
&  & t_4  \end{array} \right) , z \right) \mapsto  \left( t_1B , ( t_1t_4,  \det B) \right) \in GL(E) \times \GSO(V_1). \]
From this, one easily calculates that
\[  I_{Q(E)}(\tau , \chi \circ \lambda_{V_1}) =
I_Q(\omega_{\tau}\chi , \tau \otimes \chi, \chi)  \boxtimes \chi^2 \omega_{\tau}. \]
\vskip 5pt

Finally, we have the Borel subgroup $B_0 = P \cap Q$ of $\GL_4$ and the principal series $I_B(\chi_1, \chi_2,\chi_3,\chi_4)$.

\vskip 15pt

\section{\bf The Main Results.}

We are now ready to state our main results concerning the computation of local theta correspondences. But before stating them, let us note that for any quadratic space $V$, there is an action of $\GO(V)$ on  $\Pi(\GSO(V))$, and we denote the class of the equivalence classes defined by this action by
\[
    \Pi(\GSO(V))/\sim.
\]
 \vskip 15pt

\begin{Thm} \label{T:local-theta-SO(D)}
Suppose that $D$ is the quaternion division algebra over $F$.
Let $\tau^D_1 \boxtimes \tau^D_2$ be an irreducible representation of $\GSO(D)$. Then
we have the following:
\vskip 5pt

\noindent (i) $\Theta(\tau^D_1 \boxtimes \tau_2^D)$ is an irreducible representation of $\GSp_4(F)$.

\vskip 5pt

 \noindent (ii) If $\tau^D_1 = \tau^D_2 =\tau_D$, then
\[  \Theta(\tau_D \boxtimes \tau_D) = \pi_{ng}(JL(\tau_D)), \]
which is the unique non-generic summand of $I_{Q(Z)}(1,JL(\tau_D))$.
\vskip 5pt

\noindent (iii) If $\tau^D_1 \ne  \tau^D_2$, then
$\Theta(\tau^D_1 \boxtimes \tau_2^D) = \Theta(\tau^D_2 \boxtimes \tau^D_1)$ is a non-generic supercuspidal representation of $\GSp_4(F)$.
\vskip 5pt

\noindent (iv) The map
\[  \tau^D_1 \boxtimes \tau^D_2  \mapsto \Theta(\tau^D_1 \boxtimes \tau_2^D) \]
 defines a bijection between
 \[  \Pi(\GSO(D))/\sim\]
and
\[
  \Pi(\GSp_4)^{temp}_{ng}: = \{\text{non-generic tempered representations of $\GSp_4$} \}. \]
\end{Thm}
\vskip 20pt

\begin{Thm}  \label{T:local-theta-GSO(2,2)}
Let $\tau_1 \boxtimes\tau_2$ be an irreducible representation of $\GSO_{2,2}(F) $ and let $\theta(\tau_1 \boxtimes\tau_2)$ be its small theta lift to $\GSp_4(F)$. Then $\theta(\tau_1 \boxtimes\tau_2) = \theta(\tau_2 \boxtimes \tau_1)$ can be determined as follows.
\vskip 5pt

\noindent (i) If $\tau_1 = \tau_2 = \tau$ is a discrete series representation, then
\[  \theta(\tau_1 \boxtimes \tau_2) = \pi_{gen}(\tau), \]
which is the unique generic constituent of $I_{Q(Z)}(1,\tau)$.
\vskip 5pt

\noindent (ii) If $\tau_1 \ne \tau_2$ are both supercuspidal, then $\theta(\tau_1 \boxtimes\tau_2)$ is supercuspidal.
\vskip 5pt

\noindent (iii) If $\tau_1$ is supercuspidal and $\tau_2 = st_{\chi}$, then
\[  \theta(\tau_1 \boxtimes \tau_2) = St(\tau_1 \otimes \chi^{-1}, \chi). \]
\vskip 5pt

\noindent (iv) Suppose that $\tau_1 = st_{\chi_1}$ and $\tau_2 = st_{\chi_2}$ with $\chi_1 \ne \chi_2$,
so that $\chi_1^2 = \chi_2^2$. Then
\[  \theta(\tau_1 \boxtimes\tau_2) = St(st_{\chi_1/\chi_2}, \chi_2) = St(st_{\chi_2/\chi_1}, \chi_1). \]

\vskip 5pt

\noindent (v) Suppose that $\tau_1$ is a discrete series representation and $\tau_2 \hookrightarrow \pi(\chi,\chi')$ with $|\chi/\chi'| = |-|^{-s}$ and $s\geq 0$, so that $\tau_2$ is non-discrete series. Then
\[  \theta(\tau_1\otimes\tau_2) =  J_{P(Y)}( \tau_1 \otimes \chi^{-1}, \chi).  \]
\vskip 5pt

\noindent (vi)  Suppose that
\[
\tau_1 \hookrightarrow \pi(\chi_1, \chi_1') \quad \text{and} \quad
\tau_2 \hookrightarrow \pi(\chi_2, \chi_2') \]
with
\[  |\chi_i/\chi'_i| = |-|^{-s_i}, \quad   s_1 \geq s_2 \geq 0. \]
Then
\[  \theta(\tau_1 \boxtimes\tau_2) = J_B(\chi_2'/\chi_1,  \chi_2/\chi_1; \chi_1). \]
\vskip 5pt
In particular, the map $\tau_1 \boxtimes \tau_2 \mapsto \theta(\tau_1 \boxtimes \tau_2)$ defines an injection
\[  \Pi(\GSO_{2,2} )/\sim\;  \hookrightarrow \Pi(\GSp_4). \]
\end{Thm}

\vskip 15pt

\begin{Thm}  \label{T:local-theta}
Let $\pi$ be an irreducible representation of $\GSp_4(F)$ and consider its theta lift to $\GSO(V)$.

\vskip 5pt

\noindent (i) The representation $\theta(\pi)$ is nonzero iff $\pi \notin \Pi(\GSp_4)_{ng}^{temp}$, in which case it is irreducible.
\vskip 5pt

\noindent (ii) The map $\pi \mapsto \theta(\pi)$ defines an injective
map from $\Pi(\GSp_4) \smallsetminus \Pi(\GSp_4)_{ng}^{temp}$ into the set of irreducible representations of $\GSO(V)$.
\vskip 10pt

Moreover, for $\pi \notin \Pi(\GSp_4)_{ng}^{temp}$, $\theta(\pi)$ can be described in terms of $\pi$ as follows.
\vskip 10pt

\noindent (iii) ({\bf Supercuspidal}) If $\pi$ is supercuspidal, then $\theta(\pi)$ is supercuspidal
unless $\pi$ has a nonzero theta lift to $\GSO_{2,2}(F) $. If $\pi$ is the theta lift of $\tau_1 \boxtimes \tau_2$ from $\GSO_{2,2}(F) $, then
\[  \theta(\pi) = I_P(\tau_1,\tau_2) \boxtimes \omega_{\tau_1}. \]
\vskip 10pt

\noindent (iv) ({\bf Generalized Steinberg}) Suppose that $\pi = St(\chi,\tau)$ as in \ref{SSS:DS}(a). Then
\[  \theta(St(\chi,\tau)) =   St(\tau) \boxtimes \omega_{\tau}\chi. \]
\vskip 5pt
\noindent On the other hand, if $\pi = St(\tau,\mu)$ as in \ref{SSS:DS}(b), then
\[  \theta(St(\tau,\mu))  = I_P(\tau \otimes \mu , st \otimes \mu)  \boxtimes \mu^2. \]
\vskip 10pt

\noindent (v) ({\bf Twisted Steinberg}) If $\pi = St_{\PGSp_4} \otimes \chi$ is a twisted Steinberg representation, then
\[  \theta(St_{\PGSp_4} \otimes \chi) = St_{\chi} \boxtimes \chi^2. \]
\vskip 10pt

\noindent (vi) ({\bf Non-discrete series}) We consider the different cases according to \ref{SSS:NDS}:
\vskip 10pt

(a) Suppose that  $\pi \hookrightarrow I_{Q(Z)}(\chi, \tau)$ as in \ref{SSS:NDS}(a), so that  $|\chi| = |-|^{-s}$ with $s \geq 0$. Then
\[  \theta(\pi) = J_P(\tau, \tau \cdot \chi) \boxtimes \omega_{\tau}\chi. \]
\vskip 10pt

(b) Suppose that $\pi \hookrightarrow I_{P(Y)}(\tau, \chi)$ as in \ref{SSS:NDS}(b),  so that
$|\omega_{\tau}| = |-|^{-2s}$ with $s\geq 0$. Then
\[  \theta(\pi) = J_Q(1, \tau, \omega_{\tau}) \cdot\chi \boxtimes \chi^2 \omega_{\tau}. \]
\vskip 10pt

(c) Suppose that
\[ \pi \hookrightarrow I_B(\chi_1, \chi_2; \chi) \]
where $|\chi_1| = |-|^{-s_1}$ and $|\chi_2| = |-|^{-s_2}$ with $s_1 \geq s_2 \geq  0$. Then
\[  \theta(\pi) = \chi \cdot  J_{B_0}(1,\chi_2, \chi_1, \chi_1\chi_2) \boxtimes \chi^2 \chi_1
\chi_2. \]
 \end{Thm}
  \vskip 10pt

 \noindent{\bf Remarks:}   In Section \ref{S:table}, we display the results of the above three theorems in the form of three tables. The results there are given in terms of the usual Langlands classification which describes $\pi$ as a unique irreducible quotient of a standard module, as opposed to describing $\pi$ as a unique submodule as we have done in the theorems.

\vskip 15pt

\section{\bf Proof of Theorem \ref{T:local-theta-SO(D)}} \label{S:proof1}

We begin with the proof of Thm. \ref{T:local-theta-SO(D)}.
\vskip 5pt

\noindent (i) Let $\pi = \tau_1^D \boxtimes  \tau_2^D$ be an irreducible representation of $\GSO(D)$ which is contained in a (unique)  irreducible representation $\tilde{\pi}$ of $\GO(D)$.
Then by Prop. \ref{P:super}, $\Theta(\tilde{\pi})$ is either zero or an irreducible representation of $\GSp_4(F)$. By Thm. \ref{T:roberts}, at least one of $\Theta(\tilde{\pi})$ or $\Theta(\tilde{\pi} \otimes \nu_0)$ is nonzero. Thus we conclude that
\[ \Theta(\tau_1^D \boxtimes \tau_2^D) = \Theta(\tau_2^D \boxtimes \tau_1^D)  \]
is nonzero and irreducible.  This proves (i). Moreover, Prop. \ref{P:super} implies that one has an injection
\[  \Pi(\GSO(D))/\sim\; \hookrightarrow \Pi(\GSp(W)). \]
 \vskip 5pt

\noindent (ii)  Now suppose that $\tau_1^D = \tau_2^D = \tau_D$. Then it is well-known that $\pi  =\tau_D \boxtimes \tau_D$ participates in the theta correspondence with $\GSp(W') = \GSp_2 = \GL_2$. Indeed, one has
  \[  \Theta_{D, W_1}(\pi) = JL(\tau_D). \]
  On the other hand, by Theorem \ref{T:Jacquet_module2} in the appendix, one has an $\GSO(D) \times Q(Z)$-equivariant surjection
  \[  R_{Q(Z)}(\Omega_{D, W}) \twoheadrightarrow \Omega_{D, W'}.  \]
By Frobenius reciprocity, one has a nonzero $\GSO(D) \times \GSp(W)$-equivariant map
\[   \Omega_{D, W'} \longrightarrow  \pi \boxtimes I_{Q(Z)}( 1, JL(\tau_D)). \]
In view of (i), we see that $\Theta(\pi)$ is either equal to $\pi_{gen}(JL(\tau_D))$ or
$\pi_{ng}(JL(\tau_D))$.  By Cor. \ref{C:Whit1}(i), we see that the former possibility cannot occur, so that
\[  \Theta(\tau_D \boxtimes \tau_D) = \pi_{ng}(JL(\tau_D)). \]
We have thus shown (ii). Moreover, from our enumeration of the non-supercuspidal representations of $\GSp(W)$ given in \S \ref{SS:non-super}, we see that the representations of $\GSp(W)$ obtained in this way are precisely the essentially tempered non-generic non-supercuspidal representations.
\vskip 5pt

\noindent (iii) On the other hand, suppose that $\tau_1^D \ne \tau_2^D$. Then by Thm. \ref{T:roberts}, $\Theta(\tau_1^D \boxtimes \tau_2^D)$ is supercuspidal. Cor. \ref{C:Whit1}(i) implies that it is non-generic. This proves (iii).
\vskip 5pt

\noindent (iv) By the above, we see that the map $\Theta$ gives an injection
\[  \Pi(\GSO(D))/\sim\; \hookrightarrow \Pi(\GSp(W))^{temp}_{ng}. \]
Moreover, as we remarked in the proof of (ii), any non-supercuspidal representation in
$\Pi(\GSp(W))^{temp}_{ng}$ lies in the image. Thus, to prove (iv),
 we need to show that every non-generic supercuspidal representation lies in the image.
\vskip 5pt

Suppose then that $\pi$ is a non-generic supercuspidal representation of $\GSp(W)$ which does not participate in the theta correspondence with $\GSO(D)$. By Thm. \ref{T:dichotomy}, $\pi$ must participate in the theta correspondence with $\GSO(V) = \GSO_{3,3}(F)$.  If this is the first occurrence of $\pi$ in the tower of split orthogonal similitude groups, then $\Theta_{W,D}(\pi)$ is a supercuspidal representation of $\GSO_{3,3}(F)$ which is necessarily generic. By  Cor. \ref{C:Whit2}, this implies that $\pi$ is itself generic, which is a contradiction. Thus, $\pi$ must participate in the theta correspondence with $\GSO(V_2) = \GSO_{2,2}(F)$, which is the lower step of the tower.  Moreover,
\[  \sigma = \Theta_{W, V_2}(\pi) \]
 is then a supercuspidal representation of $\GO(V_2)$,  since no supercuspidal representations of $\GSp(W)$ participate in the theta correspondence with $\GO(V_1) =\GO_{1,1}(F)$.
 Since $\sigma$ is necessarily generic,  we see by Cor. \ref{C:Whit1}(ii) that $\pi= \Theta_{V_2,W}(\sigma) $ must also be generic. This contradiction completes the proof of (iv).
  \vskip 5pt

\noindent Thm. \ref{T:local-theta-SO(D)} is proved.

\vskip 15pt

\section{\bf Proof of Theorem \ref{T:local-theta-GSO(2,2)}} \label{S:proof2}

 We now give the proof of Theorem \ref{T:local-theta-GSO(2,2)}.
The  key step is the computation of the normalized Jacquet modules
of the induced Weil representation $\Omega_{V_2,W}$,
where $V_2$ is the split four dimensional quadratic space.
Before coming to this computation,
we first introduce some more notations.
\vskip 5pt

Recall that $V_2 = X \oplus X^*$, where $X$ is a two dimensional isotropic space.
We can write
\[
 X = F \cdot u_1 \oplus F \cdot u_2
 \quad \text{and} \quad
 X^* = F \cdot v_1\oplus F \cdot v_2
\]
with $(u_i, v_j) = \delta_{ij}$.
Let $P(X)$ be the parabolic subgroup of $\GSO(V_2)$ stabilizing $X$
with Levi factor
\[
 M(X) \cong \GL(X) \times \GL_1.
\]
Let $J = F \cdot u_1$ be the isotropic line spanned by $u_1$ in $X$
and let $B(J)$ be the stabilizer of $J$ in $M(X)$;
it is also the stabilizer of the isotropic line spanned by $v_2$ in $X^*$.
With respect to the basis $\{ u_1, u_2 \}$ of $X$,
$B(J)$ is the group of upper triangular matrices
in $M(X) \cong \GL(X) \times \GL_1$.
We write
\[
 (t(a,b), \lambda) = \left(
 \begin{pmatrix}
  a & 0 \\
  0 & b
 \end{pmatrix},
 \lambda \right) \in B(J) \subset M(X).
\]
Similarly, recall that $W = Y^* \oplus Y$,
\[
 Y^* = F \cdot e_1 \oplus F \cdot e_2
 \quad \text{and} \quad
 Y = F \cdot f_1 \oplus F \cdot f_2
\]
with $\langle e_i, f_j \rangle = \delta_{ij}$.
The stabilizer of $Y$ in $\GSp(W)$ is
the Siegel parabolic subgroup $P(Y)$ with Levi factor
\[
 M(Y) \cong \GL(Y) \times \GL_1
\]
and the stabilizer of $Z = F \cdot f_1$ in $\GSp(W)$
is the Klingen parabolic subgroup $Q(Z)$ with Levi factor
\[
 L(Z) \cong \GL(Z) \times \GSp(W')
\]
where $W' = F \cdot e_2 \oplus F \cdot f_2$.
\vskip 5pt

\subsection{\bf Jacquet modules. }
With the above notations, we can now compute the Jacquet module of $\Omega_{V_2,W}$ relative to $P(X)$.  This follows the lines of \cite{K}  and we give the detailed computation in the appendix.

\vskip 5pt

\begin{Prop}
\label{P:jacquet-split}
The normalized Jacquet module $R_{P(X)}(\Omega_{V_2,W})$
of $\Omega_{V_2,W}$ along $P(X)$ has a natural three step filtration
as an $M(X) \times \GSp(W)$-module
whose successive quotients are described as follows:
\begin{enumerate}
\item
The top quotient is
\[
 C \cong S(F^{\times}).
\]
Here the action of $(m, \lambda) \in M(X) \cong \GL(X) \times \GL_1$
on $S(F^{\times})$ is given by
\[
 ((m, \lambda) \cdot f)(t)
 = |{\det}_X(m)|^{3/2} \cdot |\lambda|^{-3/2} \cdot f(\lambda \cdot t).
\]
\item
The middle subquotient is
\[
 B \cong
 I_{B(J) \times Q(Z)}(S(F^{\times}) \otimes S(F^{\times} \cdot v_2 \otimes f_1)).
\]
Here the action of $(t(a,b), \lambda) \in B(J)$
on $S(F^{\times}) \otimes S(F^{\times} \cdot v_2 \otimes f_1))$ is given by
\[
 ((t(a,b), \lambda) \cdot f)(t, x) =
 |a|  \cdot |\lambda|^{-3/2} \cdot f(\lambda \cdot t, b \cdot x),
\]
whereas the action of $(\alpha, g) \in L(Z) \cong \GL(Z) \times \GSp(W')$
is given by
\[
 ((\alpha, g) \cdot f)(t, x)
 =  |\nu_{W'}(g)|^{-1}
 \cdot f(\nu_{W'}(g) \cdot t, \alpha^{-1} \cdot \nu_{W'}(g) \cdot x).
\]
\item
Finally, the submodule is
\[
 A \cong I_{P(Y)} (S(F^{\times}) \otimes S(\Isom(X,Y)))
\]
where $\Isom(X,Y)$ is the set of isomorphisms from $X$ to $Y$
as vector spaces
(which is a torsor for $\GL(X)$ as well as for $\GL(Y)$).
Here the action of $(m, \lambda) \in M(X) \cong \GL(X) \times \GL_1$
on $S(F^{\times}) \otimes S(\Isom(X,Y))$ is given by
\[
 ((m, \lambda) \cdot f)(t, h)
 =   f(\lambda \cdot t, h \circ m),
\]
whereas the action of $(m', \lambda') \in M(Y) \cong \GL(Y) \times \GL_1$
is given by
\[
 ((m', \lambda') \cdot f)(t, h)
 =  f(\lambda' \cdot t, \lambda' \cdot m'^{-1} \circ h).
\]
\end{enumerate}
\end{Prop}
\vskip 5pt

\begin{proof}
This is Theorem \ref{T:Jacquet_module1}, with $m = 4$, $n=2$ and $t=2$, and with Remark \ref{R:kt} taken into account.
\end{proof}

\vskip 5pt

\begin{Cor}
\label{C:jacquet-split}
Let $\sigma = \pi(\chi_1, \chi_2) \boxtimes \tau$ be a representation of
$\GSO(V_2) \cong (\GL_2(F) \times \GL_2(F)) / F^{\times}$
such that $\tau$ is irreducible but $\pi(\chi_1, \chi_2)$ may be reducible,
so that $\omega_{\tau} = \chi_1 \chi_2$ and
\[
 \sigma = I_{P(X)}(\tau^{\vee} \cdot \chi_1, \chi_2).
\]
Then
\[
 \Hom_{\GSO(V_2)}(\Omega, \sigma)
 = \Hom_{M(X)}(R_{P(X)}(\Omega), \tau^{\vee} \cdot \chi_1 \boxtimes \chi_2).
\]
\begin{enumerate}
\item If $\chi_1 / \chi_2 \ne |-|^3$, then
\[
 \Hom_{M(X)}(C, \tau^{\vee} \cdot \chi_1 \boxtimes \chi_2) = 0.
\]
\item
If $R_B(\tau)$ does not have $\chi_1 |-|^{-1} \boxtimes \eta$
as a subquotient for any character $\eta$,
then
\[
 \Hom_{M(X)}(B, \tau^{\vee} \cdot \chi_1 \boxtimes \chi_2) = 0.
\]
\item
If the conditions in $\mathrm{(i)}$ and $\mathrm{(ii)}$ hold,
then
\[
 \Hom_{M(X)}(R_{P(X)}(\Omega), \tau^{\vee} \cdot \chi_1 \boxtimes \chi_2)
 \subset  \Hom_{M(X)}(A, \tau^{\vee} \cdot \chi_1 \boxtimes \chi_2)
 = I_{P(Y)}(\tau \cdot \chi_1^{-1}, \chi_1)^*,
\]
where $^\ast$ indicates the full linear dual.
\end{enumerate}
\end{Cor}
\begin{proof}
The first equality is just the Frobenius reciprocity.  
\vskip 5pt

\noindent (1) This follows from the fact that the center of $\GL(X)$ acts by $|\det_X|^3$ on $C$ and by the character $\chi_1/\chi_2$ on $\tau^{\vee} \cdot \chi_1 \boxtimes \chi_2$. 
\vskip 5pt

\noindent (2) We have:
\begin{align}
 &\Hom_{M(X)}(B, \tau^{\vee} \cdot \chi_1 \boxtimes \chi_2)  \notag \\
 =&\Hom_{M(X)}(  I_{B(J) \times Q(Z)}(S(F^{\times}) \otimes S(F^{\times} \cdot v_2 \otimes f_1)),
  \tau^{\vee} \cdot \chi_1 \boxtimes \chi_2) \notag \\
= &\Hom_{B(J)}(S(F^{\times}) \otimes S(F^{\times} \cdot v_2 \otimes f_1)), R_{\overline{B}}(\tau^{\vee} \cdot \chi_1)  \boxtimes \chi_2), \notag 
\end{align}
where $R_{\overline{B}}$ denotes the normalized Jacquet module relative to the opposite Borel subgroup $\overline{B}$.
From the action of $B(J)$ described in Proposition \ref{P:jacquet-split}(2), one sees that the element $t(a,1) \in B(J)$ acts on 
$S(F^{\times}) \otimes S(F^{\times} \cdot v_2 \otimes f_1))$ by the character $|a|$. It follows that 
if the last Hom space is nonzero, $ R_{\overline{B}}(\tau^{\vee} \cdot \chi_1)$ must contain $|-| \boxtimes \eta^{-1}$ as a subquotient, for some character $\eta$. This is equivalent to $R_B(\tau)$ having $\chi_1  |-|^{-1}  \boxtimes \eta$ as a subquotient. This proves (2).

 \vskip 5pt

\noindent (3) By (1) and (2), one obtains the inclusion in (3).  Now let $(\tau^{\vee} \cdot \chi_1 \boxtimes \chi_2)\otimes\Pi$ be the maximal $\tau^{\vee} \cdot \chi_1 \boxtimes \chi_2$-isotypic quotient of $A$. Then
\[
    \Hom_{M(X)}(A, \tau^{\vee} \cdot \chi_1 \boxtimes \chi_2)
    =\Hom_{M(X)}((\tau^{\vee} \cdot \chi_1 \boxtimes \chi_2)\otimes\Pi, \tau^{\vee} \cdot \chi_1 \boxtimes \chi_2)
    =\Hom_{\C}(\Pi,\C)=\Pi^\ast.
\]
But by \cite[Lemma 9.4]{GG}, $\Pi$ is of the form $I_{P(Y)}(\pi_0)$, where $\pi_0$ is such that $(\tau^{\vee} \cdot \chi_1 \boxtimes \chi_2)\otimes\pi_0$ is the maximal $\tau^{\vee} \cdot \chi_1 \boxtimes \chi_2$-isotypic quotient of $S(F^{\times}) \otimes S(\Isom(X,Y))$ which we view as a representation of $\GL(X)\times\GL(Y)$. Now by \cite[Lemma, Pg. 59]{MVW}, with the actions of $\GL_1$ taken into account, one can see that $\pi_0=\tau\cdot\chi_1^{-1}\boxtimes\chi_1$. This competes the proof.
\end{proof}

\vskip 5pt

\subsection{\bf Proof of Theorem \ref{T:local-theta-GSO(2,2)}.}
We are now ready to give the proof of Theorem \ref{T:local-theta-GSO(2,2)}.
In the following, we shall repeatedly use the following simple fact:
\vskip 5pt

\begin{itemize}
\item if $\sigma$ is an irreducible representation of $\GSO(V_2)$, then
\[  \Theta(\sigma)^* \cong \Hom_{\GSO(V_2)}(\Omega_{V_2,W},  \sigma). \]
\end{itemize}
\vskip 5pt

Let $\tau_1 \boxtimes \tau_2$ be an irreducible representation
of $\GSO(V_2) \cong (\GL_2(F) \times \GL_2(F)) / F^{\times}$.
As in the proof of Thm. \ref{T:local-theta-SO(D)}(i), it follows from Thm. \ref{T:roberts} that
\[
 \Theta(\tau_1 \boxtimes \tau_2) = \Theta(\tau_2 \boxtimes \tau_1) \ne 0.
\]
We now consider the various cases in Theorem \ref{T:local-theta-GSO(2,2)}
in turn.

\subsection*{\underline{Supercuspidal representations}}
\hfill

Suppose that $\tau_1 \boxtimes \tau_2$ is supercuspidal.
Then one knows that
$\Theta(\tau_1 \boxtimes \tau_2) = \theta(\tau_1 \boxtimes \tau_2)$
is non-zero and irreducible.
Moreover, if $\tau_1 \ne \tau_2$,
then the theta lift of $\tau_1 \boxtimes \tau_2$ to $\GSp(W') \cong \GL_2$
is zero and hence $\theta(\tau_1 \boxtimes\tau_2)$ is supercuspidal.

\vskip 5pt

On the other hand, if $\tau_1 = \tau_2 = \tau$,
then $\tau \boxtimes \tau$ participates
in the theta correspondence with $\GSp(W') \cong \GL_2$
and its big theta lift to $\GL_2$ is $\tau$.
As in the proof of Thm. \ref{T:local-theta-SO(D)}(ii),
there is a $\GSO(V_2) \times L(Z)$-equivariant surjective map
\[
 R_{Q(Z)}(\Omega_{V_2,W}) \longrightarrow \Omega_{V_2,W'}.
\]
By Frobenius reciprocity,
one has a non-zero $\GSO(V_2) \times \GSp(W)$-equivariant map
\[
 \Omega_{V_2,W} \longrightarrow
 (\tau \boxtimes \tau) \boxtimes I_{Q(Z)}(1, \tau).
\]
Thus, we see that
\[
 \Theta(\tau \boxtimes \tau) \hookrightarrow I_{Q(Z)}(1, \tau).
\]
We know that $I_{Q(Z)}(1, \tau)$ is the direct sum of two irreducible
constituents with a unique generic constituent $\pi_{\mathit{gen}}(\tau)$.
It follows from Cor. \ref{C:Whit1}(ii) that
\[
 \Theta(\tau \boxtimes \tau) = \pi_{\mathit{gen}}(\tau).
\]
\vskip 5pt

\subsection*{\underline{Discrete series representations}}
\hfill

Suppose that $\sigma = \st_{\chi} \boxtimes \tau$
where $\st_{\chi}$ is a twisted Steinberg representation
and $\tau$ is a discrete series representation
so that $\omega_{\tau} = \chi^2$.
Note that $\tau$ is either supercuspidal or equal to $\st_{\mu}$.
Then
\[
 \sigma \hookrightarrow
 \pi(\chi |-|^{1/2}, \chi |-|^{-1/2}) \boxtimes \tau
 = I_{P(X)}(\tau^{\vee} \cdot \chi |-|^{1/2}, \chi|-|^{-1/2}).
\]
We would like to apply Corollary \ref{C:jacquet-split} (3)
and so we need to verify that the conditions in
Corollary \ref{C:jacquet-split} (1) and (2) hold.
The condition in Corollary \ref{C:jacquet-split} (1) obviously holds,
and that in Corollary \ref{C:jacquet-split} (2) holds
when $\tau$ is supercuspidal.
If $\tau = \st_{\mu}$ is a twisted Steinberg representation
(so that $\chi^2 = \mu^2$),
then
\[
 R_B(\tau) = \mu |-|^{1/2} \boxtimes \mu |-|^{-1/2}
 \ne \chi |-|^{-1/2} \boxtimes \eta
\]
for any character $\eta$.
Hence the condition in Corollary \ref{C:jacquet-split} (2)
also holds when $\tau$ is a twisted Steinberg representation.
In particular, we conclude by Corollary \ref{C:jacquet-split} (3) that
\[
 I_{P(Y)}(\tau \cdot \chi^{-1} |-|^{-1/2}, \chi |-|^{1/2})
 \twoheadrightarrow \Theta(\sigma).
\]
By Lemma \ref{L:P(Y)},
the above induced representation is multiplicity free
and of length two with a unique irreducible quotient,
so that $\Theta(\sigma)$ is multiplicity free
and $\theta(\sigma)$ is irreducible.
Moreover,
\[
 \theta(\sigma) =
 \begin{cases}
  \St(\tau \cdot \chi^{-1}, \chi) & \text{if $\tau \ne \st_{\chi}$;} \\
  \pi_{\mathit{gen}}(\tau) & \text{if $\tau = \st_{\chi}$.}
\end{cases}
\]


\subsection*{\underline{Non-discrete series representations I}}
\hfill

Suppose that
\[
 \sigma \hookrightarrow
 \pi(\chi_1, \chi_2) \boxtimes \tau
 = I_{P(X)}(\tau^{\vee} \cdot \chi_1, \chi_2)
\]
where $\tau$ is a discrete series representation with
$\omega_{\tau} = \chi_1 \chi_2$ and
\[
 |\chi_1 / \chi_2| = |-|^{-s_0}
 \quad \text{and} \quad
 s_0 \ge 0.
\]
Again, we would like to apply Corollary \ref{C:jacquet-split} (3)
and so we need to verify the conditions there.
As before, the only issue is
the condition in Corollary \ref{C:jacquet-split} (2)
when $\tau = \st_{\chi}$ is a twisted Steinberg representation,
in which case
\[
 R_B(\tau) = \chi |-|^{1/2} \boxtimes \chi |-|^{-1/2}
\]
and we need to show that this is different from
$\chi_1|-|^{-1} \boxtimes \eta$ for any character $\eta$.
In other words, we need to show that
\[
 \chi / \chi_1 \ne |-|^{-3/2}.
\]
But observe that
\[
 |\chi|^2 = |\chi_1 \chi_2| = |\chi_1|^2 \cdot |\chi_2 / \chi_1|
 = |\chi_1|^2 \cdot |-|^{s_0},
\]
so that
\[
 |\chi / \chi_1| = |-|^{s_0/2} \ne |-|^{-3/2}.
\]
This verifies that the conditions
in Corollary \ref{C:jacquet-split} (1) and (2) hold,
so that we conclude that
\[
 I_{P(Y)}(\tau \cdot \chi_1^{-1}, \chi_1)
 \twoheadrightarrow \Theta(\sigma).
\]
Since the above induced representation is multiplicity free
with a unique irreducible quotient,
we conclude that $\Theta(\sigma)$ is multiplicity free and
$\theta(\sigma) = J_{P(Y)}(\tau \cdot \chi_1^{-1}, \chi_1)$ is irreducible.

\subsection*{\underline{Non-discrete series representations II}}
\hfill

Finally, we consider the case where
\[
 \sigma \hookrightarrow
 \pi(\chi_1, \chi_1') \boxtimes \pi(\chi_2, \chi_2')
\]
with $\chi_1 \chi_1' = \chi_2 \chi_2'$ and
\[
 |\chi_i / \chi_i'| = |-|^{-s_i}
 \quad \text{and} \quad
 s_1 \ge s_2 \ge 0.
\]
We consider two subcases:
\begin{itemize}
\item[$\mathrm{(a)}$]
$\chi_2 / \chi_2' \ne |-|^{-1}$;
in this case $\pi(\chi_2, \chi_2') = \pi(\chi_2', \chi_2)$ is irreducible and
\[
 \sigma \hookrightarrow
 I_{P(X)}(\pi(\chi_2, \chi_2')^{\vee} \cdot \chi_1, \chi_1').
\]
Again,
to apply Corollary \ref{C:jacquet-split} (3),
we need to verify the conditions there,
and in particular the condition in Corollary \ref{C:jacquet-split} (2).
We have
\[
 R_B(\pi(\chi_2, \chi_2'))
 = (\chi_2 \boxtimes \chi_2') \oplus (\chi_2' \boxtimes \chi_2)
\]
up to semisimplification and so we need to verify that
\[
 \chi_2 \ne \chi_1 |-|^{-1}
 \quad \text{and} \quad
 \chi_2' \ne \chi_1|-|^{-1}.
\]
To see these, we argue by contradiction.
If $\chi_2 = \chi_1|-|^{-1}$,
then $\chi_2' = \chi_1' |-|$, so that
\[
 |-|^{-s_2} = |\chi_2 / \chi_2'|
 = |\chi_1 / \chi_1'| \cdot |-|^{-2} = |-|^{-s_1-2}.
\]
This would give $s_2 = s_1+2$ which contradicts $s_1 \ge s_2$.
On the other hand,
if $\chi_2' = \chi_1 |-|^{-1}$,
then $\chi_2 = \chi_1' |-|$, so that
\[
 |-|^{s_2} = |\chi_2' / \chi_2|
 = |\chi_1 / \chi_1'| \cdot |-|^{-2} = |-|^{-s_1-2}.
\]
This would give $s_2 = -s_1-2 < 0$, which is a contradiction.
Thus, we may apply Corollary \ref{C:jacquet-split} (3) to conclude that
\[
 I_{P(Y)}(\pi(\chi_2', \chi_2) \cdot \chi_1^{-1}, \chi_1)
 = I_B(\chi'_2 / \chi_1, \chi_2 / \chi_1; \chi_1)
 \twoheadrightarrow \Theta(\sigma).
\]
This shows that $\Theta(\sigma)$ is multiplicity free
with a unique irreducible quotient
\[
 \theta(\sigma) = J_B(\chi_2'/\chi_1, \chi_2/\chi_1; \chi_1).
\]
\item[$\mathrm{(b)}$]
$\chi_2 / \chi_2' = |-|^{-1}$;
in this case, $\pi(\chi_2,\chi_2')$ is reducible
and has the one dimensional representation $\chi_2 |-|^{1/2}$
as its unique irreducible submodule.
Then
\[
 \sigma \hookrightarrow \pi(\chi_1, \chi_1') \boxtimes \chi_2 |-|^{1/2}
 = I_{P(X)}(\chi_1 \chi_2^{-1} |-|^{-1/2}, \chi_1').
\]
Applying Corollary \ref{C:jacquet-split} (3)
(we leave the verification of the conditions there to the reader),
we conclude that
\[
 I_{P(Y)}(\chi_1^{-1} \chi_2 |-|^{1/2}, \chi_1)
 \twoheadrightarrow \Theta(\sigma).
\]
Observe that
\[
 I_B(\chi_2' / \chi_1, \chi_2 / \chi_1; \chi_1)
 \twoheadrightarrow I_{P(Y)}(\chi_1^{-1} \chi_2 |-|^{1/2}, \chi_1)
\]
and the former induced representation is a standard module.
This shows that $\Theta(\sigma)$ is multiplicity free
with a unique irreducible quotient
\[
 \theta(\sigma) = J_B(\chi_2' / \chi_1, \chi_2 / \chi_1; \chi_1).
\]
\end{itemize}

This completes the proof of Theorem \ref{T:local-theta-GSO(2,2)}.


\vskip 15pt

\section{\bf Proof of Theorem \ref{T:local-theta}} \label{S:proof3}

In this section, we give the proof of Theorem \ref{T:local-theta}.

\vskip 5pt

\subsection{\bf Jacquet Modules.}
 The key step is the computation of the normalized Jacquet modules of the induced Weil representation
$\Omega$ with respect to $Q(Z)$, $P(Y)$ and $P(J)$. This is carried out in the appendix, following the lines of \cite{K}.

\vskip 10pt

\begin{Prop} \label{P:jacquet1}
Let $R_{P(J)}(\Omega)$ denote the normalized Jacquet module of
$\Omega$ along $P(J)$. Then we have
a short exact sequence of $M(J) \times \GSp(W)$-modules: \vskip 5pt
\[  \begin{CD}
0 @>>> A @>>> R_{P(A)}(\Omega) @>>> B @>>> 0 \end{CD}. \] Here,
\[  B \cong \Omega_{W, V_2}, \]
\vskip 5pt \noindent where $\Omega_{W,V_2}$ is the induced Weil
representation for $\GSp(W) \times \GSO(V_2)$, and \vskip 3pt
\[  A \cong   I_{Q(Z)} \left(  S(F^{\times})   \otimes \Omega_{W',V_2}  \right), \]
\vskip 5pt \noindent  where the action of $(GL(J) \times \GSO(V_2))\times
(GL(Z) \times GSp(W'))$ on $S(F^{\times})$ is given by: \vskip 5pt
\[  ((a, h), (b, g)) \cdot f (x) = f(b^{-1}\cdot x \cdot a  \cdot \lambda_{W'}(g)), \]
\vskip 5pt \noindent and $\Omega_{W',V_2}$ denotes the induced Weil
representation of $\GSp(W') \times \GSO(V_2)$.
\end{Prop}

\vskip 5pt

\begin{proof}
This is Theorem \ref{T:Jacquet_module1} with $m = 6$, $n=2$ and $t=1$, and with Remark \ref{R:kt} taken into account.
\end{proof}

 \vskip 10pt

 \begin{Prop} \label{P:jacquet2}
 Let $R_{Q(Z)}(\Omega)$ denote the normalized Jacquet module of $\Omega$ along $Q(Z)$.
 Then we have a short exact sequence of $\GSO(V) \times L(Z)$-modules
 \vskip 5pt
 \[ \begin{CD}
 0 @>>> A'@>>> R_{Q(Z)}(\Omega_D) @>>> B' @>>> 0.\end{CD}  \]
 Here,
 \[  B' \cong |det_Z| \cdot |\lambda_W|^{-1/2}  \boxtimes  \Omega_{W', V} \]
 \vskip 5pt
 \noindent where $\Omega_{W',V}$ is the induced Weil representation of $\GSp(W') \times \GSO(V)$ and
 \vskip 3pt
 \[  A' \cong  I_{P(J)} \left( S(F^{\times}) \otimes \Omega_{W', V_2}  \right), \]
\vskip 5pt

\noindent where the action of $(\GL(J) \times \GSO(V_2)) \times
(\GL(Z) \times \GSp(W'))$ on $S(F^{\times})$ is given by \vskip 5pt
\[  ((a,h), (b,g)) \cdot f(x) = f( a^{-1} \cdot \lambda_{W'}(g)^{-1} \cdot x \cdot b), \]
\vskip 5pt
\noindent and $\Omega_{W',V_2}$ is the induced Weil representation of
$\GSp(W') \times \GSO(V_2)$.
 \end{Prop}
 \vskip 10pt

\begin{proof}
This is Theorem \ref{T:Jacquet_module2}, with $m=6$, $n=2$ and $k=1$ and with Remark \ref{R:kt} taken into account.
\end{proof}
\vskip 10pt

\begin{Prop} \label{P:jacquet3}
Let $R_{P(Y)}(\Omega)$ denote the normalized Jacquet module of $\Omega$ along $P(Y)$.
Then as a representation of $M(Y) \times \GSO(V)$,
$R_{P(Y)}(\Omega)$ has a 3-step filtration whose successive quotients are given as follows:
\vskip 5pt

\begin{itemize}
\item[(i)] the top piece of the filtration is:
\[ A'' = S(F^{\times})\otimes |{\det}_Y|^{3/2} \cdot |\lambda_W|^{-3/2},\]
where $(a ,\lambda, h) \in GL(Y) \times \GL_1 \times \GSO(V)$ acts on $S(F^{\times})$ by
\[  (a,\lambda, h)\phi(t) = \phi(t \cdot \lambda \cdot \lambda(h)). \]
\vskip 5pt

\item[(ii)] the second piece in the filtration is:
\vskip 5pt
\[ B'' = I_{B \times P(J)}( S(F^{\times} \times F^{\times})) \]
\vskip 5pt
\noindent where the action of the diagonal torus in $B$ on $S(F^{\times}\times F^{\times})$ is given by
\vskip 5pt
\[  \left( \left( \begin{array}{cc}
a & \\
& d \end{array} \right) \cdot \phi \right)(\lambda, t) = |a| \cdot \phi(\lambda, td). \]
\vskip 5pt

\item[(iii)]  the bottom piece of the filtration is:
\vskip 5pt
\[  C'' = I_{Q(E)}(S(F^{\times}) \otimes S(\GL_2)), \]
\vskip 5pt
\noindent where the action of  $(\GL(Y) \times \GL_1) \times (\GL(E) \times \GSO(V_1))$ on
$S(F^{\times}) \otimes S(\GL_2)$ is given by:
\vskip 5pt
\[  (a, \lambda; b, h)\phi(t, g) = \phi( t \cdot \lambda \cdot \lambda_{V_1}(h), b^{-1} g a \cdot
\lambda_{V_1}(h)). \]
  \end{itemize}
\end{Prop}
\vskip 5pt

\begin{proof}
This is Theorem \ref{T:Jacquet_module2}, with $m=6$, $n=2$ and $k=2$, and with Remark \ref{R:kt} taken into account.
\end{proof}
\vskip 15pt

\subsection{\bf Consequences.}
Applying Frobenius reciprocity as well as Props. \ref{P:jacquet1},
\ref{P:jacquet2} and \ref{P:jacquet3}, we obtain the following 3 propositions as consequences. Since the proofs of these 3 propositions are similar, we shall only give the proof of Prop. \ref{P:Hom2}.

\vskip 5pt

\begin{Prop}  \label{P:Hom1}
Assume that $\chi \ne |-|$. Then as a representation of
$\GSO(V)$,
\[  \Hom_{\GSp(W)}(\Omega, I_{Q(Z)}(\chi, \tau)) = I_{P(J)}(\chi^{-1}, (\tau \cdot \chi) \boxtimes (\tau \cdot \chi))^* \quad \text{(full linear dual)}. \]
\end{Prop}

\vskip 10pt

\begin{Prop}  \label{P:Hom2}
Suppose that $\tau$ is a discrete series representation of
$\GL(Y)$ and $\omega_{\tau} \ne |-|^3$. Then as a representation of $\GSO(V)$,
\[  \Hom_{\GSp(W)}(\Omega, I_{P(Y)}(\tau, \chi)) \hookrightarrow I_{Q(E)}(\tau^{\vee}, (\chi \omega_{\tau})\circ \lambda_{V_1})^*. \]
Further, if $\tau$ is supercuspidal,  then
\[  \Hom_{\GSp(W)}(\Omega, I_{P(Y)}(\tau, \chi)) =  I_{Q(E)}(\tau^{\vee}, (\chi \omega_{\tau})\circ \lambda_{V_1})^*. \]
\end{Prop}
\vskip 10pt

\begin{Prop} \label{P:Hom3}
(i)  Consider the space
\[  \Hom_{\GSO(V)}(\Omega, I_{P(J)}(\chi, \tau_1 \boxtimes \tau_2) ) \]
as a representation of $\GSp(W)$. Then we have: \vskip 5pt

\begin{itemize}
\item[(a)] If $\chi \ne 1$, then
\[   \Hom_{\GSO(V)}(\Omega, I_{P(J)}(\chi, \tau_1 \boxtimes \tau_2))
 = 0 \]
unless
\[  \tau_1 = \tau_2 = \tau, \]
in which case
\[ \Hom_{\GSO(V)}(\Omega, I_{P(J)}(\chi, \tau \boxtimes \tau)) =  I_{Q(Z)}\left( \chi^{-1}, \tau \otimes  \chi \right)^*  . \]
\vskip 5pt

\item[(b)] If $\chi=1$ but $\tau_1 \ne \tau_2$, then
\[  \Hom_{\GSO(V)}(\Omega, I_{P(J)}(\chi, \tau_1 \boxtimes \tau_2) )
= \Theta_{W,V_2}(\tau_1 \boxtimes  \tau_2)^*, \] where
$\Theta_{W,V_2}(\tau_1 \boxtimes \tau_2)$ denotes the big theta lift
of $\tau_1 \boxtimes \tau_2$ from $\GSO(V_2)$ to $\GSp(W)$. \vskip 5pt

\item[(c)] If $\chi = 1$ and $\tau_1 = \tau_2 = \tau$, then we have an exact sequence:
\[  \begin{CD} 0 @>>> \Theta_{W,V_2}(\tau \boxtimes \tau)^*@>>>
\Hom_{\GSO(V)}(\Omega, I_{P(J)}(\chi, \tau \boxtimes \tau)) @>>>
 \left( I_{Q(Z)}(1, \tau) \right)^* \end{CD} \]
\end{itemize}
\vskip 10pt
 \end{Prop}
 \vskip 10pt

  \begin{proof}[Proof of Prop. \ref{P:Hom2}]
  By Frobenius reciprocity, we have:
\[  \Hom_{\GSp(W)}(\Omega, I_{P(Y)}(\tau, \chi)) = \Hom_{M(Y)}(R_{P(Y)}(\Omega), \tau \boxtimes \chi)). \]
The 3-step filtration of $R_{P(Y)}(\Omega)$ thus induces one on this Hom space.
\vskip 5pt

 For $\tau$ as in the proposition, we see that
 \[  \Hom_{M(Y)}(A'',\tau \boxtimes \chi) = 0. \]
 This is because the center of $GL(Y)$ acts by the character $|-|^3$ on $A''$ and by
 $\omega_{\tau}$ on $\tau \boxtimes \chi$, and by our assumption, these two characters
 are different.
 \vskip 5pt

 We claim now that
 \[  \Hom_{M(Y)}(B'', \tau \boxtimes \chi) = 0 \]
as well. This is clear if $\tau$ is supercuspidal. On the other hand, suppose that $\tau= st_{\mu}$ is a twisted Steinberg representation. If $\Hom_{M(Y)}(B'', \tau\boxtimes \chi) \ne 0$, then one deduces that
\[  \Hom_{GL(Y)}(I_B(  |-| \boxtimes V), st_{\mu})) \ne 0, \]
where $|-| \boxtimes V$ is a representation of the diagonal torus $\GL_1 \times \GL_1$.
By Frobenius reciprocity, this implies that
\[  \Hom_{\GL_1 \times \GL_1}(R_B(st_{\mu^{-1}}) ,  |-|^{-1}  \boxtimes V^{\vee}) \ne 0. \]
 But $R_B(st_{\mu^{-1}}) = |-|^{1/2} \mu^{-1} \boxtimes |-|^{-1/2} \mu^{-1}$. This would imply that
 $\mu = |-|^{3/2}$. However, this is ruled out by the assumption of the proposition.
 \vskip 5pt

 Hence, we have shown that
 \[   \Hom_{\GSp(W)}(\Omega, I_{P(Y)}(\tau, \chi)) \hookrightarrow
 \Hom_{M(Y)}(C'', \tau \boxtimes \chi). \]
Now by arguing in the same way as the proof of Cor. \ref{C:jacquet-split}(3)
\[   \Hom_{M(Y)}(C'', \tau \boxtimes \chi) = I_{Q(E)}(\tau^{\vee}, (\chi \omega_{\tau})\circ \lambda_{V_1})^*. \]
Suppose further that $\tau$ is supercuspidal.  Since the representations $A''$ and  $B''$ do not contain any supercuspidal constituents and hence belong to different Bernstein components of $\GSp_4$,
one has
 \[   \Hom_{\GSp(W)}(\Omega, I_{P(Y)}(\tau, \chi)) =
 \Hom_{M(Y)}(C'', \tau \boxtimes \chi). \]
 The proposition is proved.
 \end{proof}
\vskip 15pt

\subsection{\bf Proof of Thm. \ref{T:local-theta}.}
Now we can prove Thm. \ref{T:local-theta}. In the following, we shall repeatedly use the following two simple facts:
\vskip 5pt

\begin{itemize}
\item[(a)] if $\pi$ is an irreducible representation of $\GSp(W)$, then
\[  \Theta(\pi)^* \cong \Hom_{\GSp(W)}(\Omega,  \pi). \]
\vskip 5pt

\item[(b)] if $\Pi$ is an irreducible representation of $\GSO(V)$ such that
\[  \Pi^{\vee}  \hookrightarrow \Hom_{\GSp(W)}(\Omega, \Sigma), \]
where $\Sigma$ is not necessarily irreducible (typically, $\Sigma$ is a principal series representation), then there is a nonzero equivariant map
\[  \Omega \longrightarrow \Pi \boxtimes \Sigma. \]
In particular, $\Theta(\Pi) \ne 0$.
The analogous result with the roles of $\GSp(W)$ and $\GSO(V)$ exchanged also holds.
\end{itemize}
\vskip 5pt

By [KR, Thm. 3.8] (i.e. Thm. \ref{T:KR}(i)), we know that if $\pi \in \Pi(\GSp_4)_{ng}^{temp}$, then $\Theta(\pi) =0$. Parts (i) and (ii) of the theorem will follow if we can determine $\theta(\pi)$ for
$\pi \notin \Pi(\GSp_4)_{ng}^{temp}$, i.e. if we can demonstrate parts (iii), (iv), (v) and (vi).
We consider the different cases separately.
\vskip 15pt

\noindent{\bf \underline{Supercuspidal Representations}} \vskip 5pt

If $\pi$ is a supercuspidal representation of $\GSp_4(F)$ and $\pi \notin \Pi(\GSp_4)_{ng}^{temp}$,
then $\pi$ is generic and we know that $\Theta(\pi)$ is a nonzero irreducible representation of
$\GSO(V)$.  It is supercuspidal unless the theta lift of $\pi$ to $\GSO_{2,2}(F) $ is nonzero, in which case its theta lift to $\GSO_{2,2}(F) $ is also supercuspidal.
Suppose that
\[ \pi = \Theta(\tau_1 \boxtimes \tau_2) = \Theta(\tau_2 \boxtimes \tau_1), \]
so that $\tau_1$ and $\tau_2$ are supercuspidal representations of $\GL_2(F)$ with the same central character. Then it follows by Prop. \ref{P:jacquet1} that
\[  \Hom_{\GSp(W) \times \GSO(V)}(\Omega, \pi \boxtimes I_{P(J)}(1, \tau_1 \boxtimes \tau_2)) \ne 0. \]
This shows that
\[  \Theta(\pi) = I_{P(J)}(1,\tau_1 \boxtimes \tau_2) = I_P(\tau_1, \tau_2) \boxtimes \omega_{\tau_1}. \]
This proves Thm. \ref{T:local-theta}(iii).
\vskip 15pt

\noindent{\bf \underline{The Generalized Steinberg Representation $St(\chi, \tau)$}}
\vskip 5pt

Now we consider the first family of generalized Steinberg representations, so that
$\pi = St(\chi, \tau)$ is as in \ref{SSS:DS}(a) with $\chi$ a non-trivial quadratic character and $\tau$ supercuspidal so that $\tau \otimes \chi = \tau$.  Recall that we need to show:
\[   \theta(St(\chi,\tau)) =   St(\tau) \boxtimes \omega_{\tau}\chi. \]
 Since
\[  St(\tau) \boxtimes \omega_{\tau}\chi \hookrightarrow I_{P(J)}(\chi |-|, (\tau \cdot \chi |-|^{-1/2})
\boxtimes \tau |-|^{-1/2}), \]
we deduce by the fact (a) above and Prop. \ref{P:Hom3}(i)(a) that
\[  \Theta(St(\tau) \boxtimes \omega_{\tau}\chi)^* \hookrightarrow \Hom_{\GSO(V)}(\Omega,
I_{P(J)}(\chi |-|, (\tau \cdot \chi|-|^{-1/2}) \boxtimes \tau|-|^{-1/2})), \]
which vanishes unless $\tau \otimes \chi \cong \tau$,  in which case
one has:
\[  I_{Q(Z)}(\chi^{-1} |-|^{-1},  \tau|-|^{1/2}) \twoheadrightarrow  \Theta(St(\tau) \boxtimes \omega_{\tau}\chi).
\]
Recall that the above induced representation has $St(\chi, \tau)$ as its
unique irreducible quotient (since $\chi \ne 1$). From this, we conclude
that: \vskip 5pt
\begin{itemize}
\item  $\theta(St(\tau) \boxtimes \omega_{\tau}\chi)  \subseteq St(\chi, \tau)$;
\vskip 5pt

\item $\theta(St(\chi, \tau)) \ne 0$.
\end{itemize}
 \vskip 10pt

On the other hand, since $\chi \ne 1$, one may apply  Prop.
\ref{P:Hom1} to $I_{Q(Z)}(\chi |-|, \tau|-|^{-1/2})$
and arguing as above, one obtains:
\vskip 5pt
\begin{itemize}
 \item $\theta(St(\chi, \tau)) \subseteq St(\tau) \boxtimes \omega_{\tau}\chi$;

\item  $\theta(St(\tau) \boxtimes \omega_{\tau}\chi) \ne 0$.
\end{itemize}
\vskip 5pt

\noindent  Hence, we have shown that \vskip 5pt
\[   \begin{cases}
\theta(St(\tau) \boxtimes \omega_{\tau}\chi) =  St(\chi, \tau); \\
\theta(St(\chi, \tau)) =  St(\tau) \boxtimes \omega_{\tau}\chi. \end{cases} \] \vskip 10pt

\vskip 15pt

\noindent{\bf \underline{The Generalized Steinberg  Representation $St(\tau,\mu)$}} \vskip 5pt

Now we consider the second family of generalized Steinberg representations, so that
 $\pi = St(\tau,\mu)$ as in \ref{SSS:DS}(b), with $\tau \ne st$ a discrete series representation
 of $\PGL(Y)$.  Recall that we need to show:
\[  \theta(St(\tau,\mu))  = I_P(\tau \otimes \mu , st \otimes \mu)  \boxtimes \mu^2. \]
\vskip 5pt

Since
\[  St(\tau,\mu) \hookrightarrow I_{P(Y)}(\tau|-|^{1/2},\mu|-|^{-1/2}), \]
Prop. \ref{P:Hom2} implies that
\[  I_{Q(E)}(\tau |-|^{-1/2}, (\mu|-|^{1/2}) \circ \lambda_{V_1}) \twoheadrightarrow
\Theta(St(\tau, \mu))^*. \]
Now note that as a representation of $\GL_4(F) \times \GL_1(F)$,
\[  I_{Q(E)}(\tau |-|^{-1/2}, (\mu|-|^{1/2}) \circ \lambda_{V_1}) = \mu \cdot
I_Q(|-|^{-1/2},\tau, |-|^{1/2}) \boxtimes \mu^2, \]
and the latter representation has a unique irreducible quotient isomorphic to
$\mu \cdot I_P(\tau , st) \boxtimes \mu^2$. Hence, we have shown that
\[  \theta(St(\tau, \mu)) \subseteq  I_P(\tau\otimes \mu , st\otimes  \mu) \boxtimes \mu^2. \]
On the other hand, by Thm. \ref{T:local-theta-GSO(2,2)}, one knows that the representation
$St(\tau, \mu)$ participates in the theta correspondence with $\GSO(V_2)$: it is the theta lift of
$(\tau \otimes \mu) \boxtimes (st \otimes \mu)$. Hence it follows by Prop. \ref{P:jacquet1} that $ \Theta(St(\tau, \mu))  \ne 0$ and so
\[    \theta(St(\tau, \mu)) =  I_P(\tau\otimes \mu , st\otimes \mu) \boxtimes \mu^2,\]
as desired.

\vskip 15pt
\noindent{\bf \underline{Twisted Steinberg Representations}}
\vskip 5pt

Now consider the twisted Steinberg representation $St_{\PGSp_4} \otimes \chi$. Since
\[  St_{\PGSp_4} \otimes \chi \hookrightarrow I_{Q(Z)}(|-|^2, st_{\chi} |-|^{-1}), \]
and
\[  St_{\chi} \boxtimes \chi^2  \hookrightarrow I_{P(J)}(|-|^2, st_{\chi} |-|^{-1} \boxtimes st_{\chi} |-|^{-1}), \]
we may apply Props. \ref{P:Hom1} and \ref{P:Hom3}
 to conclude that
\[  \theta(St_{\chi} \boxtimes \chi^2) = St_{\PGSp_4} \otimes \chi \]
and
\[   \theta( St_{\PGSp_4} \otimes \chi ) = St_{\chi} \boxtimes \chi^2. \]
Since the arguments are similar to the above, we omit the details.

\vskip 15pt

\noindent{\bf \underline{Non-Discrete Series Representations}} \vskip 5pt

Finally, we come to the non-discrete series representations in part (vi) of Thm. \ref{T:local-theta}.
We will consider the three cases (a), (b) and (c) separately.
\vskip 10pt

\noindent{\bf (a)} Suppose that
\[   \pi \hookrightarrow I_{Q(Z)}(\chi, \tau) \]
as in \ref{SSS:NDS}(a), so that $|\chi| = |-|^{-s}$ with $s \geq 0$. Recall that we need to show:
\[  \theta(\pi) = J_P(\tau, \tau \cdot \chi) \boxtimes \omega_{\tau}\chi. \]
By Prop. \ref{P:Hom1}, we deduce that $\Theta(\pi)$ is a quotient of
\[  I_{P(J)}(\chi^{-1}, (\tau \cdot \chi) \boxtimes  (\tau \cdot \chi))
 = I_P(\tau, \tau \cdot\chi) \boxtimes (\omega_{\tau} \cdot \chi).  \]
But this induced representation has a unique irreducible quotient
$J_P(\tau, \tau \cdot \chi) \boxtimes \omega_{\tau}\chi$,
since it is a standard module. This shows that
\begin{itemize}
\item $\theta(\pi) \subseteq  J_P(\tau, \tau \cdot \chi) \boxtimes \omega_{\tau}\chi)$;
\vskip 5pt
\item $\theta(J_P(\tau, \tau \cdot \chi) \boxtimes \omega_{\tau}\chi)) \ne 0$.
\end{itemize}
\vskip 10pt

On the other hand, since
\[  J_P(\tau, \tau \cdot \chi) \boxtimes \omega_{\tau}\chi \hookrightarrow
I_P(\tau \cdot \chi, \tau) \boxtimes \omega_{\tau}\chi \cong
I_{P(J)}(\chi, \tau \boxtimes \tau),  \]
\vskip 5pt
\noindent we may apply Prop. \ref{P:Hom3}(a) and (c) to conclude that
\vskip 5pt

\begin{itemize}
\item  $\theta(J_P(\tau, \tau \cdot \chi) \boxtimes \omega_{\tau}\chi) \subseteq \pi$;
\vskip 5pt

\item $\Theta(\pi) \ne 0$.
\end{itemize}
\vskip 5pt

\noindent From this, we conclude that
\[ \theta(\pi) =  J_P(\tau, \tau \cdot \chi) \boxtimes \omega_{\tau}\chi, \]
as desired.
\vskip 15pt

\noindent{\bf (b)} Suppose now that
\[  \pi \hookrightarrow I_{P(Y)}(\tau, \chi) \]
as in \ref{SSS:NDS}(b), so that $\tau$ is a twist of a discrete series representation with $|\omega_{\tau}| = |-|^{-2s}$ with $s \geq 0$. Applying Prop. \ref{P:Hom2}, one deduces that $\Theta(\pi)$ is a quotient of
\[  I_{Q(E)}(\tau^{\vee},  (\omega_{\tau} \chi) \circ \lambda_{V_1})) \cong
I_Q(1, \tau, \omega_{\tau}) \cdot \chi \boxtimes \chi^2 \omega_{\tau}. \]
This induced representation has a unique irreducible quotient
$J_Q(1, \tau, \omega_{\tau}) \cdot\chi \boxtimes \chi^2 \omega_{\tau}$. Thus,
one sees that
\vskip 5pt
\[  \theta(\pi) \subseteq J_Q(1, \tau, \omega_{\tau}) \cdot\chi \boxtimes \chi^2
\omega_{\tau}. \]
\vskip 5pt
On the other hand, since $\pi$ participates in the theta correspondence with $\GSO(V_2)$ by
Thm. \ref{T:local-theta-GSO(2,2)}, one knows that $\theta(\pi) \ne 0$. Hence, one has
\[  \theta(\pi) = J_Q(1, \tau, \omega_{\tau}) \cdot\chi \boxtimes \chi^2
\omega_{\tau}, \]
 as desired.
 \vskip 15pt

 \noindent{\bf (c)} Suppose first that
\[  \pi \hookrightarrow I_B(\chi_1, \chi_2; \chi) \]
as in \ref{SSS:NDS}(c)(i), so that $\chi_2 \ne |-|^{-1}$.
Thus, $|\chi_1| = |-|^{-s_1}$ and $|\chi_2| = |-|^{-s_2}$ with
$s_1 \geq s_2 \geq  0$, but $\chi_2  \ne |-|^{-1}$. In this case, we have
\[  \pi \hookrightarrow   I_{Q(Z)}(\chi_1, \pi(\chi \chi_2, \chi)) \]
as the unique irreducible submodule. Applying Prop. \ref{P:Hom1}, one deduces that as a representation of $\GL_4(F) \times \GL_1(F)$, $\Theta(\pi)$ is a quotient of
\begin{align}
&I_{P(J)}(\chi_1^{-1},  \pi(\chi\chi_1\chi_2,\chi\chi_1) \boxtimes
  \pi(\chi\chi_1\chi_2,\chi\chi_1)) \notag \\
=&\chi \cdot I_P(\pi(\chi_2, 1), \pi(\chi_1\chi_2, \chi_1)) \boxtimes \chi^2\chi_1\chi_2 \notag \\
=&\chi \cdot I_{B_0}(1,\chi_2, \chi_1, \chi_1\chi_2)  \boxtimes \chi^2\chi_1\chi_2. \notag
\end{align}
 \vskip 5pt
\noindent  This induced representation has a unique irreducible quotient
\[ \Pi = \chi \cdot
J_{B_0}(1, \chi_2, \chi_1, \chi_1\chi_2)  \boxtimes \chi^2\chi_1\chi_2.\]
Hence, we have:
\[  \theta(\pi) \subseteq \Pi \quad \text{and} \quad \Theta(\Pi) \ne 0. \]
\vskip 10pt

In fact, if $I_B(\chi_1,\chi_2;\chi)$ is irreducible, so that it is equal to $\pi$, then
we would have
\[  \Theta(\pi)  = I_{B_0}(1,\chi_2, \chi_1, \chi_1\chi_2)  \boxtimes \chi^2\chi_1\chi_2.\]
This is the case, for example, when $\chi_1$ (and hence $\chi_2$) is unitary. In that case, we have the desired identity $\theta(\pi) = \Pi$.

\vskip 10pt

To prove the desired identity in general, we may thus assume that $\chi_1 \ne 1$.
Then one may apply Prop. \ref{P:Hom3}(a) to the representation
\[  I_{P(J)}(\chi_1, \pi(\chi\chi_2, \chi) \boxtimes  \pi(\chi\chi_2, \chi)) \]
 which contains $\Pi$ as its unique irreducible submodule. By Prop. \ref{P:Hom3}(a), one has
\[ \Hom_{\GSO(V)}(\Omega, I_{P(J)}(\chi_1, \pi(\chi\chi_2, \chi) \boxtimes
 \pi(\chi\chi_2, \chi)))  = I_{Q(Z)}(\chi_1^{-1},  \pi(\chi\chi_1\chi_2,\chi\chi_1) \boxtimes
  \pi(\chi\chi_1\chi_2,\chi\chi_1))^*. \]
 It follows from this that
 \[  \theta(\Pi) \subseteq \pi \quad \text{and} \quad \Theta(\pi) \ne 0. \]
 Hence we have the desired equality $\theta(\pi) = \Pi$ in general.
 \vskip 10pt

Finally, we need to treat the case when
\[   \pi \hookrightarrow   I_{Q(Z)}(\chi_1, \chi \cdot  |-|^{-1/2})  \]
as the unique irreducible submodule, as given in \ref{SSS:NDS}(c)(ii), so that
$|\chi_1| = |-|^{-s_1}$ with $s_1 \geq 1$.
Application of Prop. \ref{P:Hom1} shows that $\Theta(\pi)$ is a quotient of
\[  I_{P(J)}(\chi_1^{-1}, \chi_1\chi|-|^{-1/2} \boxtimes \chi_1\chi|-|^{-1/2}) =
\chi|-|^{-1/2}\cdot I_P(1, \chi_1) \boxtimes \chi_1\chi^2|-|^{-1}.
 \]
But $ I_P(1, \chi_1)$ is a quotient of $I_{B_0}(|-|^{1/2}, |-|^{-1/2}, \chi_1|-|^{1/2},\chi_1|-|^{-1/2})$ which is a standard module. This shows that
\[  \theta(\pi) \subset \Pi :=  \chi|-|^{-1/2} J_{B_0}(|-|^{1/2}, |-|^{-1/2}, \chi_1|-|^{1/2},\chi_1|-|^{-1/2}) \boxtimes \chi_1\chi^2 |-|^{-1}. \]
On the other hand, since
\[  \Pi \hookrightarrow
I_{P(J)}(\chi_1,    \chi|-|^{-1/2} \boxtimes   \chi |-|^{-1/2}) \]
and $\chi_1 \ne 1$, one may apply Prop. \ref{P:Hom3}(a) to conclude that $\theta(\pi) = \Pi$.
 This completes the proof of Thm. \ref{T:local-theta}.
\qed

\vskip 15pt

\section{\bf Some Corollaries}  \label{S:corollaries}

We note some corollaries of our explicit determination of theta correspondences.
\vskip 5pt

The following is an immediate consequence of Thms. \ref{T:local-theta-SO(D)}, \ref{T:local-theta-GSO(2,2)} and \ref{T:local-theta}.
\vskip 5pt

\begin{Cor}  \label{C:dichotomy}
 The dichotomy statement of Theorem \ref{T:dichotomy} holds for all irreducible representations of $\GSp_4(F)$.
\end{Cor}

\vskip 10pt

The following result was stated in [GT1, Thm. 5.6(iv)]:
\vskip 5pt

 \begin{Cor} \label{C:comp}
Let $\pi$ be an irreducible representation of $\GSp_4(F)$ with central character $\mu$ and suppose that
$\pi$ participates in the theta correspondence with $\GSO(V_2)$, so that
\[  \pi = \theta(\tau_1 \boxtimes \tau_2) = \theta(\tau_2 \boxtimes \tau_1). \]
Let $\Pi \boxtimes \mu$ be the small theta lift of $\pi$ to $\GSO(V)$, with $\Pi$ a representation of
$\GL_4(F)$. Then we have the following equality of $L$-parameters for $\GL_4 \times \GL_1$:
\[  \phi_{\Pi} \times \mu  = (\phi_{\tau_1} \oplus \phi_{\tau_2}) \times \mu. \]
\end{Cor}

\begin{proof}
Suppose first that $\pi$ is a discrete series representation so that $\tau_i$ is also discrete series.
By Thm. \ref{T:local-theta}, we know that $\theta(\pi)$ is irreducible. By Prop. \ref{P:jacquet1} and Frobenius reciprocity, we see that there is a nonzero map
\[  \Omega \longrightarrow \pi \boxtimes I_P(\tau_1,\tau_2). \]
Since $I_P(\tau_1,\tau_2)$ is irreducible, we see that $\Pi = I_P(\tau_1,\tau_2)$ and we have the desired equality of $L$-parameters.
\vskip 5pt

If $\pi$ is not a discrete series representation, then $\pi$ is of the type occurring in Thm. \ref{T:local-theta-GSO(2,2)}(v) or (vi). On the other hand, we can determine $\Pi$ from Thm. \ref{T:local-theta}(vi)(b) or (c). Let us illustrate this for the case when $\pi$ is as in Thm. \ref{T:local-theta-GSO(2,2)}(vi), so that
\[  \pi = \theta(\tau_1 \boxtimes \tau_2) = J_B(\chi'_2/ \chi_1, \chi_2/\chi_1;\chi_1), \]
and $\tau_i \hookrightarrow \pi(\chi_i, \chi_i')$ is non-discrete series. Then
\[  \phi_{\tau_1} \oplus \phi_{\tau_2} = \chi_1 \oplus \chi_1' \oplus \chi_2 \oplus \chi_2'. \]
On the other hand, since
\[  J_B(\chi'_2/ \chi_1, \chi_2/\chi_1;\chi_1) \hookrightarrow I_B(\chi_1/\chi_2', \chi_1/\chi_2; \chi_1'),\]
it follows by Thm. \ref{T:local-theta}(vi)(c) that
\[  \Pi = J_B(\chi_1', \chi_2', \chi_2, \chi_1) \]
and so
\[  \phi_{\Pi} = \chi_1 \oplus \chi_1' \oplus \chi_2 \oplus \chi_2' \]
as well. This proves the corollary.
\end{proof}
\vskip 10pt

Now we consider the theta lifts of unramified representations. The dual group of $\GSp_4$ is $\GSp_4(\C)$ while that of $\GSO(V) = \GSO_{3,3}$ is a subgroup of $\GL_4(\C) \times \GL_1(\C)$.
There is a natural embedding of dual groups
\[  \iota: \GSp_4(\C) \hookrightarrow \GSO(V)^{\vee} \subset \GL_4(\C) \times \GL_1(\C), \]
where the first projection is given by the tautological embedding and the second projection is given by the similitude character.
The following corollary gives the lifting of unramified representations in terms of their Satake parameters. It was stated in [GT1, Prop. 3.4].

\vskip 5pt

\begin{Cor}  \label{C:unram}
 Let $\pi = \pi(s)$ be an unramified representation of $\GSp(W)$ corresponding to the semisimple class $s \in \GSp_4(\C)$. Then $\theta(\pi(s))$ is the unramified representation of $\GSO(V)$ corresponding to the semisimple class $\iota(s) \in \GL_4(\C) \times \GL_1(\C)$.
\end{Cor}

\begin{proof}
If $\pi(s) \hookrightarrow I_B(\chi_1,\chi_2;\chi)$ and we set $\chi_i(\varpi) = t_i$ and $\chi(\varpi) = \nu$,  then
\[  s = \left( \begin{array}{cccc}
\nu t_1 t_2 & & & \\
& \nu  t_1 & & \\
& & \nu t_2 & \\
&  &  &  \nu  \end{array} \right). \]
The unramified representation of $\GL_4(F) \times \GL_1(F)$ with Satake parameter $\iota(s)$ is the unramified constituent of
\[   \chi \cdot I_B(1,\chi_2 ,\chi_1, \chi_1\chi_2) \boxtimes \chi^2 \chi_1 \chi_2. \]
The corollary follows by Thm. \ref{T:local-theta}(vi)(c).
\end{proof}
\vskip 10pt

\section{\bf $L$-parameters and Genericity.}

 In [GT1], using Theorem \ref{T:summary} in the introduction, the local Langlands correspondence for $\GSp_4$ was obtained from that for $\GL_2$ and $\GL_4$.
Given an irreducible representation $\pi$ of $\GSp_4(F)$, we briefly recall how one obtains its $L$-parameter:
\[  \phi_{\pi} : WD_F = W_F \times \SL_2(\C) \longrightarrow \GSp_4(\C) \]
where $WD_F$ (resp.  $W_F$) denotes the Weil-Deligne (resp. Weil) group of $F$.

\vskip 5pt
Firstly, we note that
in [GT1, Lemma 6.1], it was shown that the embedding
\[  \iota: \GSp_4(\C) \hookrightarrow  \GL_4(\C) \times \GL_1(\C) \]
induces an injection
\[  \Phi(\GSp_4) \hookrightarrow \Phi(\GL_4) \times \Phi(\GL_1) \]
where $\Pi(G)$ denotes the set of $L$-parameters of $G$.
 In particular,  $\phi_{\pi}$ can be specified by describing it as a 4-dimensional representation of $WD_F$ and giving its similitude character $\simi \phi_{\pi}$.

 \vskip 10pt
The following describes the construction of $\phi_{\pi}$:
\vskip 5pt

\begin{itemize}
\item[(a)] Suppose that $\pi$ participates in the theta correspondence with $\GSO(D)$, where $D$ is possibly split. Then we have:
\[  \pi = \theta_{D,W}(\tau_1^D \boxtimes \tau_2^D), \]
where $\tau_1^D$ and $\tau_2^D$ have the same central characters.
Let $\phi_i$ denote the $L$-parameter of the Jacquet-Langlands lift of $\tau_i^D$ to $\GL_2(F)$. Then
one sets:
\[  \phi_{\pi} =  \phi_1 \oplus \phi_2 \quad \text{with $\simi \phi_{\pi} = \det \phi_1 = \det \phi_2$.} \]
\vskip 5pt

\item[(b)] Suppose that $\pi$ participates in the theta correspondence with $\GSO(V) = \GSO_{3,3}(F)$. Then we have
\[  \theta_{W,V}(\pi) = \Pi \boxtimes \mu \]
for an irreducible representation $\Pi$ of $\GL_4(F)$ and $\mu = \omega_{\pi}$ is such that $\omega_{\Pi} = \mu^2$. One then sets:
\[  \phi_{\pi} = \phi_{\Pi}  \quad \text{with $\simi \phi_{\pi}  = \mu$.} \]
\end{itemize}
\vskip 10pt

Now by our explicit determination of local theta correspondence, we can explicitly write down the $L$-parameter of any non-supercuspidal representation. This is given in the following proposition.
\vskip 5pt

\begin{Prop} \label{P:L-par}
Let $S_n$ denote the $n$-dimensional irreducible representation of $\SL_2(\C)$. The $L$-parameter of
a non-supercuspidal representation $\pi$ of $\GSp_4$ can be given as follows.
\vskip 5pt

(i) If $\pi = St(\chi,\tau)$, then
\[  \phi_{\pi} = \phi_{\tau} \boxtimes S_2: W_F \times \SL_2(\C) \longrightarrow \GO_2(\C) \times \SL_2(\C) \longrightarrow \GSp_4(\C),\]
with similitude factor $\omega_{\tau} \cdot \chi$, so that the composite
\[  W_F \longrightarrow \GO_2(\C) \longrightarrow \GO_2(\C)/\GSO_2(\C) \cong \{ \pm 1\} \]
is the quadratic character $\chi$.
\vskip 5pt

(ii) If $\pi = St(\tau,\mu)$, then
\[  \phi_{\pi} = \mu \cdot \phi_{\tau} \oplus (\mu \boxtimes S_2): WD_F \longrightarrow
(\GL_2(\C) \times \GL_2(\C))^0 \longrightarrow \GSp_4(\C), \]
with $\simi \phi_{\pi} = \mu^2$.
\vskip 5pt

(iii) If $\pi = St_{\PGSp_4} \otimes \chi$ is a twisted  Steinberg representation, then
\[ \phi_{\pi} = \chi \boxtimes S_4: W_F \times \SL_2(\C) \longrightarrow \GSp_4(\C), \]
with $\simi \phi_{\pi}  = \chi^2$.
\vskip 5pt

(iv) If $\pi \hookrightarrow I_{Q(Z)}(\chi, \tau)$ as in Thm. \ref{T:local-theta}(vi)(a), then
\[  \phi_{\pi} = \phi_{\tau} \oplus \phi_{\tau} \cdot \chi : WD_F \longrightarrow M(\C) \longrightarrow \GSp_4(\C) \]
where $M(\C)$ is the Levi subgroup of the Siegel parabolic subgroup of $\GSp_4(\C)$, and
$\simi \phi_{\pi} = \chi \cdot \omega_{\tau}$.
\vskip 5pt

(v) If $\pi \hookrightarrow I_{P(Y)}(\tau,\chi)$ as in Thm. \ref{T:local-theta}(vi)(b), then
\[  \phi_{\pi} = \chi \oplus \chi \cdot \phi_{\tau} \oplus \chi \omega_{\tau}: WD_F \longrightarrow
L(\C) \longrightarrow \GSp_4(\C) \]
where $L(\C)$ is the Levi subgroup of the Klingen (or Heisenberg) parabolic subgroup of $\GSp_4(\C)$, and $\simi \phi_{\pi} = \chi^2 \cdot \omega_{\tau}$.
\vskip 5pt

(vi) If $\pi \hookrightarrow I_B(\chi_1, \chi_2; \chi)$ as in Thm. \ref{T:local-theta}(vi)(c), then
\[  \phi_{\pi} = \chi\chi_1\chi_2 \oplus \chi  \oplus \chi\chi_1 \oplus \chi\chi_2: WD_F \longrightarrow T(\C) \longrightarrow \GSp_4(\C), \]
where $T(\C)$ is the diagonal maximal torus of $\GSp_4(C)$ and $\simi \phi_{\pi} = \chi^2 \chi_1 \chi_2$.
\vskip 10pt

In particular, we see that $\phi_{\pi}$ is a discrete series parameter if and only if $\pi$ is a discrete series representation. Moreover, the map $\pi \mapsto \phi_{\pi}$ defines a bijective map
\[  \Pi(\GSp_4)_{NDS} \smallsetminus \Pi(\GSp_4)_{ng}^{temp} \longrightarrow \Phi(\GSp_4)_{NDS} \]
where the subscript {\em NDS} on both sides stand for ``non-discrete series".
\end{Prop}
\vskip 10pt

The reader can easily verify that the above $L$-parameters agree with the prescription given in
[RS, \S A.5].

\vskip 5pt

Finally, the following proposition was used in the proof of [GT1, Main Theorem (vii)], which relates the genericity of $\pi$ with its adjoint $L$-factor.  A proof of this was also given in [AS], but our verification is more concise.
\vskip 10pt

\begin{Prop} \label{P:generic}
For the $L$-parameters $\phi$ described in Prop. \ref{P:L-par}, the adjoint $L$-factor $L(s, Ad \circ \phi)$ is holomorphic at $s =1$ if and only if the $L$-packet $L_{\phi}$ contains a generic representation.
\end{Prop}

\begin{proof}
This is a simple calculation using Prop. \ref{P:L-par} and the knowledge of reducibility points of various principal series. But let us make a few simple observations:
\vskip 5pt

\begin{itemize}
\item[(a)]  If $\phi: WD_F \longrightarrow \GSp_4(\C)$ is an $L$-parameter, then $Ad \circ \phi = \Sym^2 \phi \otimes \simi(\phi)^{-1}$;
\vskip 5pt

\item[(b)] If $\rho \boxtimes S_r$ is a representation of $W_F \times \SL_2(\C)$, then $L(s, \rho \boxtimes S_r)
= L(s+ (r-1)/2, \rho)$. Thus $L(s, \rho \boxtimes S_r)$ has a pole at $s = 1$ if and only if
$\rho$ contains the unramified character $|-|^{-(r+1)/2}$ (regarded as a character of $W_F$)
as a constituent.
\vskip 5pt

\item[(c)] In the context of Prop. \ref{P:L-par}(iv), (v) and (vi) (but with $\chi \ne 1$ in case (iv)), the $L$-packet
for $\phi_{\pi}$ is a singleton set containing only $\pi$. Moreover, $\pi$ is generic precisely when the relevant principal series containing $\pi$ as the unique irreducible submodule is irreducible. This is because the standard module conjecture holds for $\GSp_4$.
\end{itemize}
\vskip 15pt

Now we consider each case of Prop. \ref{P:L-par} in turn.
\vskip 5pt

\noindent (i) If $\pi = St(\chi,\tau)$,  which is generic, then by (a),
\[  Ad \circ \phi_{\pi} = (\chi \cdot Ad(\phi_{\tau}) \boxtimes S_3) \oplus (\chi \boxtimes S_1) \]
so that
\[  L(s, Ad \circ \phi_{\pi}) = L(s+1, Ad(\phi_{\tau}) \times \chi) \cdot L(s, \chi). \]
Since the only 1-dimensional characters in $Ad(\phi_{\tau})$ are precisely those quadratic $\chi_K$ such that $\tau \otimes \chi_K = \tau$, we see that $L(s, Ad \circ \phi_{\pi})$ is holomorphic at $s=1$ by (b).
\vskip 10pt

\noindent (ii) If $\pi = St(\tau, \mu)$, which is generic, then by (a),
\[  Ad \circ \phi_{\pi} = Ad(\phi_{\tau}) \oplus (1 \boxtimes S_3) \oplus (\phi_{\tau} \boxtimes S_2). \]
If $\tau$ is supercuspidal, then as in (i), the only characters contained in $Ad(\phi_{\tau})$ are quadratic. It follows by (b) that the adjoint $L$-factor is holomorphic at $s=1$. On the other hand, if
$\tau = st_{\chi}$ is a twisted Steinberg representation with $\chi$ a non-trivial quadratic character, then
\[  Ad \circ \phi_{\pi} = 2 \cdot (1 \boxtimes S_3) \oplus (\chi \boxtimes S_3) \oplus (\chi \boxtimes S_1). \]
It follows from (b) that the adjoint $L$-factor is holomorphic at $s=1$.
\vskip 10pt

\noindent (iii) If $\pi = St_{\PGSp_4} \otimes \chi$, which is generic, then by (a),
\[  Ad \circ \phi_{\pi} = (1 \boxtimes S_3) \oplus (1 \boxtimes S_7), \]
so that $L(s, Ad \circ \phi_{\pi}) = \zeta(s+1) \cdot \zeta(s+3)$, which is clearly holomorphic at $s = 1$.
\vskip 10pt

\noindent (iv) If $\pi \hookrightarrow I_{Q(Z)}(\chi ,\tau)$, where $|\chi| = |-|^{-s_0}$ with $s_0 \geq 0$, then
\[  Ad \circ \phi_{\pi} = \chi \cdot Ad(\phi_{\tau}) \oplus \chi^{-1} \cdot Ad(\phi_{\tau}) \oplus
(\phi_{\tau}  \otimes \phi_{\tau}^{\vee}). \]
If $\tau$ is supercuspidal, then it follows from (b) that the adjoint $L$-factor is non-holomorphic
at $s=1$ if and only if there is a quadratic character $\chi_0$ such that
\[  \tau \otimes \chi_0 = \tau \quad \text{and} \quad \chi \cdot \chi_0 = |-|^{-1}. \]
Similarly, when $\tau = st_{\mu}$ is a twisted Steinberg representation, then
\[  Ad \circ\phi_{\pi} = ((\chi \oplus 1 \oplus \chi^{-1}) \boxtimes S_3) \oplus (1 \boxtimes S_1). \]
By (b), it follows that the adjoint $L$-factor is non-holomorphic at $s=1$ precisely when
\[  \chi = |-|^{-2}. \]
Comparing this with Lemma \ref{L:Q(Z)} and taking note of (c), we see that when $\chi \ne 1$, $\pi$ is non-generic iff $I_{Q(Z)}(\chi, \tau)$ is reducible iff the adjoint $L$-factor is holomorphic at $s=1$. If $\chi =1$, then
$I_{Q(Z)}(1,\tau)$ is the sum of two representations which form an $L$-packet. This $L$-packet thus contains a generic element and the adjoint $L$-factor is indeed holomorphic at $s=1$.
\vskip 10pt

\noindent (v) If $\pi \hookrightarrow I_{P(Y)}(\tau, \chi)$, where $|\omega_{\tau}| = |-|^{-s_0}$ with
$s_0 \geq 0$, then  by (a),
 \[ Ad \circ \phi_{\pi} = \omega_{\tau} \oplus \omega_{\tau}^{-1}  \oplus \phi_{\tau} \oplus
 \phi_{\tau}^{-1} \oplus Ad(\phi_{\tau}) \oplus 1. \]
If $\tau$ is supercuspidal, it follows by (b) that the adjoint $L$-factor is non-holomorphic at $s =1$ precisely when
\[  \omega_{\tau} = |-|^{-1}. \]
Similarly, when $\tau = st_{\mu}$ is a twisted Steinberg representation, then
\[  Ad \circ \phi = ((\mu^2 \oplus \mu^{-2} \oplus 1) \boxtimes S_1)  \oplus (\mu \oplus \mu^{-1}) \boxtimes S_2) \oplus (1 \boxtimes S_3). \]
By (b), it follows that the adjoint $L$-factor is non-holomorphic at $s=1$ iff
\[  \mu^2 = |-|^{-1} \quad \text{or} \quad  \mu = |-|^{-3/2}. \]
In view of (c), a comparison with Lemma \ref{L:P(Y)} gives the desired result.
\vskip 10pt

\noindent (vi) If $\pi \hookrightarrow I_B(\chi_1, \chi_2 ; \chi)$, then it follows by (a) that
$Ad \circ \phi_{\pi}$ is the direct sum of the following 1-dimensional characters:
\[  \chi_1^{\pm 1},  \, \chi_2^{\pm 1}, \, (\chi_1\chi_2)^{\pm 1}, \, (\chi_1/\chi_2)^{\pm 1}, \, 1 \,
\text{(with multiplicity 2)}. \]
By (b), the adjoint $L$-factor is non-holomorphic at $s=1$ precisely when one of the above character is equal to $|-|^{-1}$. By [ST] (see also [RS, Pg. 37]), this is precisely when the induced representation $I_B(\chi_1,\chi_2;\chi)$ is reducible so that $\pi$ is non-generic. This proves the desired result.

\end{proof}
\vskip 15pt

\section{\bf Tables of Explicit Local Theta Correspondence.}  \label{S:table}
In this section, we  display the results of local theta correspondences in the form of tables.  Note however that, unlike Thms.  \ref{T:local-theta-SO(D)}, \ref{T:local-theta-GSO(2,2)} and \ref{T:local-theta},  we have described the representation $\pi$ in terms of the usual Langlands classification, so that $\pi$ is the unique irreducible quotient of a standard module.  So, for example, $J(\pi(\chi_1,\chi_2))$ stands for the unique irreducible quotient of the principal series representation $\pi(\chi_1,\chi_2)$ of $\GL_2(F)$.

\vskip 15pt
\newpage

\begin{table}
\centering \caption{Explicit theta lifts from $\GSp_4$}
\label{maintable} \vspace{-3ex}
$$
\renewcommand{\arraystretch}{1.5}
 \begin{array}{|c|c|c|c|c|c|c|}
  \hline&&\multicolumn{2}{|c|}{\mbox{$\pi$}}
   &\mbox{$\theta_{(3,3)}(\pi)$}
   &\mbox{$\theta_{(2,2)}(\pi)$}
   &\mbox{$\theta_{(4,0)}(\pi)$}\\\hline\hline
   &\mbox{a}&\multicolumn{2}{|c|}{\mbox{not a lift from $\GO_{2,2}$ or $\GO_{4,0}$}}
   &\mbox{S.C.}
   &0
   &0\\ \cline{3-7}
   {\mbox{S.C.}}
   &\mbox{b}&\multicolumn{2}{|c|}{\mbox{$\theta(\tau_1\boxtimes\tau_2)$, $\tau_1\neq\tau_2$, both S.C.}}
   &I_P(\tau_1,\tau_2)\boxtimes\omega_{\tau_1}
   &\tau_1\boxtimes\tau_2
   &0\\  \cline{3-7}

   &\mbox{c}&\multicolumn{2}{|c|}{\mbox{$\theta(\tau^D_1\boxtimes\tau^D_2)$, $\tau^D_1
   \neq\tau^D_2$}}
    &0
   & 0
   &\tau^D_1\boxtimes\tau^D_2\\
  \hline
   &\mbox{a}&\multicolumn{2}{|c|}{\mbox{$St(\chi,\tau)$}}
   &\mbox{$St(\tau)\boxtimes\omega_{\tau}\chi$}
   &0
   &0\\ \cline{3-7}
 \mbox{D.S.} &\mbox{b}&\multicolumn{2}{|c|}{\mbox{$St(\tau,\mu)$}}
   &\mbox{$I_P(\tau\cdot \mu, st\cdot  \mu)\boxtimes\mu^2$}
   & (\tau \cdot \mu)\boxtimes st_{\mu}
   &0\\ \cline{3-7}
  &\mbox{c}& \multicolumn{2}{|c|}{\mbox{$St_{\PGSp_4}\otimes\chi$}}
   &St_{\PGL_4}\boxtimes\chi^2
   &0
   &0\\\hline

  &\mbox{a}& \multicolumn{2}{|c|}{\mbox{$ J_{Q(Z)}(\chi,\tau)$, $\chi \ne 1$}}
   &  J_P(\tau \cdot\chi , \tau)\boxtimes\omega_{\tau}\chi
        &0
   &0\\ \cline{3-7}

      &\mbox{b}  & \raisebox{-0.5ex}[0.5ex]{$\pi \hookrightarrow$}
   & \pi = \pi_{gen}(\tau)
   &  J_P(\tau , \tau)\boxtimes\omega_{\tau}
  &\tau \boxtimes \tau
   &0 \\ \cline{4-5}\cline{6-7}

     \mbox{N.D.S.}  &\mbox{c}&  \raisebox{1.0ex}[-1.0ex]{$I_{Q(Z)}(1,\tau)$}
   &\pi=\pi_{ng}(\tau)
   & 0
   &0
   & \tau_D\boxtimes \tau_D \\ \cline{3-7}

   &\mbox{d}&\multicolumn{2}{|c|}{J_{P(Y)}(\tau,\chi)}
   &J_Q(\omega_{\tau},\tau,1)\cdot\chi\boxtimes\chi^2\omega_\tau
   & (\tau \cdot \chi) \boxtimes J(\pi(\omega_{\tau}\chi, \chi))
   &0\\ \cline{3-7}
   &&\multicolumn{2}{|c|}{}
   &\chi\cdot J_{B_0}(\chi_1\chi_2,\chi_1,\chi_2,1)
   & J(\pi(\chi\chi_1,\chi\chi_2))
   &\\
   &\raisebox{2.5ex}[-2.5ex]{\mbox{e}}
   &\multicolumn{2}{|c|}{\raisebox{2.5ex}[-2.5ex]{$J_B(\chi_1,\chi_2;\chi)$}}
   &\qquad\qquad\quad\boxtimes\chi^2\chi_1\chi_2
   &\quad \boxtimes J(\pi(\chi\chi_1\chi_2,\chi))
  &\raisebox{2.5ex}[-2.5ex]{0}
   \\\hline
\end{array}
$$
\end{table}

\quad\\
\quad\\

\begin{table}
\centering \caption{Explicit theta lifts from $\GSO_{2,2}$ to $\GSp_4$}
\label{maintable2} \vspace{-3ex}
$$
\renewcommand{\arraystretch}{1.5}
 \begin{array}{|c|c|c|}
  \hline&\mbox{$\tau_1\boxtimes\tau_2$}
   &\mbox{$\theta(\tau_1\boxtimes\tau_2) =\theta(\tau_2\boxtimes\tau_1) $}\\\hline\hline
   \mbox{a}& \tau_1=\tau_2=\tau=\text{D.S.}
   &\pi_{gen}(\tau)\\ \hline
   \mbox{b}&\mbox{$\tau_1 \ne \tau_2$ both S.C.}
   &\mbox{S.C.}\\ \hline
   \mbox{c}&\tau_1=\text{S.C.}, \quad  \tau_2=st_{\chi}
   &St(\tau_1\otimes\chi^{-1},\chi)\\ \hline
   \mbox{d}& \tau_1=st_{\chi_1},\quad   \tau_2=st_{\chi_2},\;\chi_1\neq\chi_2
   &St(st_{\chi_1/\chi_2},\chi_2)=St(st_{\chi_2/\chi_1},\chi_1)\\ \hline
   \mbox{e}&\tau_1=\text{D.S.},\quad \tau_2 = J(\pi(\chi',\chi))
   &J_{P(Y)}(\tau_1\otimes\chi^{-1},\chi) \\ \hline
   \mbox{f} &\tau_1 = J(\pi(\chi'_1,\chi_1)), \quad \tau_2 =  J(\pi(\chi'_2,\chi_2))
   &J_B(\chi_2'/\chi_1,\chi_2/ \chi_1;\chi_1)\\ \hline
\end{array}
$$
\end{table}

\begin{table}
\centering \caption{Explicit theta lifts from $\GSO_{4,0}$ to $\GSp_4$}
\label{maintable3} \vspace{-3ex}
$$
\renewcommand{\arraystretch}{1.5}
 \begin{array}{|c|c|c|}
  \hline&\mbox{$\tau_1^D\boxtimes\tau_2^D$}
   &\mbox{$\Theta(\tau_1^D\boxtimes\tau_2^D)=\theta(\tau_2^D\boxtimes\tau_1^D)$}\\\hline\hline
   \mbox{a}&\tau_1^D=\tau_2^D=\tau_D
   &\pi_{ng}(JL(\tau_D))\\ \hline
   \mbox{b}&\tau_1^D\neq\tau_2^D
   &\mbox{non-generic, \;S.C.}\\ \hline
\end{array}
$$
\end{table}

\newpage

\appendix\section{\bf Jacquet modules of the Weil representation}  \label{S:jacquet}
In this appendix, we give a detailed computation of the Jacquet module of the (induced) Weil representation. We deal not only with the dual pairs for the small ranks we worked with but for all  ranks. Hence throughout this appendix, $(V_m,(-,-))$ will denote a symmetric bilinear space with $\dim V_m=m$ even and $(W_n, \langle-,-\rangle)$ a symplectic space with $\dim W_n=2n$. And $\chi_{V}$ denotes the character of $V_m$ as usual. We fix a polarization
\[
    V_m=X_r+ V_{an}+ X_r^\ast
\]
of $V_m$ where $X_{an}$ is anisotropic and $X_r + X_r^\ast=\H^r$. We let $\{v_1,\dots,v_r\}$ (resp. $\{v_1^\ast,\dots,v_r^\ast\}$) be a basis of $X_r$ (resp. $X_r^\ast$) with $(v_i,v_j^\ast)=\delta_{ij}$. Also we fix a polarization
\[
    W_n=Y_n\oplus Y_n^\ast
\]
of $W_n$ and we let $\{e_1,\dots,e_n\}$ (resp. $\{e_1^\ast,\dots,e_n^\ast\}$) be a basis of $Y_n$ (resp. $Y_n^\ast$) with $\langle e_i,e_j^\ast\rangle=\delta_{ij}$.

Let
\[
    R_{m,n}=R=\GO(V_m)\times \GSp(W_n)^+
\]
and $\omega_{m,n}=\omega_{V_m W_n}$ be the Weil representation of
\[
    R_0=\{(h,g)\in R_{m,n}:\lambda_{V}(h)\cdot\lambda_{W}(g)=1\}.
\]
Note that in this appendix, we consider various subspaces of $V_m$ and $W_n$ and their similitude groups. For the similitude characters of those similitude groups, we always use the same symbols $\lambda_V$ and $\lambda_W$ because this will not produce any confusion. Also let
\[
    \Omega_{m,n}=\Omega_{V_m,W_n}=ind_{R_0}^R\omega_{m,n}
\]
be the induced Weil representation of $\GO(V_m)\times\GSp(W_n)$.

Now let
\[
    X_t=\Span\{v_1,\dots,v_t\}\text{ and }X_t^\ast=\Span\{v_1^\ast,\dots,v_t^\ast\}
\]
and write
\[
    V_m=X_t+V_{m_0}+X_t^\ast
\]
so that $V_{m_0}=V_{an}+\H^{r-t}$ and
\[
    \dim V_{m_0}=m_0=m-2t.
\]
Let $P(X_t)$ be the parabolic subgroup that stabilizes $X_t$. Then we write
\[
    P(X_t)=M(X_t)N(X_t)
\]
where $M(X_t)\cong\GL(X_t)\times\GO(V_{m_0})$ is the Levi part and $N(X_t)$ is the unipotent part. We use $[a,h]$ to denote an element in $M(X_t)$ where $(a,h)\in\GL(X_t)\times\GO(V_{m_0})$.

Next let
\[
    Y_k=\Span\{e_1,\dots,e_k\}\text{ and }Y_k=\Span\{e_1^\ast,\dots,e_k\}
\]
and we write
\[
    W_n=Y_k+W_{n_0}+Y_k^\ast
\]
so that
\[
    \dim W_{n_0}=2n_0=2(n-k).
\]
Let $Q(Y_k)=M(Y_k)N(Y_k)$ be the parabolic that stabilizes $Y_k$, so that $M(Y_k)\cong\GL(Y_k)\times\GSp(W_{n_0})$.
We use $[b,g]$ to denote an element in $M(Y_k)$ where $(b,g)\in\GL(Y_k)\times\GO(W_{n-k})$.

Then we compute the Jacquet modules $R_{P(X_t)}(\Omega_{m,n})$ and $R_{Q(Y_k)}(\Omega_{m,n})$ of $\Omega_{m,n}$ along the parabolic subgroups $P(X_t)$ and $Q(Y_k)$, respectively. For the former, assuming $t\geq k$, let
\[
    X_{t-k}=\Span\{v_{1},\dots,v_{t-k}\}\text{ and }X_{t-k}^\ast=\Span\{v_{1}^\ast,\dots,v_{t-k}^\ast\}
\]
and $P(X_{t-k},X_t)$ be the parabolic subgroup of $M(X_t)$ that preserves the flag
\[
    0\subseteq X_{t-k}\subseteq X_t\subseteq V_m,
\]
so that
\[
    P(X_{t-k}, X_t)\cong R(X_{t-k}, X_t)\times\GO(V_{m_0}),
\]
where $R(X_{t-k}, X_t)$ is the parabolic subgroup of $\GL(X_t)$ that preserves $0\subseteq X_{t-k}\subseteq X_t$, i.e.
\[
    R(X_{t-k}, X_t)=\{a=\begin{pmatrix}a_1&\ast\\&a_2\end{pmatrix}:a_1\in\GL(X_{t-k}), a_2\in\GL(X'_k)\},
\]
and
\[
    X'_{k}=\Span\{v_{t-k+1},\dots,v_t\}\text{ and }{X'_k}^\ast=\Span\{v_{t-k+1}^\ast,\dots,v_t^\ast\}.
\]
For the latter, assuming $k\geq t$, let
\[
     Y_{k-t}=\Span\{e_{1},\dots,e_{k-t}\}\text{ and }Y_{k-t}^\ast=\Span\{e_{1}^\ast,\dots,e_{k-t}^\ast\}
\]
and $Q(Y_{k-t}, Y_k)$ be the parabolic subgroup of $M(Y_k)$ that preserves the flag
\[
    0\subseteq Y_{k-t}\subseteq Y_k\subseteq W_n,
\]
so that
\[
    Q(Y_{k-t}, Y_k)\cong R(Y_{k-t}, Y_k)\times \GSp(W_{n_0}),
\]
where $R(Y_{k-t}, Y_k)$ is the parabolic subgroup of $\GL(Y_k)$ that preserves $0\subseteq Y_{k-t}\subseteq Y_k$, i.e.
\[
    R(Y_{k-t}, Y_k)=\{b=\begin{pmatrix}b_1&\ast\\&b_2\end{pmatrix}:b_1\in\GL(Y_{k-t}), b_2\in\GL(Y'_t)\},
\]
and
\[
    Y'_{t}=\Span\{e_{k-t+1},\dots,e_k\}\text{ and }{Y'_k}^\ast=\Span\{e_{k-t+1}^\ast,\dots,e_k^\ast\}.
\]
Then we have

\begin{Thm}\label{T:Jacquet_module1}
The normalized Jacquet module $R_{P(X_t)}(\Omega_{m,n})$ of the Weil representation $\Omega_{m,n}$ along the parabolic $P(X_t)$ has an $M(X_t)\times\GSp(W_n)$ invariant filtration
\[
    \{0\}\subseteq J^{(\min\{t,n\})}\subseteq\cdots\subseteq J^{(1)}\subseteq J^{(0)}=R_{P(X_t)}(\Omega_{m,n})
\]
with the successive quotient
\[
    J^k:=J^k/J^{k+1}\cong\Ind_{P(X_{t-k},X_t)\times Q(Y_k)^+}^{M(X_t)\times\GSp(W_n)^+}
(S(\Isom(X'_k,Y_k))\otimes\Omega_{m_0,n-k}),
\]
where $\Omega_{m_0,n-k}$ is the Weil representation of the group $\GO(V_{m_0})\times\GSp(W_{n-k})$, and the group $P(X_{t-k},X_t)\times Q(Y_k)^+$ acts on $S(\Isom(X'_k,Y_k))$ as follows: Let $\varphi(A)\in S(\Isom(X'_k,Y_k))$. Then each element $[\begin{pmatrix}a_1&\ast\\&a_2\end{pmatrix},h]\in P(X_{t-k},X_t)$ acts as
\[
    [\begin{pmatrix}a_1&\ast\\&a_2\end{pmatrix},h]\cdot\varphi(A)
    =\chi_V(\det a_2)|\lambda_{V}(h)|^{e_0}|\det a_1|^{e_1}\varphi(A a_2)
\]
where
\begin{align*}
    e_0&=-\frac{1}{4}(m-2t)k-\frac{1}{2}tn+\frac{1}{4}mt-\frac{1}{4}t(t+1)\\
    e_1&=n-\frac{1}{2}m+\frac{1}{2}t-\frac{1}{2}(k-1), \quad (\text{for $k < t$})
\end{align*}
and each element $[b,g]n\in Q(Y_k)$ acts as
\[
    [b,g]n\cdot\varphi(A)
    =|\lambda_{W}(g)|^{f_0}\varphi(\lambda_{W}(g)b^{-1}A),
\]
where
\begin{align*}
    f_0&=-\frac{1}{2}kn+\frac{1}{4}k(k-1).
\end{align*}
Note that the induction is normalized.
\end{Thm}

\begin{Thm}\label{T:Jacquet_module2}
The normalized Jacquet module $R_{Q(Y_k)}(\Omega_{m,n})$ of the Weil representation $\Omega_{m,n}$ along the parabolic $Q(Y_k)$ has an $\GO(V_m)\times M(Y_k)^+$ invariant filtration
\[
    \{0\}\subseteq J^{(\min\{k,r\})}\subseteq\cdots\subseteq J^{(1)}\subseteq J^{(0)}=R_{Q(Y_k)}(\Omega_{m,n})
\]
with the successive quotient
\[
    J^t:=J^t/J^{t+1}\cong\Ind_{P(X_t)\times Q(Y_{k-t}, Y_k)^+}^{\GO(V_m)\times M(Y_k)^+}
(S(\Isom(Y'_t, X_t))\otimes\Omega_{m-2t,n_0}),
\]
where $\Omega_{m-2t,n_0}$ is the Weil representation of the group $\GO(V_{m-2t})\times\GSp(W_{n_0})$, and the group $P(X_t)\times Q(Y_{k-t}, Y_k)^+$ acts on $S(\Isom(Y'_t, X_t))$ as follows: Let $\varphi(A)\in S(\Isom(Y'_t,X_t))$. Then each element $[\begin{pmatrix}b_1&\ast\\&b_2\end{pmatrix},g]\in Q(Y_{k-t}, Y_k)^+$ acts as
\[
    [\begin{pmatrix}b_1&\ast\\&b_2\end{pmatrix},g]\cdot\varphi(A)
    =\chi_V(\det b_1)\chi_V(\det b_2)|\lambda_{W}(g)|^{e_0}|\det b_1|^{e_1}\varphi(Ab_2)
\]
where
\begin{align*}
    e_0&=  - \frac{1}{2}t(n-k) - \frac{1}{4}mk +\frac{1}{2}kn-\frac{1}{4}k(k-1)\\
    e_1&=\frac{1}{2}m-n+\frac{1}{2}k-\frac{1}{2}(t+1), \quad (\text{for $t < k$})
\end{align*}
and each element $[a,h]n\in P(X_t)$ acts as
\[
    [a,h]n\cdot\varphi(A)
    =|\lambda_{V}(h)|^{f_0}\varphi(\lambda_V(h)a^{-1}A),
\]
where
\begin{align*}
    f_0&=-\frac{1}{4}mt+\frac{1}{4}t(t+1).
\end{align*}
Note that the induction is normalized.
\end{Thm}

\begin{Rmk}
If $m_0$ or $n-k$ in $\Omega_{m_0, n-k}$ (resp. $m-2k$ or $n_0$ in $\Omega_{m-2k, n_0}$) is zero, then $\Omega_{m_0, n-k}$ (resp. $\Omega_{m-2k, n_0}$) is the induced representation of the trivial representation. For example if $m_0=0$, then $\Omega_{m_0, n-k}$ is realized in the space $S(F^\times)$ where $([1,\lambda],[1,g])$ acts as $([1,\lambda],[1,g])\cdot f(x)=(x\cdot\lambda\cdot\lambda_W(g))$ for $\lambda\in\GL_1$, $g\in\GSp_{n-k}$ and $f\in S(F^\times)$.
\end{Rmk}
\begin{Rmk}
Note that the roles of $k$ and $t$ are switched in Theorems  \ref{T:Jacquet_module1} and \ref{T:Jacquet_module2}.
\end{Rmk}
\vskip 5pt

\begin{Rmk}  \label{R:kt}
For an induced Weil representation $\Omega_{W,V}$, and any character $\chi$, one has
\[  \Omega_{V,W} \otimes ((\chi \circ \lambda_W) \boxtimes (\chi \circ \lambda_V)) \cong  \Omega_{W,V}. \]
Thus, in Theorems \ref{T:Jacquet_module1}
and \ref{T:Jacquet_module2}, one could replace the pair $(e_0, f_0)$ by $(e_0 - f_0, 0)$.
Now observe that when $k=t$ in Theorems \ref{T:Jacquet_module1}
and \ref{T:Jacquet_module2}, one has $e_0 = f_0$. Thus, for $k = t$, one could simply take
$e_0 = f_0 = 0$.
\end{Rmk}

\begin{Rmk}
One can obtain the analogous theorems for the isometry case simply by replacing the induced Weil representation $\Omega_{m_0, n-k}$ (or $\Omega_{m-2t, n_0}$) by the Weil representation $\omega_{m_0, n-k}$ (or $\omega_{m-2t, n_0}$) and disregarding the similitude factors. For example, the Jacquet module $R_{P(X_t)}(\omega_{m,n})$ along the parabolic $P(X_t)$ of $\OO(V_m)$ has the analogous filtration where each successive quotient has the inducing data of the form $S(\Isom(X'_k,Y_k))\otimes\omega_{m_0,n-k}$ where the action of the relevant subgroup of $\OO(V_m)\times\Sp(W_n)$ on $S(\Isom(X'_k,Y_k))$ is simply the restriction of the similitude case. And the similar statement holds for $R_{Q(Y_k)}(\omega_{m,n})$.
\end{Rmk}

The rest of the appendix is devoted to the proof of those theorems. The proof in the context of the isometry groups appears in Kudla (\cite{K}). Our proof closely follows Kudla's computation, though we give more details. Also in \cite{K}, Kudla considers the Jacquet module along the parabolic of the symplectic group, but we do the other way around. Namely we consider the Jacquet module $R_{P(X_t)}(\Omega_{m,n})$ along the parabolic $P(X_t)$ of $\GO(V_m)$ i.e. Theorem \ref{T:Jacquet_module1}, and leave the other case to the reader.

First recall that
\[
    \W:=V_m\otimes W_n
\]
is equipped with the obvious symplectic structure $\langle\langle-,-\rangle\rangle$ and for each polarization $\W=\Y+\Y^\ast$, the Weil representation $\omega_{m,n}$ can be realized in the space $S(\Y)$ of Schwartz functions on $\Y$ called the Schrodinger model of $\omega_{m,n}$ with respect to the polarization. To compute the Jacquet module of the Weil representation $\omega_{m,n}$, one needs to consider the Schrodinger models for various polarizations. First, consider the Schrodinger model with respect to the polarization
\[
    \W=V_m\otimes Y_n^\ast+V_m\otimes Y_n
\]
so that $\omega_{m,n}$ is realized in the space $S(V_m\otimes Y_n^\ast)$. Then each $(h,g)\in R_0$ acts as
\[
    (h,g)\varphi(x)=|\lambda_W(g)|^{\frac{1}{4}mn}\omega_{m,n}(1,g_1)\cdot\varphi(h^{-1}x)
\]
where
\[
    g_1=g\begin{pmatrix}\lambda_W(g)^{-1}&0\\0&1\end{pmatrix},
\]
and $\omega_{m,n}(1,g_1)$ is the action of the Weil representation for the usual isometry group.

Next let $t\leq r$ and $m_0$ be such that $2t+m_0=m$. Consider
\[
    V_m=X_t+ V_{m_0}+X_t^\ast
\]
where $X_t$, $V_{m_0}$ and $X_t^\ast$ are as above. Then we compute the Jacquet module of $\Omega_{m,n}$ along the parabolic subgroup $P(X_t)=M(X_t)N(X_t)$ of $\GO(V_m)$. Let $N(X_t)_0$ be the center of $N(X_t)$, so that it fits in the exact sequence
\[
    1\longrightarrow N(X_t)_0\longrightarrow N(X_t)\longrightarrow Hom(V_{m_0}, X_t)\longrightarrow 1.
\]
Indeed in terms of the obvious matrix realization of $\GO(V_m)$, $N(X_t)$ is written as
\[
    N(X_t)=\{n(c,d)=\begin{pmatrix}1&c&d-\frac{1}{2}c\circ c^\ast\\&1&-c^\ast\\&&1\end{pmatrix}:
    c\in Hom(V_{m_0},X_t), d\in Hom(X_t^\ast,X_t)\},
\]
where $c^\ast$ is the adjoint of $c$, i.e. $c^\ast\in Hom(X_t^\ast,V_{m_0})$ such that
\[
    (cv, u)=(v,c^\ast u)\quad\text{ for } v\in V_{m_0}, u\in V_{m_0},
\]
and $d$ is such that $(dx_1,x_2)=(x_1,-dx_2)$ for all $x_1,x_2\in X_t^\ast$. Note that if we make the identification $Hom(X_t^\ast,X_t)=\GL_t$ with respect to the above chosen bases of $X_t$ and $X_t^\ast$, we have $^td=-d$. Note that
\[
    N(X_t)_0=\{n(c,d):c=0\}=\{\begin{pmatrix}1&&d\\&1&\\&&1\end{pmatrix}:\; ^td=-d\}.
\]

Now we consider the polarization
\[
    \W=\Y^\ast+\Y=(W_n\otimes X_t^\ast+V_{m_0}\otimes Y_n^\ast)+(W_n\otimes X_t + V_{m_0}\otimes Y_n)
\]
to realize the Weil representation $\omega_{m,n}$ of $R_0\subset\GO(V_m)\times \GSp(W_n)^+$. So the space of $\omega_{m,n}$ is
\[
    S(\Y^\ast)=S(W_n\otimes X_t^\ast+V_{m_0}\otimes Y_n^\ast)\cong S(W_n\otimes X_t^\ast)\otimes S(V_{m_0}\otimes Y_n^\ast).
\]
Now for any subgroup $H$ of $R$, let us define
\[
    R_0(H):=R_0\cap H.
\]
Write
\[
    W_n\otimes X_t^\ast=Y_n\otimes X_t^\ast+Y_n^\ast\otimes X_t^\ast
\]
and denote each element $w\in W_n\otimes X_t^\ast$ as
\[
    w=y+y^\ast\in Y_n\otimes X_t^\ast+Y_n^\ast\otimes X_t^\ast
\]
where $y\in Y_n\otimes X_t^\ast$ and $y^\ast\in Y_n^\ast\otimes X_t^\ast$. Then the action of $R_0$ is described as follows: Let $([a,h],g)\in R_0(M(X_t)\times\GSp(W_n)^+)$ where $(a,h)\in\GL(X_t)\times\GO(V_{m_0})$. Then for $\phi_1(y+y^\ast)\otimes\phi_2(x)\in S(W_n\otimes X_t^\ast)\otimes S(V_{m_0}\otimes Y_n^\ast)$,
\begin{align*}
    &([a,h],g)\cdot\phi_1(y+y^\ast)\otimes\phi_2(x)\\
    &\qquad\qquad=|\det a|^n|\lambda_{V}(h)|^{-\frac{1}{2}tn}\phi_1((g_1^{-1}\otimes a^{\ast})(y+\lambda_{V}(h)^{-1}y^\ast))
    \otimes \omega_{{m_0},n}(h,g)\phi_2(x)
\end{align*}
where $x\in V_{m_0}\otimes Y_n^\ast$, and $a^\ast\in\GL(X_t^\ast)$ is such that $(ax_1,x_2)=(x_1,a^\ast x_2)$ for all $x_1\in X_t$ and $x_2\in X_t^\ast$. Note that $\omega_{{m_0},n}$ is the Weil representation for the pair $(\GO(V_{m_0}),\GSp(W_n)^+)$. Also the action of $N(X_t)$ is described as follows: Let $\phi(y+y^\ast+x)\in S(\Y^\ast)=S(W_n\otimes X_t^\ast+V_{m_0}\otimes Y_n^\ast)$. Then
\[
    n(c,d)\cdot\phi(y+x+y^\ast)=\psi(\langle\langle y,dy^\ast \rangle\rangle)\rho(-c^\ast(y+y^\ast))\phi(y+x+y^\ast)
\]
where $\rho$ is the action of the Heisenberg group $H(\W)$ in $S(\Y^\ast)$. Those actions can be shown by looking at how $R_0$ acts on the Weil representation realized in the space
\[
    S(V_n\otimes Y_n^\ast)=S(Y_n^\ast\otimes X_t+V_{m_0}\otimes Y_n^\ast +Y_n^\ast\otimes X_t^\ast)
\]
and the partial Fourier transform
\[
    \mathcal{F}:S(Y_n^\ast\otimes X_t+V_{m_0}\otimes Y_n^\ast +Y_n^\ast\otimes X_t^\ast)
    \rightarrow S(\Y^\ast)=S(Y_n\otimes X_t^\ast+V_{m_0}\otimes Y_n^\ast+Y_n^\ast\otimes X_t^\ast)
\]
given by
\[
    \mathcal{F}(\varphi)(y+x+y^\ast)=\int_{Y_n^\ast\otimes X_t}
    \psi(\langle\langle y,z \rangle\rangle) \varphi\begin{pmatrix}z\\x\\y^\ast\end{pmatrix}\,dz,
\]
where $y+x+y^\ast\in Y_n\otimes X_t^\ast+V_{m_0}\otimes Y_n^\ast+Y_n^\ast\otimes X_t^\ast$ and $z\in Y_n^\ast\otimes X_t$. For example, let $[a,1]\in M(X_t)$. Then
\begin{align*}
[a,1]\cdot \mathcal{F}(\varphi)(y+x+y^\ast)
&=\mathcal{F}([a,1]\cdot\varphi)(y+x+y^\ast)\\
&=\int_{Y_n^\ast\otimes X_t}
    \psi(\langle\langle y,z \rangle\rangle) [a,1]\cdot\varphi\begin{pmatrix}z\\x\\y^\ast\end{pmatrix}\,dz\\
&=\int_{Y_n^\ast\otimes X_t}
    \psi(\langle\langle y,z \rangle\rangle) \varphi\begin{pmatrix}a^{-1}z\\x\\a^\ast y^\ast\end{pmatrix}\,dz\\
&=|\det a|^n\int_{Y_n^\ast\otimes X_t}
    \psi(\langle\langle y,az \rangle\rangle) \varphi\begin{pmatrix}z\\x\\a^\ast y^\ast\end{pmatrix}\,dz\\
&=|\det a|^n\int_{Y_n^\ast\otimes X_t}
    \psi(\langle\langle a^\ast y,z \rangle\rangle) \varphi\begin{pmatrix}z\\x\\a^\ast y^\ast\end{pmatrix}\,dz\\
&=|\det a|^n \mathcal{F}(\varphi)(a^\ast y+x+a^\ast y^\ast).
\end{align*}

In particular, the group $N(X_t)_0$ simply acts as multiplication by $\psi(\langle\langle y,dy^\ast \rangle\rangle)$. It is easy to see that $\langle\langle y,dy^\ast \rangle\rangle=\frac{1}{2}\langle\langle w,dw \rangle\rangle$, where $w=y+y^\ast\in W_n\otimes X_t^\ast$. Now let
\[
    W_0=\{w\in W_n\otimes X_t^\ast: \langle\langle w,dw \rangle\rangle=0 \text{ for all } d=-^td\}.
\]
If we write $w=\sum_{i=1}^t w_i\otimes v_i^\ast$, we have $\langle\langle w,dw \rangle\rangle=\sum_{i\neq j}\langle w_i,w_j \rangle(v_i^\ast,dv_j^\ast)$. Since for each pair $i$ and $j$ with $i\neq j$ one can choose $d$ so that $\sum_{i\neq j}\langle w_i,w_j \rangle(v_i^\ast,dv_j^\ast)=2\langle w_i,w_j\rangle$, we have
$W_0=\{w\in W_n\otimes X_t^\ast: \langle w_i, w_j \rangle=0\}$. Or by identifying $W_n\otimes X_t^\ast$ with $Hom(X_t, W_n)$ in the obvious way, one can see that
\[
    W_0=\{\phi\in Hom(X_t, W_n): \phi^\ast(\langle-,-\rangle)=0\},
\]
where $\phi^\ast(\langle-,-\rangle)$ is the pullback of the symplectic form $\langle-,-\rangle$ on $W_n$ to $X_t$ via $\phi$. Then it is clear that the restriction map
\[
    S(W_n\otimes X_t^\ast)\otimes S(V_{m_0}\otimes Y_n^\ast)\rightarrow S(W_0)\otimes S(V_{m_0}\otimes Y_n^\ast)
\]
induces an isomorphism
\[
    {(\omega_{m,n})}_{N(X_t)_0}\cong S(W_0)\otimes S(V_{m_0}\otimes Y_n^\ast).
\]
We can decompose the space $W_0$ as
\[
    W_0=\coprod_{k=0}^{\min\{t,n\}} W_{0,k}
\]
where
\[
    W_{0,k}=\{\phi\in W_0:\dim {\rm Im}  \phi=k\}.
\]
Define $T^{(k)}\subseteq S(W_0)\otimes S(V_{m_0}\otimes Y_n^\ast)$ by
\[
    T^{(k)}=\{\varphi\in S(W_0):\varphi|_{W_{0,i}}=0\text{ for all } i<k\}\otimes S(V_{m_0}\otimes Y_n^\ast).
\]
Clearly $T^{(k+1)}\subseteq T^{(k)}$. One can see that each $T^{(k)}$ is invariant under $R_0(P(X_t)\times\GSp(W_n)^+)$, and moreover we have a short exact sequence of $R_0(P(X_t)\times\GSp(W_n)^+)$ modules
\[
    0\rightarrow T^{(k+1)}\rightarrow T^{(k)}\rightarrow S(W_{0,k})\otimes S(V_{m_0}\otimes Y_n^\ast)\rightarrow 0,
\]
where the map $T^{(k)}\rightarrow S(W_{0,k})\otimes S(V_{m_0}\otimes Y_n^\ast)$ is the obvious restriction map. Hence we have an $R_0(P(X_t)\times\GSp(W_n)^+)$ invariant filtration
\[
    \{0\}\subseteq T^{(\min\{t,n\})}\subseteq \cdots\subseteq T^{(1)}\subseteq T^{(0)}=S(W_0)\otimes S(V_{m_0}\otimes Y_n^\ast),
\]
where the successive quotient is given by
\[
    T^{k}:=T^{(k)}/T^{(k+1)}=S(W_{0,k})\otimes S(V_{m_0}\otimes Y_n^\ast).
\]
Then we have
\begin{Lem}\label{L:filtration}
There is an isomorphism of $P(X_t)\times\GSp(W_n)^+$ modules
\[
    (\Omega_{m,n})_{N(X_t)_0}=(ind_{R_0}^R\omega_{m,n})_{N(X_t)_0}\overset{\sim}{\longrightarrow}
    ind_{R_0(P(X_t)\times\GSp(W_n)^+)}^{P(X_t)\times\GSp(W_n)^+}\;S(W_0)\otimes S(V_{m_0}\otimes Y_n^\ast).
\]
Moreover the above filtration induces a $P(X_t)\times\GSp(W_n)^+$ invariant filtration
\[
    \{0\}\subseteq \tT^{(\min\{t,n\})}\subseteq \cdots\subseteq \tT^{(1)}\subseteq \tT^{(0)}=S(W_0)\otimes S(V_{m_0}\otimes Y_n^\ast),
\]
where
\[
    \tT^{(k)}=ind_{R_0(P(X_t)\times\GSp(W_n)^+)}^{P(X_t)\times\GSp(W_n)^+}\;T^{(k)}
\]
and the successive quotient is given by
\[
    \tT^{k}:=\tT^{(k)}/\tT^{(k+1)}\cong
    ind_{R_0(P(X_t)\times\GSp(W_n)^+)}^{P(X_t)\times\GSp(W_n)^+} S(W_{0,k})\otimes S(V_{m_0}\otimes Y_n^\ast).
\]
\end{Lem}
\begin{proof}
This follows because an induction is an exact functor.
\end{proof}

Now we will describe the representation $\tT^{k}$ of $P(X_t)\times\GSp(W_n)^+$ in terms of a certain induced representation, and then compute ${\tT^{k}}_{N(X_t)}$, which gives the desired filtration of the Jacquet module of the Weil representation $\Omega_{m_0,n}$. For that purpose, we need to describe the representation $T^{k}=S(W_{0,k})\otimes S(V_{m_0}\otimes Y_n^\ast)$ of $R_0(P(X_t)\times\GSp(W_n)^+)$ in terms of an induced representation. For this, let us write
\[
    R_0(P(X_t)\times\GSp(W_n)^+)=R_0(M(X_t)\times\GSp(W_n)^+)\ltimes N(X_t)
\]
and fix $w_0\in W_{0,k}$ given by
\[
    w_0=e_1\otimes v_{t-k+1}^\ast+e_2\otimes v_{t-k+2}^\ast+\cdots +e_k\otimes v_t^\ast.
\]
Let
\[
    H=\{([a,h],g)\in R_0(M(X_t)\times\GSp(W_n)^+):(g_1^{-1}\otimes a^\ast)w_0=w_0\}.
\]
Recalling
\[
    Y_k=\Span\{e_1,\dots,e_k\},
\]
we have
\[
    H\subseteq R_0(M(X_t)\times Q(Y_k)^+)
\]
where $Q(Y_k)^+\subseteq \GSp(W_n)^+$ is the maximal parabolic that preserves the flag $0\subseteq Y_k \subseteq W_n$. Recall we denote each element $q\in Q(Y_k)^+$ by
\[
    q=[b,g]n, \quad [b,g]\in M(Y_k)\cong \GL(Y_k)\times \GSp(W_{n-k})^+
\]
where $b\in\GL(Y_k)$, $g\in\GSp(W_{n-k})^+$ and $n$ is in the unipotent radical of $Q(Y_k)$. So each element in $H$ can be denoted by $([a,h], [b,g]n)$ where $[a,h]\in M(X_t)$ and $[b,g]n\in Q(Y_k)$. Then we define a representation
\[
    (\tau^k, S(V_{m_0}\otimes Y_n^\ast))
\]
of $H\ltimes N(X_t)$ on the space $S(V_{m_0}\otimes Y_n^\ast)$ as follows: For $([a,h], q)\in H$, we define
\[
    \tau^k([a,h],q)=\xi(\det a)|\lambda_{V}(h)|^{-\frac{1}{2}tn}\omega_{{m_0}, n}(h,q),
\]
where $\omega_{{m_0}, n}$ is the Weil representation of the pair $(\GO(V_{m_0}), \GSp(W_n)^+)$ and
\[
    \xi(\det a)=|\det a|^n.
\]
Also for $n(c,d)\in N(X_t)$,
\[
    \tau^k(n(c,d))=\rho_0(-c^\ast w_0)
\]
where $\rho_0$ is the action of the Heisenberg group $H(V_{m_0}\otimes W_n)$ on $S(V_{m_0}\otimes Y_n^\ast)$. Note that $-c^\ast w_0\in V_{m_0}\otimes W_n$ and hence the action of $\rho_0(-c^\ast w_0)$ on $S(V_{m_0}\otimes Y_n^\ast)$ makes sense. Then we have

\begin{Lem}
There is an isomorphism
\[
    T^k\cong ind_{H\ltimes N(X_t)}^{R_0(M(X_t)\times\GSp(W_n)^+)\ltimes N(X_t)}\tau^k
\]
of $R_0(M(X_t)\times\GSp(W_n)^+)\ltimes N(X_t)$-modules.
\end{Lem}
\begin{proof}
The proof is almost identical to Lemma 5.2 of \cite{K}.
\end{proof}

We would like to compute ${T^k}_{N(X_t)}$. But as in \cite[Lemma 5.3]{K}, we have
\[
     {T^k}_{N(X_t)}\cong (ind_{H\ltimes N(X_t)}^{R_0(M(X_t)\times\GSp(W_n)^+)\ltimes N(X_t)}\tau^k)_{N(X_t)}
     \cong ind_{H}^{R_0(M(X_t)\times\GSp(W_n)^+)}({\tau^k}_{N(X_t)}).
\]
Hence we first compute ${\tau^k}_{N(X_t)}$. For this purpose, let us write
\[
    Y_n^\ast=Y_k^\ast+{Y'_{n-k}}^\ast
\]
where $Y_k^\ast=\Span\{e_1^\ast,\dots,e_k^\ast\}$ and ${Y'_{n-k}}^\ast=\Span\{e_{k+1}^\ast,\dots,e_n^\ast\}$, and consider the polarization
\[
    V_{m_0}\otimes W_n=(V_{m_0}\otimes Y_n^\ast)+(V_{m_0}\otimes Y_n).
\]
Then we can write the Schwartz space as
\[
    S(V_{m_0}\otimes Y_n^\ast)=S((V_{m_0}\otimes Y_k^\ast)+(V_{m_0}\otimes {Y'_{n-k}}^\ast))
    \cong S(V_{m_0}\otimes Y_k^\ast)\otimes S(V_{m_0}\otimes {Y'_{n-k}}^\ast),
\]
where $S(V_{m_0}\otimes {Y'_{n-k}}^\ast)$ provides a model of the Weil representation $\omega_{{m_0},n-k}$ for the pair $(\GO(V_{m_0}), \GSp(W_{n-k})^+)$. Note that for $\varphi_1(x_1)\otimes\varphi_2(x_2)\in S(V_{m_0}\otimes Y_k^\ast)\otimes S(V_{m_0}\otimes {Y'_{n-k}}^\ast)$, each $([1,h],[b,g])\in R_0(M(X_t)\times M(Y_k))$ acts as
\begin{align*}
    &\omega_{{m_0},n}([1,h],[b,g])\cdot \varphi_1(x_1)\otimes\varphi_2(x_2)\\
    &\qquad\qquad=\xi'(\det(\lambda_{W}(g)^{-1}b))|\lambda_V(h)|^{-\frac{1}{4}m_0k}
    \varphi_1((h^{-1}\otimes\lambda_{W}(g)^{-1}b^\ast) x_1)\otimes \omega_{{m_0},n-k}(h,g)\varphi_2 (x_2),
\end{align*}
where
\[
    \xi'(x)=\chi_V(x)|x|^{\frac{m_0}{2}}\quad\text{for $x\in F^\times$}.
\]

Then let us define a representation
\[
    (\omega_0, S(V_{m_0}\otimes {Y'_{n-k}}^\ast))
\]
of $H$ by
\[
    \omega_0([a,h], [b,g]n)\varphi
    =\xi(\det a)\xi'(\det(\lambda_{W}(g)^{-1}b))\eta(\lambda_V(h))\omega_{{m_0},n-k}(h,g)\varphi
\]
for $([a,h], [b,g]n)\in H\subseteq R_0(M(X_t)\times Q(Y_k)^+)$ and $\varphi\in S(V_{m_0}\otimes {Y'_{n-k}}^\ast)$, where
\[
    \eta(\lambda_V(h))=|\lambda_V(h)|^{-\frac{1}{2}tn-\frac{1}{4}m_0k}.
\]
Then
\begin{Lem}
There is a natural isomorphism of $H$-modules
\[
    {\tau^k}_{N(X_t)}\cong\omega_0
\]
induced by the surjection
\[
    S((V_{m_0}\otimes Y_k^\ast)+(V_{m_0}\otimes {Y'_{n-k}}^\ast))\twoheadrightarrow S(V_{m_0}\otimes {Y'_{n-k}}^\ast)
\]
defined by $\varphi(x_1+x_2)\mapsto \varphi(0+x_2)$ where $x_1\in V_{m_0}\otimes Y_k^\ast$ and $x_2\in V_{m_0}\otimes {Y'_{n-k}}^\ast$.
\end{Lem}
\begin{proof}
Since $-c^\ast x_0\in V_{m_0}\otimes Y_k$, we have
\[
    \tau^k(-c^\ast x_0)\varphi(x_1+x_2)=\rho_0(-c^\ast x_0)\varphi(x_1+x_2)
    =\psi(\langle\langle x_1,-c^\ast x_0\rangle\rangle)\varphi(x_1+x_2).
\]
By looking at this action, one can easily see that the lemma follows.
\end{proof}

Thus we have proven
\[
    {T^k}_{N(X_t)}\cong ind_{H}^{R_0(M(X_t)\times\GSp(W_n)^+)}\omega_0.
\]
Now let
\[
    \mu_{tk}:=ind_{H}^{R_0(P(X_{t-k}, X_t)\times Q(Y_k)^+)}\omega_0
\]
where
\[
    X_{t-k}=\Span\{v_1,\dots,v_{t-k}\}
\]
and $P(X_{t-k}, X_t)$ is the parabolic subgroup of $M(X_t)$ that preserves the flag
\[
    \{0\}\subseteq X_{t-k}\subseteq X_t\subseteq V_m.
\]
Note that
\[
    P(X_{t-k}, X_t)\cong R(X_{t-k}, X_t)\times\GO(V_{m_0}),
\]
where $R(X_{t-k}, X_t)$ is the parabolic subgroup of $\GL(X_t)$ that preserves $0\subseteq X_{t-k}\subseteq X_t$, i.e.
\[
    R(X_{t-k}, X_t)=\{a=\begin{pmatrix}a_1&\ast\\&a_2\end{pmatrix}:a_1\in\GL(X_{t-k}), a_2\in\GL(X'_k)\},
\]
where
\[
    X'_{k}=\Span\{v_{t-k+1},\dots,v_n\}.
\]
Then we have
\[
    {T^k}_{N(X_t)}\cong ind_{R_0(M(X_{t-k}, X_t)\times Q(Y_k)^+)}^{R_0(M(X_t)\times\GSp(W_n)^+)}\mu_{tk}.
\]

In what follows, we realize $\mu_{tk}$ in a more concrete space. For this, let $\Isom(X'_k,Y_k)$ be the space of isomorphisms from $X'_k$ to $Y_k$ as vector spaces. Note $\GL(X'_k)\times\GL(Y_k)$ acts on this space in the obvious way. One can also see that
\[
    H=\{([\begin{pmatrix}a_1&\ast\\&a_2\end{pmatrix},h],[b,g]n)\in P(X_{t-k}, X_t)\times Q(Y_k):
    (\lambda_{W}(g)b^{-1}\otimes a_2^\ast)\cdot w_0=w_0\}.
\]
Notice that if $I\in\Isom(X'_k,Y_k)$ is such that $I(v_{t-k+i})=e_i$, then
\[
    (\lambda_{W}(g)b^{-1}\otimes a_2^\ast)\cdot w_0=w_0 \Longleftrightarrow b=I\lambda_W(g)a_2 I^{-1}.
\]
Also notice that the set
\[
    \{([\begin{pmatrix}1&\ast\\&a_2\end{pmatrix},1],[1,1])\in P(X_{t-k}, X_t)\times Q(Y_k): a_2\in\GL(X'_k)\}
\]
is a set of representatives of
\[
    H\backslash P(X_{t-k}, X_t)\times Q(Y_k).
\]
Then we have
\begin{Lem}\label{L:mu_tk}
There is an isomorphism
\[
    \mu_{tk}=ind_{H}^{R_0(P(X_{t-k}, X_t)\times Q(Y_k)^+)}\omega_0\cong S(\Isom(X'_k,Y_k))\otimes\omega_0
\]
of $R_0(P(X_{t-k}, X_t)\times Q(Y_k)^+)$-modules, where each $([\left(\begin{smallmatrix}a_1&\ast\\&a_2\end{smallmatrix}\right),h],[b,g]n)$ acts in the following way:
\begin{align*}
    &([\begin{pmatrix}a_1&\ast\\&a_2\end{pmatrix},h], [b,g]n)\cdot\varphi(A)\\
    &\qquad\qquad\qquad=\xi(\det a_2)\xi'(\det a_2)\xi(\det a_1)\eta(\lambda_V(h))
    \omega_{m_0,n-k}(h,g)\varphi(\lambda_{W}(g)b^{-1} A a_2),
\end{align*}
where $A\in\Isom(X'_k,Y_k)$, and we identify each element $\varphi\in S(\Isom(X'_k,Y_k))\otimes\omega_0$ with
\[
    \varphi:\Isom(X'_k,Y_k)\rightarrow \omega_0.
\]
\end{Lem}
\begin{proof}
Define
\[
    \alpha:ind_{H}^{R_0(P(X_{t-k}, X_t)\times Q(Y_k)^+)}\omega_0\rightarrow S(\Isom(X'_k,Y_k))\otimes\omega_0
\]
by
\[
    \alpha(F)(A)=F([1,1], [I A^{-1},1])\in \omega_0,
\]
where $F\in ind_{H}^{R_0(P(X_{t-k}, X_t)\times Q(Y_k)^+)}\omega_0$ and $([1,1], [I A^{-1},1])\in P(X_{t-k}, X_t)\times Q(Y_k)$, and also define
\[
    \beta:S(\Isom(X_k',Y_k))\otimes\omega_0\rightarrow ind_{H}^{R_0(P(X_{t-k}, X_t)\times Q(Y_k)^+)}\omega_0
\]
by
\[
    \beta(\varphi)([\begin{pmatrix}a_1&\ast\\&a_2\end{pmatrix},h], [b,g]n)
    =\omega_0([\begin{pmatrix}a_1&\ast\\&a_2\end{pmatrix},h], [I\lambda_{W}(g)a_2I^{-1}, g]n)\varphi(\lambda_{W}(g)b^{-1}Ia_2),
\]
where we identify each element $\varphi\in S(\Isom(X'_k,Y_k))\otimes\omega_0$ with $\varphi:\Isom(X'_k,Y_k)\rightarrow \omega_0$. Then it is straightforward to verify that $\alpha$ and $\beta$ are inverses to each other.

Also one can see that this map indeed intertwines the actions of $R_0(P(X_{t-k}, X_t)\times Q(Y_k)^+)$ by considering
\begin{align*}
&\alpha(\mu_{tk}([\begin{pmatrix}a_1&\ast\\&a_2\end{pmatrix},h], [b,g]n)F)(A)\\
&=\mu_{tk}([\begin{pmatrix}a_1&\ast\\&a_2\end{pmatrix},h], [b,g]n)F([1,1],[IA^{-1},1])\\
&=F([\begin{pmatrix}a_1&\ast\\&a_2\end{pmatrix},h], [IA^{-1}b,g]n)\\
&=\omega_0([\begin{pmatrix}a_1&\ast\\&a_2\end{pmatrix},h], [I\lambda_{W}(g)a_2I^{-1},g]n)
F([1,1],[I\lambda_{W}(g)^{-1}a_2^{-1}I^{-1}A^{-1}b,1])\\
&=\omega_0([\begin{pmatrix}a_1&\ast\\&a_2\end{pmatrix},h], [I\lambda_{W}(g)a_2I^{-1},g]n)
\alpha(F)(\lambda_{W}(g)b^{-1}Aa_2)\\
&=\xi(\det a_2)\xi'(\det(\lambda_{W}(g)^{-1}I\lambda_{W}(g)a_2I^{-1}))
\xi(\det a_1)\eta(\lambda_V(h))\omega_{m_0,n-k}(h,g)\alpha(F)(\lambda_{W}(g)b^{-1}Aa_2)\\
&=\xi(\det a_2)\xi'(\det a_2)\xi(\det a_1)\eta(\lambda_{V}(h))
\omega_{m_0,n-k}(h,g)\alpha(F)(\lambda_{W}(g)b^{-1}Aa_2).
\end{align*}
\end{proof}

This lemma gives
\[
    T^k_{N(X_t)}\cong ind_{R_0(P(X_{t-k}, X_t)\times Q(Y_k)^+)}^{R_0(M(X_t)\times\GSp(W_n)^+)}\mu_{tk}.
\]

Recall that we have been trying to compute $\tT^k_{N(X_t)}\cong (ind_{R_0(P(X_t)\times\GSp(W_n)^+)}^{P(X_t)\times\GSp(W_n)^+}\;T^{(k)})_{N(X_t)}$. But
note that
\begin{align*}
\tT^k_{N(X_t)}
&\cong (ind_{R_0(P(X_t)\times\GSp(W_n)^+)}^{P(X_t)\times\GSp(W_n)^+}\;T^{k})_{N(X_t)}\\
&\cong (ind_{R_0(M(X_t)\times\GSp(W_n)^+)\ltimes N(X_t)}^{(M(X_t)\times\GSp(W_n)^+)\ltimes N(X_t)}\;T^{k})_{N(X_t)}\\
&\cong ind_{R_0(M(X_t)\times\GSp(W_n)^+)}^{M(X_t)\times\GSp(W_n)^+}\;(T^{k}_{N(X_t)}),
\end{align*}
where the last isomorphism can be proven in the same way as \cite[Lemma 5.3]{K}. Therefore by inducing in stages, we obtain
\begin{align*}
    \tT^k_{N(X_t)}&\cong ind_{R_0(P(X_{t-k}, X_t)\times Q(Y_k)^+)}^{M(X_t)\times\GSp(W_n)^+}\mu_{tk}\\
    &\cong ind_{P(X_{t-k},X_t)\times Q(Y_k)^+}^{M(X_t)\times\GSp(W_n)^+}
    \left(ind_{R_0(P(X_{t-k}, X_t)\times Q(Y_k)^+)}^{P(X_{t-k},X_t)\times Q(Y_k)^+}\mu_{tk}\right).
\end{align*}

Now let $(\sigma_{tk},S(\Isom(X'_k,Y_k)))$ be the representation of $R_0(P(X_{t-k}, X_t)\times Q(Y_k)^+)$ defined by
\[
    \sigma_{tk}(([\begin{pmatrix}a_1&\ast\\&a_2\end{pmatrix},h], [b,g]n))\varphi(A)
    =\xi(\det a_2)\xi'(\det a_2)\xi(\det a_1)\eta(\lambda_{V}(h))\varphi(\lambda_{W}(g)b^{-1} A a_2)
\]
so that
\[
    \mu_{tk}=\sigma_{tk}\otimes\omega_{m_0, n-k}.
\]

\begin{Rmk}
At this point, one can derive the isometry version of the theorem simply by restricting $\mu_{tk}$ to the corresponding isometry group, and writing down the induction in the normalized form.
\end{Rmk}

We need to extend $\sigma_{tk}$ to a representation of $P(X_{t-k},X_t)\times Q(Y_k)^+$. Namely define the representation
\[
    (\tilde{\sigma}_{tk},S(\Isom(X'_k,Y_k)))
\]
of $P(X_{t-k},X_t)\times Q(Y_k)^+$
by
\begin{align*}
    \tilde{\sigma}_{tk}(([a,h], [b,g]n))\varphi(A)
    &=\eta(\lambda_{V}(h))\sigma_{tk}([a,1],[\lambda_W(g)^{-1}b,1])\varphi(A)\\
    &=\xi(\det a_2)\xi'(\det a_2)\xi(\det a_1)\eta(\lambda_{V}(h))
    \varphi(\lambda_{W}(g)b^{-1} A a_2),
\end{align*}
where
\[
    a=\begin{pmatrix}a_1&\ast\\&a_2\end{pmatrix}.
\]
Then we have
\begin{Lem}
There is a $P(X_{t-k},X_t)\times Q(Y_k)^+$-isomorphism
\[
    ind_{R_0(P(X_{t-k}, X_t)\times Q(Y_k)^+)}^{P(X_{t-k},X_t)\times Q(Y_k)^+}\sigma_{tk}\otimes\omega_{m_0,n-k}\cong
    \tilde{\sigma}_{tk}\otimes\Omega_{m_0,n-k},
\]
where $\Omega_{m_0,n-k}$ is the induced Weil representation of the pair $(\GO(V_{m_0}),\GSp(W_{n-k}))$.
\end{Lem}
\begin{proof}
Define
\[
    \alpha:ind_{R_0(P(X_{t-k}, X_t)\times Q(Y_k)^+)}^{P(X_{t-k},X_t)\times Q(Y_k)^+}\sigma_{tk}\otimes\omega_{m_0,n-k}
    \rightarrow\tilde{\sigma}_{tk}\otimes\Omega_{m_0,n-k}
\]
by
\[
    \alpha(F)(A)(h',g')=\eta(\lambda_{V}(h'))^{-1}F([1,h'],[\lambda_{W}(g'),g'])(A)
\]
for $A\in\Isom(X'_k,Y_k)$ and $(h',g')\in\GO(V_{m_0})\times\GSp(W_{n-k})$. One can verify that indeed $\alpha(F)(A)\in\Omega_{m_0,n-k}$. Here note that we identify an element $\alpha(F)\in \Isom(X'_k,Y_k)\otimes\Omega_{m_0,n-k}$ with a map
\[
    \alpha(F):S(\Isom(X'_k,Y_k))\rightarrow\Omega_{m_0,n-k}
\]
and so $\alpha(F)(A)(h',g')\in\omega_{m_0,n-k}$. Also define
\[
    \beta:\tilde{\sigma}_{tk}\otimes\Omega_{m_0,n-k}\rightarrow
    ind_{R_0(P(X_{t-k}, X_t)\times Q(Y_k)^+)}^{P(X_{t-k},X_t)\times Q(Y_k)^+}\sigma_{tk}\otimes\omega_{m_0,n-k}
\]
by
\[
    \beta(\Phi)([a,h],[b,g]n)(A)
    =\eta(\lambda_{V}(h))\sigma_{tk}([a,1],[\lambda_{W}(g)^{-1}b,1])\Phi(A)(h,g),
\]
for $A\in\Isom(X'_k,Y_k)$ and $([a,h],[b,g])\in P(X_{t-k},X_t)\times Q(Y_k)^+$. One can verify that $\beta(\Phi)$ is indeed in the induced space. Then by direction computation, one can see that $\alpha$ and $\beta$ are inverses to each other.

To see that $\alpha$ is intertwining, consider
\begin{align*}
&\alpha(([a,h],[b,g]n)\cdot F)(A)(h',g')\\
&=\eta(\lambda_{V}(h'))^{-1}(([a,h],[b,g]n)\cdot F)([1,h'],[\lambda_{W}(g'),g'])(A)\\
&=\eta(\lambda_{V}(h'))^{-1}F([a,h'h],[\lambda_W(g')b,g'g]n)(A)\\
&=\eta(\lambda_{V}(h'))^{-1}\sigma_{tk}([a,1],[\lambda_{W}(g)^{-1}b,1])
    \eta(\lambda_{V}(h'h))\eta(\lambda_{V}(h'h))^{-1}F([1,h'h],[\lambda_{W}(g'g),g'g])(A)\\
&=\eta(\lambda_{V}(h))\sigma_{tk}([a,1],[\lambda_{W}(g)^{-1}b,1])\alpha(F)(A)(h'h,g'g)\\
&=\tilde{\sigma}_{tk}([a,h],[b,g]n)\alpha(F)(A)(h'h,g'g)\\
&=\tilde{\sigma}_{tk}([a,h],[b,g]n)\Omega_{m_0,n-k}(h,g)\alpha(F)(A)(h',g').
\end{align*}

\end{proof}
Note that the induction $ind$ has not been normalized. To express it in the normalized way, recall the modular characters for the parabolic subgroups involved:
\begin{align*}
\delta_{R(X_{t-k},X_t)}([\begin{pmatrix}a_1&\ast\\&a_2\end{pmatrix}])&=|\det a_1|^k|\det a_2|^{-(t-k)}\\
\delta_{P(X_t)}([a,h]n)&=|\det a|^{m-t-1}|\lambda_{V_{m_0}}(h)|^{-\frac{mt}{2}+\frac{1}{2}t(t+1)}\\
\delta_{Q(Y_k)}([b,g]n)&=|\det b|^{2n-k+1}|\lambda_{W}(g)|^{-kn+\frac{1}{2}k(k-1)}.
\end{align*}
Hence we obtain each quotient of the filtration of the normalized Jacquet module as
\[
    J^k=\Ind_{P(X_{t-k},X_t)\times Q(Y_k)^+}^{M(X_t)\times\GSp(W_n)^+}
    (\tilde{\sigma}_{tk}\otimes\Omega_{m_0,n-k})(\delta_{R(X_{t-k},X_t)}\delta_{P(X_t)}\delta_{Q(Y_k)})^{-\frac{1}{2}}
\]
where the induction is also normalized. By writing down the characters involved explicitly, we see that
\[
    J^k=\Ind_{P(X_{t-k},X_t)\times Q(Y_k)^+}^{M(X_t)\times\GSp(W_n)^+}
    (S(\Isom(X'_k,Y_k))\otimes\Omega_{m_0,n-k}),
\]
where the group $P(X_{t-k},X_t)\times Q(Y_k)^+$ acts on $S(\Isom(X'_k,Y_k))$ as follows: Let $\varphi(A)\in S(\Isom(X'_k,Y_k))$. Then each element $[\begin{pmatrix}a_1&\ast\\&a_2\end{pmatrix},h]\in P(X_{t-k},X_t)$ acts as
\[
    [\begin{pmatrix}a_1&\ast\\&a_2\end{pmatrix},h]\cdot\varphi(A)
    =\chi_V(\det a_2)|\lambda_{V}(h)|^{e_0}|\det a_1|^{e_1}|\det a_2|^{e_2}\varphi(A a_2)
\]
where
\begin{align*}
    e_0&=-\frac{1}{4}(m-2t)k-\frac{1}{2}tn+\frac{1}{4}mt-\frac{1}{4}t(t+1)\\
    e_1&=n-\frac{1}{2}m+\frac{1}{2}t-\frac{1}{2}(k-1)\\
    e_2&=n-\frac{1}{2}(k-1),
\end{align*}
and each element $[b,g]n\in Q(Y_k)$ acts as
\[
    [b,g]n\cdot\varphi(A)
    =|\lambda_{W}(g)|^{f'_0}|\det b|^{f_1}\varphi(\lambda_{W}(g)b^{-1}A),
\]
where
\begin{align*}
    f'_0&=\frac{1}{2}kn-\frac{1}{4}k(k-1)\\
    f_1&=-e_2=-n+\frac{1}{2}(k-1).
\end{align*}

This is essentially the similitude analogue of the formula obtained by Kudla (\cite{K}). However one can simplify it further by ``absorbing away" the characters $|\det a_2|^{e_2}$ and $|\det b|^{f_1}$ in the regular representation realized in $S(\Isom(X'_k,Y_k))$ by using the following lemma.

\begin{Lem}
Let $\chi$ be a character and $\sigma$ the representation of the group of elements of the form $([\begin{pmatrix}1&\ast\\&a_2\end{pmatrix},1], [b,g])\in P(X_{t-k},X_t)\times Q(Y_k)$ realized on the space $S(\Isom(X'_k,Y_k))$ defined by
\[
    \sigma([\begin{pmatrix}1&\ast\\&a_2\end{pmatrix},1], [b,g])\varphi(A)=\chi(\det a_2)\chi(\det b)^{-1}\varphi(\lambda_{W}(g)b^{-1}Aa_2).
\]
Then $\sigma$ is equivalent to the representation $\sigma'$ realized on the same space $S(\Isom(X'_k,Y_k))$ defined by
\[
    \sigma'([\begin{pmatrix}1&\ast\\&a_2\end{pmatrix},1], [b,g])\varphi(A)=\chi(\lambda_{W}(g))^{-k}\varphi(\lambda_{W}(g)b^{-1}Aa_2).
\]
\end{Lem}
\begin{proof}
Define a map $S(\Isom(X'_k,Y_k))\rightarrow S(\Isom(X'_k,Y_k))$ by $\varphi\mapsto\tilde{\varphi}$, where $\tilde{\varphi}$ is defined as
\[
    \tilde{\varphi}(A)=\chi(\det A)\varphi(A).
\]
One can see that this map is an intertwining map from $\sigma$ to $\sigma'$.
\end{proof}

By taking $\chi=|-|^{e_2}$ in this lemma, one can see that the exponents $e_2$ and $f_1$ can be absorbed away, and the similitude factor $\lambda_W(g)$ acts by the character $|-|^{f_0}$ where $f_0=f'_0-ke_2$. The theorem follows.

\end{document}